\documentclass[10pt,a4paper]{article}
\usepackage{amsmath, amsthm, amsfonts}
\usepackage{hyperref}
\usepackage{graphicx}
\usepackage{color}
\graphicspath{{figures/}}
\usepackage[caption=false]{subfig}

\usepackage[left=3cm, right=3cm, top=3cm, bottom=3cm]{geometry}

\numberwithin{equation}{section}
\numberwithin{figure}{section}

\newtheorem{theorem}{Theorem}[section]
\newtheorem{lemma}[theorem]{Lemma}
\newtheorem{corollary}[theorem]{Corollary}
\newtheorem{proposition}[theorem]{Proposition}
\newtheorem{definition}[theorem]{Definition}

\newtheorem{remark}[theorem]{Remark}

\def \arccot {\mathrm{arccot}}
\def \argmax {\mathrm{argmax}}
\def \diag {\mathrm{diag}}
\def \Im {\mathrm{Im}}
\def \Re {\mathrm{Re}}
\def \Res {\mathrm{Res}}
\def \sgn {\mathrm{sgn}}
\def \tr {\mathrm{tr}}
\def \dif {\mathrm{d}}

\def \u {\mathbf{u}}
\def \v {\mathbf{v}}
\def \w {\mathbf{w}}
\def \A {\mathbf{A}}
\def \C {\mathbf{C}}
\def \D {\mathbf{D}}
\def \F {\mathbf{F}}
\def \I {\mathbf{I}}
\def \P {\mathbf{P}}

\def \hu {\hat{u}}

\def \fD {\mathfrak{D}}
\def \fL {\mathfrak{L}}
\def \BC {\mathfrak{B}}

\def \cL {\mathcal{L}}

\def \cO {\mathcal{O}}

\def \CC {\mathbb{C}}
\def \RR {\mathbb{R}}

\def \Lspace {\mathrm{L}}
\def \Hspace {\mathrm{H}}
\def \ACspace {\mathrm{AC}}

\def \reg {{\rm reg}}
\def \ss {{\rm ss}}
\def \ca {{\rm ca}}
\def \cc {{\rm cc}}
\def \nr {{\rm nr}}
\def \scalar {{\rm scal}}
\def \sub {{\rm sub}}

\def \alg {{\rm alg}}
\def \bp {{\rm bp}}

\begin{document}

\title{Reaction-subdiffusion systems and memory:\\ spectra, Turing instability and decay estimates}

\author{Jichen Yang\\ 
Sun Yat-sen University\\
School of Mathematics (Zhuhai)\\
519082 Zhuhai, China\\
{\tt yangjch36@mail.sysu.edu.cn}
\and 
Jens D.M. Rademacher\\
University of Bremen\\
Faculty 3 -- Mathematics\\
28359 Bremen, Germany\\
{\tt jdmr@uni-bremen.de}
}

\maketitle

\begin{abstract}
{The modelling of linear and nonlinear reaction-subdiffusion processes is more subtle than normal diffusion and causes different phenomena. The resulting equations feature a spatial Laplacian with a temporal memory term through a time fractional derivative. It is known that the precise form depends on the interaction of dispersal and reaction, and leads to qualitative differences. We refine these results by defining generalised spectra through dispersion relations, which allows us to examine the onset of instability and in particular inspect Turing type instabilities. These results are numerically illustrated. Moreover, we prove expansions that imply for one class of subdiffusion reaction equations algebraic decay for stable spectrum, whereas for another class this is exponential.}


\medskip
\noindent {\bf 2010 Math Subject Classification:} 34A08, 34D20, 35B35, 35R11 
\end{abstract}

\tableofcontents

\section{Introduction}

Anomalous diffusion arises from a transport process with a nonlinear temporal mean squared displacement (MSD) of  particles. Specifically, we say that a transport process exhibits anomalous diffusion if the MSD scales as a nonlinear power-law in time $\langle x^2(t)\rangle\sim t^\gamma$, with anomalous exponent $\gamma\neq1$, whereas the normal diffusion process based on Brownian motion has $\langle x^2(t)\rangle\sim t$. Such an anomalous diffusion process is called subdiffusion if $0<\gamma<1$, and superdiffusion if $\gamma>1$, cf.~ \cite{Metzler2000}. Anomalous diffusion can be described by continuous-time random walks (CTRW), a generalisation of the random walk, in which the waiting times until the next displacement are continuous random variables. In this paper we are concerned with subdiffusion, for which the probability density function (PDF) is non-exponential, which gives a non-Markovian process with memory.  Specifically, such PDF has asymptotic behaviour $w(t)\sim \gamma \tau^\gamma/(\Gamma(1-\gamma)t^{1+\gamma})$, $0<\gamma<1$ for $t\gg 1$~\cite{Mendez2014}. Hence, the particles have relatively high probability to remain at certain positions for very long time so that subdiffusion is, roughly speaking, slower than normal diffusion. The role of the heat equation in this context is taken by the \emph{subdiffusion equation}, which has the form
\begin{equation*}
	\partial_t u =  \fD_{0,t}^{1-\gamma} d\partial_x^2 u,\quad u(x,t)\in\RR, \; x\in\RR,\; t>0,
\end{equation*}
with time-fractional Riemann-Liouville derivative $\fD_{0,t}^{1-\gamma}$, cf.\ Definition \ref{d:RLderivative}, a non-local convolution operator in time that entails the memory, and diffusion coefficient $d$.

\medskip
Due to the non-local term in time, equations with subdiffusion are not dynamical system on the phase space of the natural initial condition $u(x,0)$. In particular, solutions do not form a cocycle since any restart at some $t>0$ requires to prescribe an initial condition on the preceding temporal interval $u(x,s)$, $0\leq s\leq t$. One may interpret this as a variable delay, which is initially zero and extends indefinitely. Indeed, solutions can cross the initial state, need not remain positive or satisfy a maximum principle. The fractional derivative operator also depends explicitly on time so stability of a steady state from a linearisation is not readily determined by spectrum of a time-independent linear operator.

\medskip
In this paper we study stability and instability properties for two classes of subdiffusion-reaction equations and systems. Indeed, the modelling of reactions in the presence of subdiffusion is a subtle task and leads to different models, which essentially come in two types. In the so-called subdiffusion-limited reactions the non-local operator acts on the diffusion terms as well as the reaction terms~\cite{Henry2006,Seki2003a,Seki2003,Yuste2004,Nec2008,Vergara2017}, e.g.,
\begin{align}\label{e:Zacher}
\partial_t u = \fD_{0,t}^{1-\gamma} [d \partial_x^2 u + f(u)]\quad \text{or}\quad \fD_{0,t}^{\gamma}(u-u_0) = \partial_x^2 u + f(u).
\end{align}
Here the reaction is ``slow'' even without subdiffusive transport $d=0$.

In the so-called activation-limited reactions the non-local operator acts on the diffusion terms only, i.e., the reactive process does not depend on the subdiffusive medium \cite{Henry2000,Henry2006,Fedotov2010,Nec2007,Nec2008}. In this paper, we focus on this type and study the following two different models.

\paragraph{Subdiffusion with extra source and sink} This model is derived by adding an extra source or sink term, cf., \eqref{e:Henry2000-balance} or \cite{Henry2000}, and the scalar case reads
\begin{equation}\label{e:Henry2000}
\partial_t u = d \fD_{0,t}^{1-\gamma} \partial_x^2 u + f(u),
\end{equation}
where $d>0$ is the diffusion coefficient and $f(u)$ is the nonlinear reaction term. For the fractional order $\gamma=1$, \eqref{e:Henry2000} becomes the classical reaction diffusion equation with the usual maximum principle. However, it has been shown in \cite{Henry2006} that \eqref{e:Henry2000} with negative linear reaction dynamics, e.g., $f(u)=-u$, possesses a Green's function with negative parts. Heuristically, the reason for the negative parts is that the sink removes particles which have not yet jumped from other positions due to the long waiting time distribution. If $u$ represents an absolute density we therefore obtain a physically unrealistic model, see also \cite{Nepomnyashchy2016}. However, in this paper we are concerned with $u$ modeling a \emph{density perturbation} from a saturated strictly positive state, and such a perturbation can be negative. Specifically, the linearisation of \eqref{e:Henry2000} about nonzero homogeneous steady state is a linear equation in which $f(u)$ is replaced by a linear term, and a corresponding model for multiple species is given by~\cite{Henry2002}
\begin{equation}\label{e:Henry2000-system}
\partial_t \u =\D \fD_{0,t}^{1-\gamma} \partial_x^2 \u + \F(\u),
\end{equation}
where $\u \in\RR ^N$ is the vector of the perturbation densities, $\D$ is the $N\times N$ diagonal matrix of diffusion coefficients with positive diagonal entries, and $\F(\u)$ is a nonlinear vector where $\F :\RR ^N\to\RR ^N$. 

\paragraph{Subdiffusion with linear creation and annihilation} In this case, the addition or removal of particles arises from the reaction during the waiting time, cf., \eqref{e:Henry2006-balance}. The scalar equation is given by~\cite{Henry2006}
\begin{equation}\label{e:Henry2006}
\partial_t u = d e^{at} \fD_{0,t}^{1-\gamma} \left( e^{-at} \partial_x^2 u \right) + a u,
\end{equation}
where $a\in\RR$ is the reaction rate. This model also coincides with the classical one at $\gamma=1$, and it preserves positivity of solutions since the amount of removed particles is less than the amount of existing particles, cf.\ Section~\ref{s:modelling}. The corresponding model for multiple species is given by~\cite{Langlands2008,Nepomnyashchy2016}, with a constant matrix $\A\in\RR^{N\times N}$ as
\begin{equation}\label{e:Langlands2008}
\partial_t \u = \D e^{\A t} \fD_{0,t}^{1-\gamma} \left( e^{-\A t} \partial_x^2 \u \right) + \A \u.
\end{equation}

\medskip
Inspired by the work of Henry and co-authors, we study the linear equations through Fourier-Laplace transform which leads to dispersion relations of the form $D(s,q^2)=0$ that relate the temporal mode through $s\in\CC$ with the spatial mode through $q\in \RR$. By analogy to evolution equations with normal diffusion, one might expect that the set of solutions determines the spectral stability, but the situation for reaction-subdiffusion systems (of the above types) is more subtle. 
Since fractional powers occur in the dispersion relation, one has to choose branch cuts, and the canonical choice of the negative real line has been used in \cite{Henry2002,Henry2005,Nec2007,Nec2008}. As expected, for some cases it has been shown that positive real parts imply exponential instability. Negative real parts, however, do not necessarily imply exponential decay. Indeed, it has been found in \cite{Henry2002} for $\gamma=1/2$ that solutions decay as a power law, and in \cite{Nec2008} for that solutions algebraically decay with an unspecific rational power.

In this paper, we refine and extend these results, as informally summarised next. 

\paragraph{Pseudo-spectrum} We consider non-canonical branch cuts and show that the choice strongly influences the existence of solutions to the dispersion relation that lie to the left of the (rightmost) branch point, which we therefore refer to as \emph{pseudo-spectrum}.

The reason for choosing non-canonical branch cuts is that it allows to locate and track otherwise invisible solutions to the dispersion relation. In particular we identify those solutions that relate to regular spectrum of classical reaction-diffusion equations for $\gamma=1$, and we prove convergence as $\gamma\to1$ (Theorems~\ref{t:ConvergenceModelA}, \ref{t:ConvergenceModelB}). 

\paragraph{Decay and growth} We show that, at least for rational $\gamma\in(0,1)$ and $N\leq 2$, strictly stable (pseudo-)spectrum in \eqref{e:Henry2000-system} implies that Fourier modes of solutions decay with the algebraic power law $t^{\gamma-2}$ (Theorem~\ref{t:ILT-ss}), while in \eqref{e:Langlands2008} it implies exponential decay whose rate, however, typically differs from that of normal diffusion and features an algebraic correction $t^{-\gamma}$ (Theorem \ref{t:ILT-ca}); the Fourier modes of both \eqref{e:Henry2000-system} and \eqref{e:Langlands2008} grow exponentially for unstable (pseudo-)spectrum. In fact, we provide explicit leading order expansions.
We also include a discussion of the scalar subdiffusion equation, in particular algebraic decay and positivity, which we found somewhat scattered in the literature, cf.\ Section~\ref{s:Subdiffusion}.

\medskip
We remark that different algebraic decay for strictly stable spectrum has been obtained in \cite[Theorem 5.1]{Vergara2017} for \eqref{e:Zacher}(2) with homogeneous Dirichlet boundary condition.  

\paragraph{Turing-type instability} Concerning the onset of instability, models \eqref{e:Henry2000-system} and  \eqref{e:Langlands2008} differ significantly from each other and from the case of normal diffusion $\gamma=1$. We focus on the case of two species systems, $N=2$, and when parameters are such that $\gamma=1$ admits a so-called Turing instability for a critical diffusion ratio. As already noticed in \cite{Henry2002,Henry2005,Nec2007,Nec2008}, in case \eqref{e:Henry2000-system}, if the spectrum for $\gamma=1$ is Turing unstable, then the spectrum for all $\gamma\in(0,1)$ is unstable. This means that, in terms of the diffusion ratio, the reaction-subdiffusion in \eqref{e:Henry2000-system} is always \emph{less stable} than normal diffusion. In particular, the threshold of Turing instability in this model is smaller than that for normal diffusion. As noticed in \cite{Nec2007,Nec2008}, considering large wavenumbers $q$ shows that the spectrum becomes unstable via \emph{infinite} wavenumber with oscillatory modes. In particular, there is no finite wavenumber selection at the onset of instability, which is a key feature of the normal Turing instability. Beyond the results in \cite{Nec2007,Nec2008}, we in particular include an analysis of pseudo-spectrum, which reveals the transition to instability. Specifically, we show that for any diffusion ratio less than the Turing threshold for normal diffusion, the pseudo-spectrum is strictly stable for all wavenumbers if $\gamma$ is close to 1 (Theorem~\ref{t:RealPartSpectrumInfinityWavenumber}).

Stable (pseudo-)spectrum of \eqref{e:Langlands2008} and its transition to instability, differs in character from that of \eqref{e:Henry2000-system}. Specifically, the (pseudo-)spectrum of \eqref{e:Langlands2008} is not close to the origin for large wavenumbers, which is similar to the regular spectrum and leads to the mentioned exponential decay, cf.\ Theorem \ref{t:ILT-ca}. We prove that stable (pseudo-)spectrum becomes unstable via finite wavenumber for $\gamma$ close to $1$, and we obtain the Turing instability threshold as a function of the anomalous exponent, and find a critical minimum anomalous exponent $\gamma_{\A}$ (Theorem~\ref{t:CriticalDiffusionRatioModel2}). In particular, and in contrast to \eqref{e:Henry2000-system}, the subdiffusive transport in \eqref{e:Langlands2008} is always \emph{more stable} than normal diffusion, and the Turing instability does not happen for $\gamma<\gamma_{\A}$.

\medskip
We illustrate these results with a number of numerical computations. 

\medskip
These insights may form the basis for further linear analysis, in particular estimates in $x$-space, but also nonlinear analysis in terms of estimates and possibly bifurcations. These are rather non-trivial due to the strongly different character of the linear dynamics and spectral properties compared to normal diffusion.

\medskip
This paper is organised as follows:
In Section~\ref{s:modelling}, we briefly illustrate the derivation of the two different reaction-subdiffusion equations. 
We introduce some background on fractional calculus, subdiffusion and Turing instability in Section~\ref{s:preliminaries}.
In Section~\ref{s:Model1}, we consider \eqref{e:Henry2000-system} and perform a detailed spectral analysis. We prove the local convergence theorem, obtain large wavenumber asymptotics and show the Turing instability results.
In Section~\ref{s:Model2}, we consider \eqref{e:Langlands2008} and give the analogous convergence theorem of spectra and the Turing instability analysis. Some technical proofs are deferred to appendices.

\section*{Acknowledgements}
This work was supported by the China Scholarship Council and the PhD degree completion stipend from University of Bremen. J.Y. is grateful for the hospitality and support from Faculty 3 -- Mathematics, University of Bremen as well as travel support through an Impulse Grants for Research Projects by University of Bremen. 

\section{Reaction-subdiffusion models}\label{s:modelling}

\paragraph{Subdiffusion with extra source and sink} In \cite{Henry2000}, a balance equation for the density of particles $u(x,t)$   at position $x$ and time $t$ is derived from the CTRW as
\begin{equation}\label{e:Henry2000-balance}
	u(x,t) = u(x,0) \Psi(t) + \int_{\RR } \int_0^t u(y,s) \phi(x-y) w(t-s) \dif s \dif y + \int_0^t h(x,s) \Psi(t-s) \dif s,
\end{equation}
where $\phi(x)$ and $w(t)$ are the PDFs of displacement and waiting time, respectively, with $\phi(x)$ the Gaussian $\phi(x)=(4\pi\sigma^2)^{-1/2}\exp\left(-x^2/(4\sigma^2)\right)$, and $w(t)$ a power law distribution with $w(t)\sim \gamma \tau^\gamma/(\Gamma(1-\gamma)t^{1+\gamma})$, $0<\gamma<1$ for $t\gg 1$~\cite[eq.~4.48]{Mendez2014}. The so-called cumulative function $\Psi(t)=\int_t^\infty w(s) \dif s$ represents the probability that particles have no jump during the time interval $(0,t)$. Lastly, the function $h(x,s)$ represents the added or removed particles at position $x$ and time $s<t$, and thus the source or sink. 

A physical interpretation of the summands on the right-hand side of \eqref{e:Henry2000-balance} as follows: The first summand represents the particles that stay at position $x$ up to the time $t$; the second summand represents the particles which moved from some position $y$ and some past time $s<t$ to position $x$ and stay until time $t$; the last summand represents the particles which were added at (removed from) position $x$ and time $s$, and remain at (do not return to) position $x$ until time $t$.

The Fourier-Laplace transform of \eqref{e:Henry2000-balance} gives an algebraic equation in terms of the Fourier transform of $\phi(x)$ and Laplace transform of $w(t)$. Considering the leading order large-scale, long-time limit and using the inverse Fourier-Laplace transform yield the reaction-subdiffusion equation \eqref{e:Henry2000}, where $f(u(x,t))$ replaces $h(x,t)$ as the (nonlinear) reaction term, cf.\ e.g.\ \cite{Metzler2000,Henry2000}. The corresponding model for multiple species \eqref{e:Henry2000-system} is obtained by extending to $N$-components system. 

\paragraph{Subdiffusion with linear creation and annihilation} The idea of the model derived in \cite{Henry2006} is that reactions occur at a constant per capita rate during the waiting time, i.e.\ $\partial_tu(x,t)=au$ which gives $u(x,t)=u(x,t_0)e^{a(t-t_0)}$. The resulting analogue to the above balance equation is 
\begin{equation}\label{e:Henry2006-balance}
	u(x,t) = u(x,0) e^{at} \Psi(t) + \int_{\RR } \int_0^t u(y,s) e^{a(t-s)} \phi(x-y) w(t-s) \dif s \dif y.
\end{equation}
Since all terms are positive for positive initial $u(x,0)$, the density $u$ is always positive; 
the amount of removed particles is always less than the existing ones. Compared to \eqref{e:Henry2000-balance}, the extra source or sink does not appear in the process in \eqref{e:Henry2006-balance}. Instead, the particles are created or annihilated intrinsically with exponential rate during the waiting time. The associated fractional differential equation reads \eqref{e:Henry2006}. In the case of multiple species, \eqref{e:Langlands2008} can be derived from the CTRW, cf.\ \cite{Langlands2008}.

Yet another model for which Turing-type instability has been discussed was proposed in \cite{Nec2007a} and reads
\begin{equation}\label{e:Nec2007a}
	\v (x,t) = \v (x,0) \delta(t) + \int_{\RR } \int_0^t \phi(x-y) w(t-s) e^{\A (t-s)} \v (y,s) \dif s \dif y,
\end{equation}
where the components of $\v (x,t)\in\RR ^N$ represent the number of particles of the corresponding species, which arrive at the position $x$ \emph{exactly} at time $t$, and $\v (x,0)\delta(t)$ represents the input of particles at initial time; here $\delta(t)$ is the Dirac delta distribution. The relation between \eqref{e:Langlands2008} and \eqref{e:Nec2007a} is given by
\begin{equation}\label{e:modelBdifference}
	\u (x,t) = \int_0^t \Psi(t-s) e^{\A (t-s)} \v (x,s) \dif s,
\end{equation}
cf.\ \cite{Langlands2008}, which can lead to different behaviour of the individual solutions for two models.

\section{Preliminaries}\label{s:preliminaries}

\subsection{Fractional calculus}\label{s:fractionalcalculus}

We define the Riemann-Liouville fractional integral and derivative for the subdiffusion range $\gamma\in(0,1)$.

\begin{definition}\label{d:fracintegral}
	Let $f(t) \in \Lspace^1(0,T)$ for any $T>0$. The integral
	\begin{equation}\label{e:FracIntegral}
		(\fD_{0,t}^{-\gamma}f)(t) := \int_0^t \frac{(t-s)^{\gamma-1}}{\Gamma(\gamma)}f(s) \dif s, \quad \gamma\in(0,1),
	\end{equation}
	is called \emph{fractional integral} of the order $\gamma$, where $\Gamma(\gamma)$ is Gamma function.
\end{definition}

This fractional integral is the (Laplace) convolution with kernel $k_{\gamma}(t):=t^{\gamma-1}/\Gamma(\gamma)$ via
\begin{equation}\label{e:Convolution}
	(k_\gamma*f)(t) := \int_0^t k_\gamma(t-s)f(s) \dif s,
\end{equation} 
i.e., $(\fD_{0,t}^{-\gamma}f)(t) = (k_{\gamma}*f)(t)$.
Notably, for $\gamma=1$ we get $(\fD_{0,t}^{-\gamma}f)(t) = \int_0^tf(s)\dif s$. 

\begin{definition}\label{d:RLderivative}
	For $f:[0,T]\to \RR$ the \emph{Riemann-Liouville fractional derivative} of order $1-\gamma$ is (formally) defined as
	\begin{equation}\label{e:RLDerivative}
		(\fD_{0,t}^{1-\gamma}f)(t) := \frac{\dif}{\dif t}(k_\gamma*f)(t) = \frac{\dif}{\dif t} \int_0^t \frac{(t-s)^{\gamma-1}}{\Gamma(\gamma)}f(s) \dif s, \quad \gamma\in(0,1).
	\end{equation}
\end{definition}

Notably, this fractional derivative is nonzero on constants, 
\begin{equation*}
	(\fD_{0,t}^{1-\gamma}1)(t) = k_{\gamma}(t),
\end{equation*}
which tends to zero as $t^{\gamma-1}$ for $t\to\infty$ and is unbounded for $t\to 0$. At $\gamma=1$, $\fD_{0,t}^{1-\gamma}$ is the identity operator, i.e., $(\fD_{0,t}^{1-\gamma}f)(t)=f(t)$, while for $\gamma=0$, $(\fD_{0,t}^{1-\gamma}f)(t)$ formally yields $(\dif/\dif t)\left(\delta*f\right)(t)=f'(t)$ with Dirac delta distribution $\delta(t)$ and $':=\dif/\dif t$.

\medskip
A simple sufficient condition for the existence of the Riemann-Liouville fractional derivative is as follows.
\begin{lemma}[\cite{Samko1993}, Lemma 2.2]\label{l:RLderivativeexist}
	Let $f(t) \in \ACspace([0,T])$, then $(\fD_{0,t}^{1-\gamma} f)(t)$ exists almost everywhere for $\gamma\in(0,1)$. Moreover $(\fD_{0,t}^{1-\gamma}f)(t) \in \Lspace^p(0,T)$, $1\leq p <1/(1-\gamma)$, and
	\begin{align*}
		(\fD_{0,t}^{1-\gamma} f)(t) = (k_\gamma*f')(t) + k_\gamma(t)f(0).
	\end{align*}
\end{lemma}
Here $\ACspace([0,T])$ is the space of functions $f$ which are absolutely continuous on $[0,T]$, i.e., $f(t) = c + \int_0^t g(s) \dif s$ for some $g \in \Lspace^1(0,T)$ and constant $c$. 

\medskip
The Laplace transform, denoted by $\fL$, will be used below and has the following property, see also \cite[Section 2.8]{Podlubny1998}, \cite[Section 2.2]{Kilbas2006}. 
\begin{lemma}\label{l:Laplace}
	Let $f:[0,\infty)\to\RR$ be exponentially bounded, i.e., $|f(t)| \leq A e^{ct}$ for some $A>0$ and $c\in\RR$. For $\Re(s)>\max\{0,c\}$ and $\gamma\in(0,1)$ the following hold.
	\begin{itemize}
	\item[(1)] If $f(t)\in \Lspace^1(0,T)$ for any $T>0$, then 
	$(\fL \fD_{0,t}^{-\gamma}f)(s) = s^{-\gamma}\left(\fL f\right)(s)$.
	
	\item[(2)] If $f(t)\in \ACspace([0,T])$ for any $T>0$, then $(\fD_{0,t}^{-\gamma} f)(0)=0$ and
	\begin{equation}\label{e:LaplaceTransformRLderivative}
	(\fL \fD_{0,t}^{1-\gamma}f)(s) = s^{1-\gamma}\left(\fL f\right)(s).
	\end{equation}
	\end{itemize}
\end{lemma}
\begin{proof}
This directly follows from \cite[Lemma 2.14 and Remark 2.8]{Kilbas2006}, where
 $\Re(s)>\max\{0,c\}$ stems from the calculation of $(\fL\fD_{0,t}^{-\gamma}f)(s)$ and guarantees convergence of the Laplace transform integral. 
In (2) we also used that $f(t)\in\ACspace([0,T])$ is bounded so that  $(\fD_{0,t}^{-\gamma}f)(0)=0$.
\end{proof}

\subsection{Subdiffusion equation}\label{s:Subdiffusion}

As mentioned, the analogue of the heat equation is the basic time-fractional diffusion equation is
\begin{equation}\label{e:Subdiffusion}
\partial_t u = d \fD_{0,t}^{1-\gamma} \partial_x^2 u, \quad x\in\RR,
\end{equation}
with diffusion coefficient $d>0$, which we refer to as the subdiffusion equation. Compared to the heat equation, it is more subtle to see that solutions \eqref{e:Subdiffusion} remain positive and to determine the decay properties. We found this scattered in the literature and next give a brief account of these results.

\medskip
For the Fourier-Laplace transform, the right-hand side of \eqref{e:Subdiffusion} defines the linear operator $\cL_{\sub} := \fD_{0,t}^{1-\gamma}\partial_x^2$ with $\cL_{\sub}: \ACspace([0,T];\Hspace^2(\RR))\to\Lspace^1(0,T;\Lspace^2(\RR))$. 

The Green's function of \eqref{e:Subdiffusion} has the Fourier transform~\cite[eq.~49]{Metzler2000}
\begin{align}\label{e:FourierSubSol}
	\widehat\Phi(q,t) = E_\gamma(-dq^2 t^\gamma),
\end{align}
where $E_\gamma(z) = \sum_{n=0}^\infty z^n/\Gamma(1+n\gamma)$ is the Mittag-Leffler function and $q$ is the wavenumber. The function \eqref{e:FourierSubSol} possesses the asymptotic behaviour~\cite[eq.~20]{Metzler2004}
\begin{align}\label{e:AsympML}
E_\gamma(-dq^2 t^{\gamma}) \sim 
\begin{cases}
\exp(-\frac{dq^2 t^\gamma}{\Gamma(1+\gamma)}), & t\ll (dq^2)^{1/\gamma}\\
(dq^2 t^\gamma \Gamma(1-\gamma))^{-1}, & t\gg (dq^2)^{1/\gamma}
\end{cases}
\end{align}
which shows the effect of memory for nonzero wavenumber: short time exponential decay and long time algebraic decay, here with power $-\gamma$. Note the above separation of decay depends on the wavenumber. In Fourier-Laplace space, the Green's function of \eqref{e:Subdiffusion} is given by
\begin{align}\label{e:GreenFL}
	(\fL\widehat{\Phi})(q,s) = (s + dq^2 s^{1-\gamma})^{-1},\quad \Re(s)>0,
\end{align}
and for the inverse Laplace transform the poles of \eqref{e:GreenFL} with nonzero $s$ contribute to the exponential growth/decay, while the trivial pole creates the algebraic decay, cf.\ Theorem \ref{t:ILT-ca} below. 

\medskip
We observe that the solutions to \eqref{e:Subdiffusion} have a self-similar scaling property, i.e.\ if $u(x,t)$ solves \eqref{e:Subdiffusion} then so does $u(\varepsilon x, \varepsilon^{2/\gamma} t)$ for $\varepsilon\in \RR$. The similarity variable $x/t^{\gamma/2}$ leads to a series expansion of the Green's function~\cite[eq.~46]{Metzler2000}\cite[eq.~4.23]{Mainardi2001} given by
\begin{equation}\label{e:GreensFunction}
\Phi(x,t) = \frac{1}{\sqrt{4d t^\gamma}}\sum_{n=0}^{\infty}\frac{(-1)^n}{n!\Gamma\left(1-\frac{\gamma}{2}-\frac{\gamma}{2}n\right)}\left(\frac{|x|}{\sqrt{d t^\gamma}}\right)^n, \quad t>0,
\end{equation}
In \cite{Mainardi2001} this is related to the Wright function, which turns out to be positive and algebraically decaying locally uniformly in $x$ for $t\gg 1$ with power $-\gamma/2$, and with power $-\gamma/(4-2\gamma)\in(-\gamma/2,0)$ for $|x|\gg\sqrt{dt^\gamma}$. We give a few details in Appendix \ref{s:WrightFunction}.

\medskip
As an aside we remark that the initial-boundary value problem \eqref{e:Subdiffusion} with homogeneous Dirichlet boundary condition has been studied in \cite{MMAK2019}. Here solutions decay pointwise algebraically with power $-\gamma$ for $t\gg1$, i.e., faster than on the unbounded domain.

\subsection{Turing instability}\label{s:Turinginstability}

Consider a classical reaction diffusion system given by
\begin{equation}\label{e:RegAISystem}
	\partial_t \u = \D \partial_x^2 \u + \F(\u), \quad x\in\RR ,
\end{equation}
where $\u \in\RR ^2$ is the vector of density of the species, $\F(\u)$ represents the reaction kinetics where $\F :\RR ^2\to\RR ^2$ and $\D =\diag(1,d)$ is the diagonal matrix of positive constant diffusion coefficients. Suppose that \eqref{e:RegAISystem} possesses a homogeneous steady state $\u _*=(u_{1*},u_{2*})^{\mathrm{T}}$, i.e., $\F (\u _*)=0$. The linearisation of \eqref{e:RegAISystem} about $\u _*$ is given by
\begin{equation}\label{e:RegLinearSystem}
	\partial_t \u = \D \partial_x^2 \u +\A \u , \quad \A = \left(\partial_\u \F(\u)\right)_{\u =\mathbf{u_*}}=:
	\begin{pmatrix}
	a_{1}	&	a_{2}\\
	a_{3}	&	a_{4}
	\end{pmatrix}.
\end{equation}
In an activator-inhibitor system $a_{1}>0$ and $a_{4}<0$. A Turing or diffusion-driven instability occurs if the homogeneous steady state of \eqref{e:RegAISystem} is strictly linearly stable in the absence of diffusion, but is linearly unstable in the presence of diffusion. Being a 2-by-2 matrix, strictly linear stability without diffusion means $\tr(\A ) = a_{1} + a_{4} <0$ and $\det(\A ) = a_{1}a_{4} - a_{2}a_{3} >0$.

The right-hand side of \eqref{e:RegLinearSystem} defines the linear operator $\cL :=\D \partial_x^2 +\A $, whose eigenvalue problem reads
\begin{equation*}
	\cL \u  = \D \partial_x^2 \u +\A \u  = s\u ,
\end{equation*}
where $s$ is the temporal eigenvalue. In Fourier space, the eigenvalue problem becomes
\begin{equation*}
	\cL \hat{\u} = -q^2\D \hat{\u} + \A \hat{\u} = s\hat{\u},
\end{equation*}
where $q$ is the wavenumber, and yields the dispersion relation
\begin{equation}\label{e:RegDR}
	D_{\reg}(s,q^2) := \det\left(s\I + q^2 \D - \A \right) = \left(s + q^2 - a_{1}\right)\left(s + dq^2 - a_{4}\right) -a_{2} a_{3} = 0.
\end{equation}
The solutions set $\Lambda_\reg := \{ s \in \CC : D_{\reg}(s,q^2) = 0\ \text{for a}\ q \in \RR \}$ is the $\Lspace^2$-\emph{spectrum} of $\cL$ with domain $(\Hspace^2(\RR))^2$, e.g., \cite{Sandstede2002}. In order to  distinguish this from spectra in the subdiffusion cases, we refer to it as the \emph{regular spectrum}. 

Concerning \eqref{e:RegAISystem}, the homogeneous steady state $\u _*$ is called strictly spectrally stable (unstable) if $\max(\Re(\Lambda_\reg))<0$ ($>0$). It is then also linearly and nonlinearly stable (unstable) for \eqref{e:RegLinearSystem} and \eqref{e:RegAISystem}, respectively \cite{Sandstede2002}. Furthermore, it is well known \cite[eq.~2.27]{Murray2003} that, if also $a_2a_3< 0$, then there exists a critical diffusion coefficient (also called Turing bifurcation point or Turing threshold) $d_c$ for which (i) $\sgn(\max(\Re(\Lambda_\reg))) = \sgn(d-d_c)$ and (ii) in case $d=d_c$ we have $D_{\reg}(s_c(q),q^2)=0$ for $q^2\approx q_c^2$ with real $s_c\approx -s_0 (q-q_c)^2$, $s_0>0$.

\section{Subdiffusion with source and sink}\label{s:Model1}

As the first reaction-subdiffusion model, we consider \eqref{e:Henry2000-system} with two components,

\begin{equation}\label{e:ModelA}
	\partial_t \u = \D \fD_{0,t}^{1-\gamma} \partial_x^2 \u + \F(\u),\quad \u\in\RR^2
\end{equation}
We will study linear stability properties of a homogeneous steady state $\u _*$ where $\F(\u_*)=0$. 

The (formal) linearisation of \eqref{e:ModelA} in $\u_*$ reads
\begin{equation}\label{e:FracLinearSystem}
	\partial_t \u = \D \fD_{0,t}^{1-\gamma} \partial_x^2 \u + \A \u,
\end{equation} its Fourier-transform with respect to $x\in\RR$ is
\begin{align}\label{e:ModelAFourier}
	\partial_t \hat\u = -q^2\D \fD_{0,t}^{1-\gamma} \hat\u + \A \hat\u,
\end{align}
and the Fourier-Laplace transform reads 
\begin{equation*}
	\left(s\I + s^{1-\gamma}q^2 \D - \A \right) \fL\hat\u = \hat{\u}_0, \quad \Re(s)>0, \ q \in \RR,
\end{equation*}
Analogous to \eqref{e:RegDR} we obtain the dispersion relation
\begin{align*}
	D_\ss(s,q^2) & := \det\left(s\I +s^{1-\gamma}q^2\D -\A \right)\nonumber\\
	& = \left(s+s^{1-\gamma} q^2- a_{1}\right)\left(s+s^{1-\gamma} dq^2- a_{4}\right)- a_{2}a_{3}=0, \quad s \in \Omega_0^+,
\end{align*}
where $\Omega_0^+ := \{s\in \CC : \arg(s)\in (-\pi/2,\pi/2) \}$. 

\paragraph{Branch and branch cut} Since $s^{1-\gamma}$ is a multivalued function in $\CC$, we need to choose a branch which preserves positive reals. For given $q\in\RR$ we choose $\theta_1(q)\in(0,\pi/2)$ such that on the branch cut 
\[
\BC_{0}^{\theta_1(q)} := \{ s\in\CC : \Im(s)/\Re(s) = \tan(\theta_1(q)),\, \Re(s)< 0\}
\]
there is no solution of the dispersion relation, i.e., $D_\ss(s,q^2) \neq 0$ for  $s\in\BC_{0}^{\theta_1(q)}$.

Since Theorems \ref{t:ILT-ss} for \eqref{e:ModelAFourier} and Theorem \ref{t:ILT-ca} for \eqref{e:Langlands2008} give decay and growth behaviour essentially independent of $\theta_1$, for simplicity we suppress the dependence of $\theta_1$ on $q$. The corresponding principal branch is defined by
	\[
	\Omega_0:=\{s \in \CC\setminus\{0\} : \arg (s) \in (-\pi+\theta_1,\pi+\theta_1),\,\theta_1 \in (0,\pi/2)\}.
	\]
Setting $s^\delta = z$, where $\delta:= 1-\gamma$, we obtain $z \in \Sigma_0:=\{z \in \CC\setminus\{0\} : \arg(z)\in ((-\pi+\theta_1) \delta, (\pi+\theta_1) \delta)\}$ if and only if $s\in\Omega_0$. Since $\Re(s)>0$, we restrict our dispersion relation to $s \in \Omega_0^+$ and $z \in \Sigma_0^+:=\{z \in \Sigma_0 : \arg(z)\in (-\pi \delta/2, \pi \delta/2)\}$. 

\begin{remark}\label{r:ExtendReason}
	In the computation of the ILT, we consider the integral along a vertical line in $\Omega_0^+$. It is natural to take the Bromwich contour (Fig.~\ref{f:ILTModelA}) and combine it with the residue theorem in order to calculate the ILT. However, this does not only depend on the roots of the dispersion relation in $\Omega_0^+$, but also in $\Omega_0^{0-} := \Omega_0 \setminus \Omega_0^+$.	Hence, in order to study the temporal behaviour of $\hat{\u}$, we extend the domain of $D_\ss$ and consider
	\begin{equation}\label{e:ModelADRExtend}
		D_\ss(s,q^2) = \left(s+s^{1-\gamma} q^2- a_{1}\right)\left(s+s^{1-\gamma} dq^2- a_{4}\right)- a_{2}a_{3}=0, \quad s \in \Omega_0.
	\end{equation}
\end{remark}

\begin{definition}
	We call the set of roots $\Lambda_\ss^+ := \{s\in\Omega_0^+: D_\ss(s,q^2)=0\ \text{for a}\ q\in\RR \}$ \emph{(subdiffusion) spectrum} of the linear operator $\cL_\ss :=\D \fD_{0,t}^{1-\gamma}\partial_x^2 +\A $,
	and the set of roots $\Lambda_\ss^{0-} := \{s\in\Omega_0^{0-}: D_\ss(s,q^2)=0\ \text{for a}\ q\in\RR \}$ \emph{(subdiffusion) pseudo-spectrum} of $\cL_\ss$.
\end{definition}

We denote $\Lambda_\ss := \Lambda_\ss^{0-} \cup \Lambda_\ss^+$, $\Omega_0^- := \Omega_0^{0-} \setminus i\RR$. As usual for spectral stability, we say the (pseudo-)spectrum of $\cL_\ss$ is (strictly) \emph{stable} (\emph{unstable}) if $\sup(\Re(\Lambda_\ss)) < 0$ ($>0$).

\begin{theorem}\label{t:ILT-ss}
	Let $\gamma\in(0,1)\cap\mathbb{Q}$ and $\lambda:=\sup(\Re(\Lambda_\ss))$. Let $\hat\u(q,t)$ be the solution to \eqref{e:ModelAFourier} with nonzero initial data $\hat\u_0$.
	\begin{itemize}
		\item[(1)] If $\lambda \geq 0$ and $S^+:=\{(s,q)\in \overline{\Omega_0^+}\times \RR : D_\ss(s,q^2)=0\ \text{and}\ \Re(s) \mbox{ maximal }\}\neq\emptyset$, then for any $(s_0,q_0)\in S^+$ we have $\hat{\u}(q_0,t) = C_{\exp} t^{k-1}e^{s_0 t} + o(t^{k-1}e^{\Re(s_0)t})$ with a nonzero $C_{\exp}\in \CC^2$ for almost all initial data and $k$ the multiplicity of $s_0$ as the root of $D_\ss(s,q_0^2)=0$.
		\item[(2)] If (i) $\lambda=0$ and $S^+=\emptyset$ or (ii) $\lambda < 0$ or (iii) $\Lambda_\ss=\emptyset$, then for any $q\in \RR\setminus\{0\}$ there exists a nonzero $C_{\alg}\in\CC^2$ such that $\hat{\u}(q,t) = C_{\alg} t^{\gamma-2} + o(t^{\gamma-2})$. Specifically, 
\begin{align*}
C_{\alg} := -\frac{q^2 \sin(\pi(1-\gamma)) \Gamma(2-\gamma)}{(a_1 a_4-a_2 a_3)^2\pi}
\begin{pmatrix}
-a_4^2 - a_2 a_3 d & a_2 a_4 + a_1 a_2 d\\
a_3 a_4 + a_1 a_3 d & -a_1^2 d - a_2 a_3
\end{pmatrix}
\hat{\u}_0.
\end{align*}
	\end{itemize}
\end{theorem}

\medskip
We defer the technical proof to Appendix \ref{s:InverseLaplaceZeroBranchPoint}.
Concerning case (1) in the theorem we remark that if $\lambda>0$, then the set $S^+$ is guaranteed to be non-empty.
Regarding the case $q=0$, note that \eqref{e:ModelAFourier} then reduces to $\partial_t \hat\u =  \A \hat\u$ whose solutions decay exponentially due to the assumption on the Turing instability.

\begin{remark} 
The determinant of the coefficient matrix in $C_{\alg}$ given by $(a_1a_4-a_2a_3)^2d$ is non-zero by assumption so that $C_{\alg}\neq0$ for non-trivial initial data $\hat\u_0$. Notably, the entries of $C_{\alg}$ grow as $q^2$, so the initial data $\u_0(x)$ should lie in $\Hspace^2$ in order to obtain solutions in $\Lspace^2$. This growth also corroborates the lack of smoothening of this model,  consistent with the observation in \cite{Henry2005}.
A more detailed decomposition into exponential terms and algebraically decaying terms, including formulae for $C_{\exp}, C_{\alg}$ is given in \eqref{e:ILTFormulaSimplePole} (\eqref{e:multicoeff-ss} in case of multiple roots). 
\end{remark}

In particular, Theorem \ref{t:ILT-ss} reveals the roots of \eqref{e:ModelADRExtend} determine the temporal behaviour of $\hat{\u}$ as exponentially growing with algebraic factor for unstable spectrum, and algebraically decaying for strictly stable pseudo-spectrum. 
Therefore, the main work of the present section is to analyse the (pseudo-)spectrum.

\begin{remark}\label{r:denominator}
Our approach does not apply to irrational $\gamma$ for which the decay remains an open problem to our knowledge.
We assume rational $\gamma$ as a technical simplification, but our result indicates a uniform behaviour with respect to $\gamma$ on the Fourier modes. For rational $\gamma$ the solution in Fourier-Laplace space possesses finitely many singularities and it is possible to explicitly compute the leading order of the inverse Laplace transform via residues. 
For irrational $\gamma$, an accumulation of solutions to the dispersion relation on certain curves occurs, which is an obstacle for a branch cut and contours that avoid singularities. 
\end{remark}

\medskip
Here we do not transfer the decay to physical $x$-space and related function spaces, since the dependence of the constants in the estimate on the wavenumber are convoluted. 

\begin{remark}
The ansatz $\u(x,t) = \v(t)e^{iqx}$ is a specific case of the Fourier transform and substitution into \eqref{e:FracLinearSystem} also gives the fractional ODE \eqref{e:ModelAFourier}. Hence, the initial condition in Theorem \ref{t:ILT-ss} can be seen as $\hat{\u}(q,0) = \hat{\v}_0e^{iqx}$, where $\hat{\v}_0\in\RR^2$. Such an ansatz can find spatially periodic solutions, and can be applied equally to the model \eqref{e:ModelB}. We do not discuss this further here.
\end{remark}

\subsection{Scalar case}\label{s:ModelA-scalar}

In order to illustrate further the fundamental difference between $\gamma=1$ (normal diffusion) and $\gamma\neq 1$ (subdiffusion), we consider the scalar case of \eqref{e:FracLinearSystem},
\begin{equation}\label{e:ModelA-scalar}
	\partial_t u = d \fD_{0,t}^{1-\gamma} \partial_x^2 u + a u, \quad u \in\RR ,\ a\in\RR .
\end{equation}
The dispersion relation is given by
\begin{equation}\label{e:ModelA-ScalarDR}
	d_{\ss}(s,q^2) = s + dq^2 s^{\delta} - a = 0, \quad s \in \Omega_0,\ q\in\RR,
\end{equation}
where $\delta=1-\gamma$. Clearly, $s=a$ is the unique solution of \eqref{e:ModelA-ScalarDR} for $q=0$. 

\begin{remark}\label{r:meaningless}
The branch point $s = 0$ is a solution to $d_{\ss} = 0$ if and only if $a = 0$, i.e., the case of the subdiffusion equation, cf.\ Section~\ref{s:Subdiffusion}. Theorem~\ref{t:ILT-ss} does not apply and the long term decay is $t^{-\gamma}$, which is as predicted by the more general Theorem \ref{t:ILT-ca} below.
\end{remark}

In the following, we give some characteristics of the (pseudo-)spectrum of \eqref{e:ModelA-scalar}.

\begin{lemma}\label{l:Scalarq0}
For any $\delta\in(0,1)$ a unique smooth curve of solutions $s=s(q)$ to \eqref{e:ModelA-ScalarDR} crosses $s=a$ at $q=0$. For $0<|q|\ll1$ we have $s(q)<a$ if either $a>0$ or $a<0$ as well as $\delta\in(0,1/2)$, while $s(q)>a$ if $a<0$ and $\delta\in(1/2,1)$.
\end{lemma}
\begin{proof}
Implicit differentiation of the left-hand side with respect to $q^2$ and $s$ at $s=a, q=0$ gives $da^\delta$ and $1$, respectively. The statement follows from the implicit function theorem and $a^\delta>0$ for $a>0$, and $\sgn(\Re(a^\delta))=\sgn(1/2-\delta)$ for $a<0$ with $\delta\in(0,1)$.
\end{proof}

\medskip
Lemma \ref{l:Scalarq0} shows the concavity of (pseudo-)spectrum for $|q|\ll 1$, and Fig.~\ref{f:Curvature} illustrate it numerically. However, the concavity changes for larger $|q|$, cf.\ Fig.~\ref{f:CurNegative}. In the following, we give the analysis and numerical computation of  (pseudo-)spectrum.
\begin{figure}[t]
\centering
\subfloat[$\delta = 2/5 < 1/2$]{\includegraphics[width=0.3\linewidth]{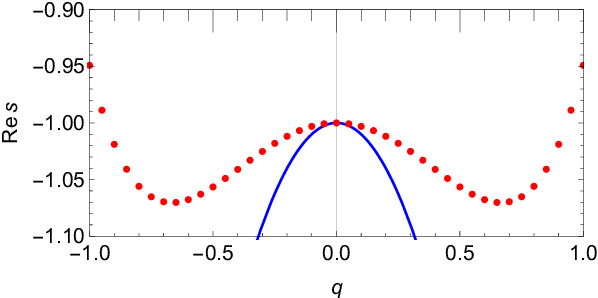}\label{f:CurNegative}}
\hfil
\subfloat[$\delta = 3/5 > 1/2$]{\includegraphics[width=0.3\linewidth]{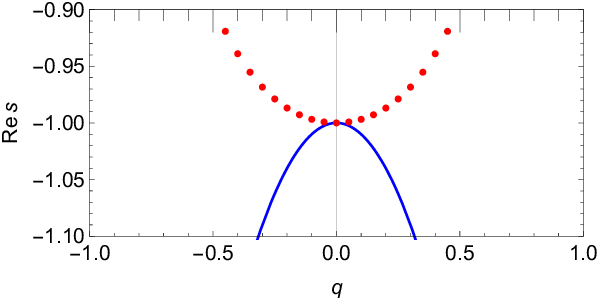}\label{f:CurPositive}}
\caption{Comparison of numerically computed regular spectra (blue solid, $\delta=0$) and pseudo-spectra (red dotted) of \eqref{e:FracLinearSystem} for $a<0$, $\theta_1=\pi/2$.}
\label{f:Curvature}
\end{figure}

\begin{lemma}\label{l:ScalarPositiveSol}
	If $a>0$, then for each $q\in\RR$ the solution $s=s(q)$ to \eqref{e:ModelA-ScalarDR} in $\Omega_0^+$ is unique, positive and satisfies $\lim_{|q|\to\infty} s(q) = 0$. 
\end{lemma}

\begin{proof}
	First, we show that the solution of \eqref{e:ModelA-ScalarDR} must be positive in $\Omega_0^+$. Set $s = r e^{i\theta}$, $r>0$ and $\theta\in(-\pi/2,\pi/2)$, then we can rewrite \eqref{e:ModelA-ScalarDR} as $r e^{i\theta} + d q^2 r^\delta e^{i\delta\theta} = a$. Since $a\in \RR$, the imaginary part of the left-hand side must be zero, i.e., $r \sin(\theta) + d q^2 r^\delta \sin(\delta\theta) = 0$, which is equivalent to $\theta = 0$ since $\delta\in(0,1)$. Hence we have $s>0$.
	
	\smallskip
	Clearly, $s=a$ is the unique solution for $q=0$. Implicit differentiation of the left-hand side of \eqref{e:ModelA-ScalarDR} with respect to $s$ gives $1+dq^2\delta s^{\delta-1}$ which is continuous and nonzero for any $q\in\RR$ and $s>0$. Therefore, the uniqueness for $|q|\ll 1$ can be extended to all $q\in\RR$.
	
	\smallskip
	Since $s>0$ we can rescale $q^2 = \kappa^2/s^\delta$ with $\kappa\in\RR$, cf.\ \cite[Section 4.1]{Nec2007}, which gives
	\begin{equation}\label{e:ModelA-ScalarDR-Rescale}
	s + d \kappa^2 -a =0, \quad s>0,\ \kappa\in\RR.
	\end{equation}
	As $a>0$, the solution of \eqref{e:ModelA-ScalarDR-Rescale} is a parabola in the $(\kappa,s)$-plane and $s\to 0$ for $\kappa^2 \to a/d$. Hence, according to the scaling $q^2\to \infty$ as $\kappa^2 \to a/d$. In contrast, $s\to 0$ for $q^2 \to \infty$ in $(q,s)$-plane. See the right column in Fig.~\ref{f:ModelA-Scalar}. 
\end{proof}
\begin{figure}[t]
\centering
\subfloat[$\delta = 3/4 \in (2/3,1)$, $a<0$]{\includegraphics[width=0.3\linewidth]{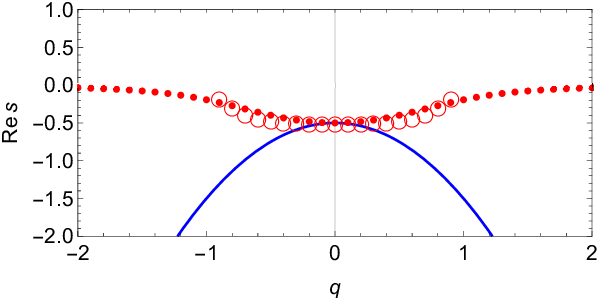}\label{f:ModelA-Scalar-Negative-largedelta}}
\hfil
\subfloat[$\delta = 3/4 \in (2/3,1)$, $a=0$]{\includegraphics[width=0.3\linewidth]{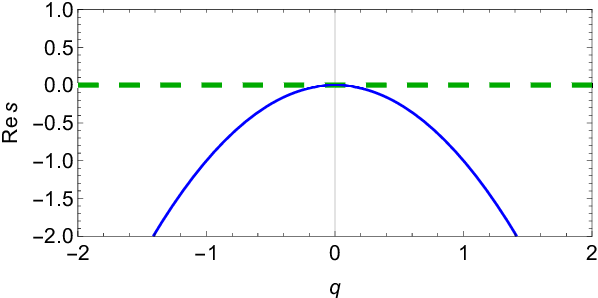}\label{f:ModelA-Scalar-Zero-largedelta}}
\hfil
\subfloat[$\delta = 3/4 \in (2/3,1)$, $a>0$]{\includegraphics[width=0.3\linewidth]{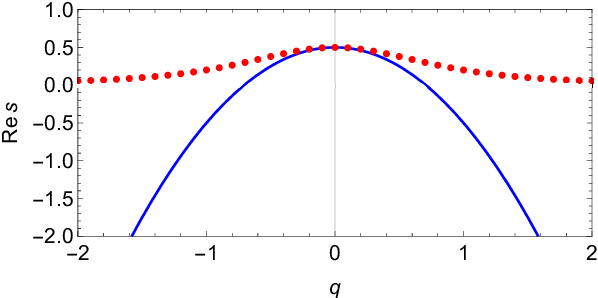}\label{f:ModelA-Scalar-Positive-largedelta}}
\hfil
\subfloat[$\delta = 1/4 \in (0,1/3)$, $a<0$]{\includegraphics[width=0.3\linewidth]{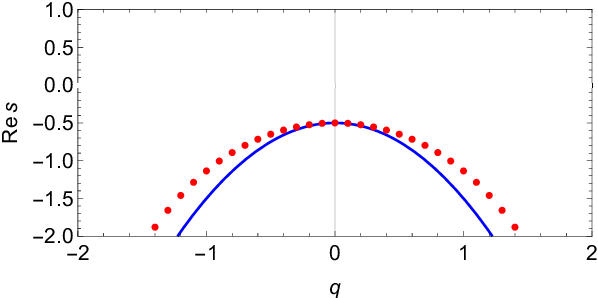}\label{f:ModelA-Scalar-Negative-smalldelta}}
\hfil
\subfloat[$\delta = 1/4 \in (0,1/3)$, $a=0$]{\includegraphics[width=0.3\linewidth]{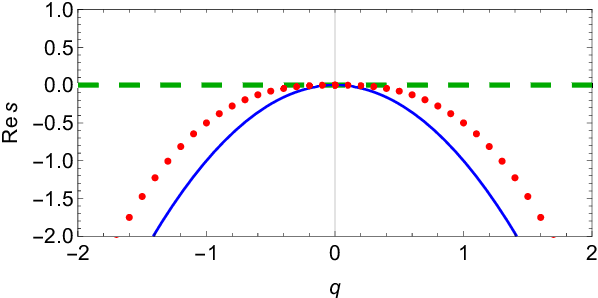}\label{f:ModelA-Scalar-Zero-smalldelta}}
\hfil
\subfloat[$\delta = 1/4 \in (0,1/3)$, $a>0$]{\includegraphics[width=0.3\linewidth]{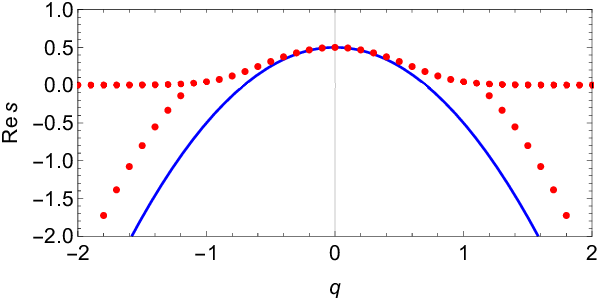}\label{f:ModelA-Scalar-Positive-smalldelta}}
\caption{Model \eqref{e:ModelA-scalar}. Blue solid lines: regular spectra; red dotted lines: subdiffusion (pseudo-)spectrum; green dashed lines: $s(q)=0\notin\Omega_0$. Here $\theta_1 = \pi/2$. In (b) the (pseudo-)spectrum is empty. In case $\delta\in[1/3,2/3]$: the (pseudo-)spectrum for $a=0$ and $a>0$ are qualitatively same as  (b) and (c), respectively, whereas the pseudo-spectrum for $a<0$ (hollow circle) is empty for large $|q|$.}
\label{f:ModelA-Scalar}
\end{figure}	

Solving \eqref{e:ModelA-ScalarDR} explicitly for general $a\in\RR$, $s\in\Omega_0$, seems not possible, but we can approximate solutions for $|q|\gg 1$ as follows. 

\begin{lemma}
	For any $\delta\in(0,1)$, there exists a $Q>0$, such that for $q>Q$, the solutions to \eqref{e:ModelA-ScalarDR} are  approximated by $s_{0}(q)$ and $s_{\infty}(q)$, 	where
	\begin{itemize}
	\item[(1)] $s_{0}(q) \in \Omega_0^+$ and $\lim_{|q|\to\infty} s_{0}(q) = 0$ for any $a > 0$,
	
	\item[(2)] $s_{0}(q) \in \Omega_0^-$ and $\lim_{|q|\to\infty} s_{0}(q) = 0$ for any $a < 0$ and $\delta \in (\pi/(\pi+\theta_1),1)$,		
	\item[(3)] ${s_{\infty}}(q) \in \Omega_0^-$ and $\lim_{|q|\to\infty} s_{\infty}(q) = -\infty$ for any $\delta \in (0, \theta_1/(\pi+\theta_1))$.
	\end{itemize}
\end{lemma}

\begin{proof}
	A straightforward computation gives the solutions to \eqref{e:ModelA-ScalarDR} for $|q| \gg 1$ as	
	\begin{align}
	s_0(q) & = \left(\frac{a}{dq^2}\right)^{1/\delta} + o\left(q^{-2/\delta}\right),\\
	s_{\infty}(q) & = \left(-dq^2\right)^{\frac{1}{1-\delta}} + \frac{a}{1-\delta} + o(1),
	\end{align}
	where $\lim_{|q|\to\infty}o(1) = 0$. See also the proof of Lemma \ref{l:AsymptoticSolution} for $s_0(q)$ and $s_\infty(q)$.
	
	\smallskip
	Concerning $s_0(q)$: For $a > 0$ the leading order term is positive real and $\lim_{|q|\to\infty} s_{0}(q) = 0$, which coincides with Lemma \ref{l:ScalarPositiveSol}. For $a < 0$ the argument of the leading order term is $\pi / \delta$ so $\arg(s_0) \in \Omega_0$ if $\pi/\delta \in (-\pi+\theta_1, \pi+\theta_1)$ which leads to $\delta > \pi/(\pi+\theta_1)$ (cf.\ Fig.~\ref{f:ModelA-Scalar-Negative-largedelta}). Since $\delta \in (0,1)$, we have $\pi/\delta > \pi$, which leads to $\Re(s_0) < 0$.
	
	\smallskip
	Concerning $s_{\infty}(q)$: We note that the leading order term is independent of $a$ and its argument is $\pi/(1-\delta)$. Then we have $\arg(s_\infty) \in \Omega_0$ if $\pi/(1-\delta) \in (-\pi+\theta_1, \pi+\theta_1)$ which implies $\delta < \theta_1/(\pi+\theta_1)$ (cf.\ bottom row in Fig.~\ref{f:ModelA-Scalar}). Since $\delta \in (0,1)$, we have $\pi/(1 - \delta) > \pi$, which leads to $\Re(s_\infty) < 0$.
\end{proof}

\begin{remark}
The choice of branch cut is relevant here: if $\theta_1 \to 0$, then $\pi/(\pi+\theta_1) \to 1$ and $\theta_1/(\pi+\theta_1) \to 0$, which leads to the disappearance of red dotted (pseudo)-spectrum in Fig.~\ref{f:ModelA-Scalar}. Nevertheless, it is instructive to choose $\theta_1 = \pi/2$ as this reveals all phenomena in the pseudo-spectrum, in particular regarding the relation to the regular spectrum. 

We note that the plotting of $s=0$ (green dashed) in the middle column of Fig.~\ref{f:ModelA-Scalar} only shows the transition of $s_0(q)$ from negative to positive but is not in the (pseudo-)spectrum. Moreover, for instance in Fig.~\ref{f:ModelA-Scalar-Zero-largedelta}, the (pseudo-)spectrum is empty. Yet, by Theorem \ref{t:ILT-ss} there is still algebraic decay, at least for rational $\delta$.
\end{remark}

\medskip
For constant initial $u(0,x)=u_0\in\RR$ \eqref{e:ModelA-scalar} is the ODE $\dot u = au$ so $a<0$ indeed yields exponential decay.
Fourier-transforming \eqref{e:ModelA-scalar} in $x$ gives
\begin{equation}\label{e:scalarFourier}
\partial_t \hat u = -dq^2 \fD_{0,t}^{1-\gamma} \hat u + a \hat u,
\end{equation}
and, for rational $\gamma=n/m$ (reduced fraction), Theorem \ref{t:ILT-ss} implies each Fourier mode decays algebraically as $t^{\gamma-2}$ for $a<0$.

\subsection{Convergence to regular spectrum}\label{s:Convergence}

We return to the system \eqref{e:FracLinearSystem} and study the convergence of subdiffusion (pseudo-)spectrum for $\gamma\to 1$, i.e., as the anomalous exponent  approaches normal diffusion $\gamma=1$. With $\delta=1-\gamma$ the dispersion relation \eqref{e:ModelADRExtend} can be written as
\begin{equation}\label{e:FracDR}
	D_\ss(s,q^2) = (s + s^\delta q^2 - a_{1})(s + s^\delta d q^2 - a_{4}) - a_{2} a_{3} = 0, \quad s \in \Omega_0,
\end{equation}
and we consider $\delta \to 0$. In preparation, we note that the difference between subdiffusion and regular dispersion relation is 
\begin{align*}
	E(s,q^2) = D_\ss(s,q^2) - D_{\reg}(s,q^2) = (s^\delta - 1) [(q^2 + d q^2) s + d q^4 s^\delta + d q^4 - a_{4} q^2 - a_{1} d q^2].
\end{align*}

\begin{lemma}\label{l:ConvergeLocallyUniformly}
	The subdiffusion (pseudo-)spectrum converges to the regular spectrum locally uniformly in $q \in \RR$ as $\gamma \to 1$.
\end{lemma}

\begin{proof}
	Denote $f(s):=D_{\reg}(s,q^2)$ and $g(s):=E(s,q^2)$, then $f(s)+g(s)=D_\ss(s,q^2)$. First, we claim that for fixed parameters and wavenumber $q$, $g(s) \to 0$ locally uniformly in $s \in \Omega_0$ as $\delta \to 0$. This follows from $s^\delta-1=e^{\delta\ln s}-1$ being holomorphic in $\Omega_0$ with $e^{\delta\ln s}-1 \to 0$ pointwise in $\Omega_0$ as $\delta\to0$.
	
	\smallskip
	Second, we discuss $f(s)$. For fixed parameters and wavenumber $q$, there are two regular eigenvalues $s_1,\, s_2 \in \CC$ and $f(s) \neq 0$ for $s \neq s_1,\, s_2$, cf.\ Fig.~\ref{f:Rouche1}. We choose two disjoint open balls $B_{r_1}(s_1)$ and $B_{r_2}(s_2)$, where $B_{r_*}(s_*) := \{s \in \Omega_0 : |s-s_*| < r_* \}$. Then we have $f(s) \neq 0$ for $s \in \partial B_{r_j}(s_j),\ j=1,2$. From the first step, we know that $g(s) \to 0$ as $\delta \to 0$, hence $|f(s)| > |g(s)|$ for $s \in \partial B_{r_j}(s_j),\ j=1,2$ as $\delta \to 0$. Since $f(s)$ and $g(s)$ are holomorphic in $\Omega_0$, Rouch\'{e}'s theorem implies $f+g$ also has two zeros $s_1(\delta)$ and $s_2(\delta)$ inside $B_{r_1}(s_1)$ and $B_{r_2}(s_2)$, respectively. Since we can choose $r_j$ arbitrarily small as $\delta \to 0$ we have $s_j(\delta) \to s_j$ locally uniformly in $q$, $j=1,2$.

	\smallskip
	However, the regular spectrum can contain $0$, whereas $g(s)$ is not holomorphic on $\BC_0^{\theta_1}$. In such a case, we take a neighbourhood $B_{\epsilon}(0)$ of the origin and choose a contour $C$ such that its interior covers a region near $0$, but excludes $B_{\epsilon}(0)$ and $\BC_0^{\theta_1}$, cf.\ Fig.~\ref{f:Rouche2}. Since there is no zero of $f$ inside $C$, Rouch\'{e}'s theorem implies that there is no zero of $f+g$ inside $C$ as $\delta \to 0$ either. Hence the zeros of $f+g$ have to be in $B_{\epsilon}(0)$ and also such zeros of $f+g$ will converge to $0$ as $\delta\to0$.
\end{proof}
\begin{figure}[t]
	\centering
	\subfloat[]{\includegraphics[width=0.35\linewidth]{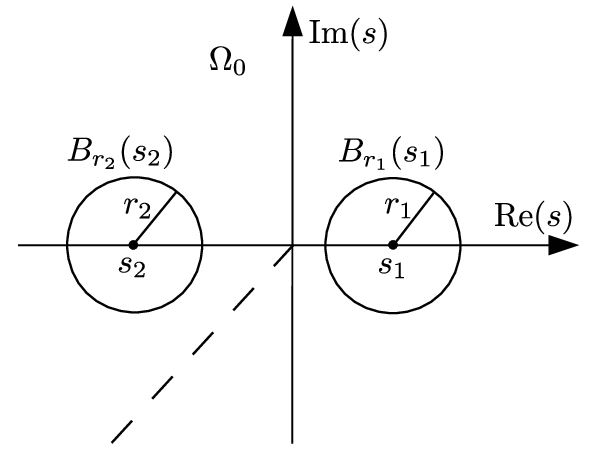}\label{f:Rouche1}}
	\hfil
	\subfloat[]{\includegraphics[width=0.35\linewidth]{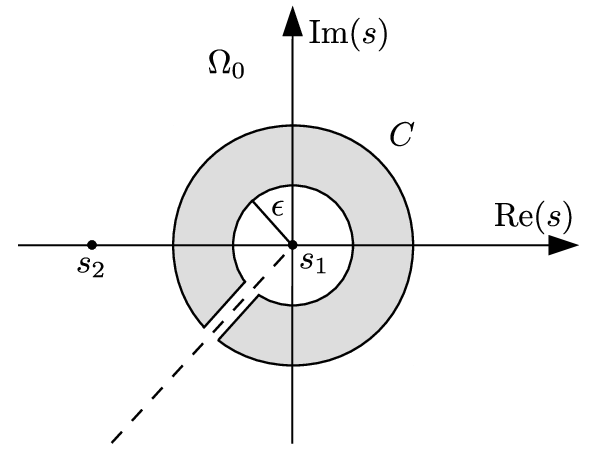}\label{f:Rouche2}}
	\caption{Dashed straight lines: branch cut $\BC_0^{\theta_1}$. Illustrations of the case when the regular spectrum at fixed $q$ is (a) different from $0$ (b) contains the origin; here the grey region is the interior of the contour $C$.}
	\label{f:Rouche}
\end{figure}

\begin{theorem}\label{t:ConvergenceModelA}
	For any compact set $K \subset \subset \Omega_0$, $\lim_{\gamma\to1} (K \cap \Lambda_\ss) = (K \cap \Lambda_\reg)$.
\end{theorem}

\begin{proof}
	Lemma \ref{l:ConvergeLocallyUniformly} shows that the subdiffusion (pseudo-)spectrum converges to the regular spectrum in any compact set of wavenumber. Lemma \ref{l:AsymptoticSolution} below tells us the subdiffusion (pseudo-)spectrum either tends to $0$ (denote by $s_0(q)$) or the real part tends to $-\infty$ for $|q|\gg 1$, i.e., outside the compact set of wavenumber. If the regular spectrum contains $0$, then the (pseudo-)spectrum naturally converges to the regular one in a compact set of $\Omega_0$. If the regular spectrum is negative for maximum, then Corollary \ref{c:smalldeltanonexist} below shows that $s_0(q)$ does not exist in $\Omega_0$ for $|q|\gg 1$ as $\delta\to0$. Therefore, the claim is proved.
\end{proof}

\medskip
We illustrate the convergence theorem numerically in Fig.~\ref{f:SpectrumComplexPlane}.

\subsection{Real spectrum and (pseudo-)spectrum for large wavenumber}\label{s:AsymptoticSpectrum}

As in the scalar case, we cannot solve the dispersion relation \eqref{e:FracDR} explicitly in general. However, as in Lemma~\ref{l:ScalarPositiveSol} we can determine real unstable spectrum as follows. This was observed in \cite[eq.~4.7]{Nec2007} and we include the proof for completeness.

\begin{lemma}\label{l:realunstablespec}
$D_{\reg}(s_r,\kappa^2) = 0$ with $\kappa \in \RR$ and $s_r > 0$ if and only if $D_\ss(s_r,q^2) = 0$ with $q = \kappa s_r^{-\delta/2} \in \RR$.
\end{lemma} 

\begin{proof}
	Rescaling the wavenumber of the regular and subdiffusion dispersion relation \eqref{e:RegDR} and \eqref{e:FracDR}, we find the relation
	\begin{align*}
	& D_{\reg}(s_r(\kappa),\kappa^2) = (s_r(\kappa) + \kappa^2 - a_{1})(s_r(\kappa) + d \kappa^2 - a_{4}) - a_{2} a_{3}\\
	&\quad = \left(s_r(\kappa) + s_r^\delta(\kappa)\frac{\kappa^2}{s_r^\delta(\kappa)} - a_{1}\right)\left(s_r(\kappa) + s_r^\delta(\kappa)\frac{d\kappa^2}{s_r^\delta(\kappa)} - a_{4}\right) - a_{2} a_{3}\\
	&\quad = D_\ss\left(s_r(\kappa),\frac{\kappa^2}{s_r^\delta(\kappa)}\right) = D_\ss(s_r(\kappa),q^2(\kappa)),
	\end{align*}
	where
	\begin{equation}\label{e:RescaleWavenumber}
	q(\kappa) =  \kappa s_r^{-\delta/2}(\kappa),\quad q(\kappa) \in \RR ,
	\end{equation}
	and $s_r(\kappa)$ denotes the regular spectrum with wavenumber $\kappa$. 
\end{proof}

\medskip
The following consequence of the lemma was noticed in \cite[Section 4.1]{Nec2007}. Again for completeness and in preparation of the following, we include the simple proof.

\begin{proposition}\label{p:PositiveRealSpectrumSufficientExplicit} 
	For any $\delta\in(0,1)$, if the diffusion coefficient $d>d_c$, then there exists a curve of real and strictly positive subdiffusion spectrum $s_0(q)$ with $q\in[\pm q_{\min},\pm\infty)$. Furthermore, $\lim_{d\to d_c}q_{\min}=\infty$ and $\lim_{|q|\to\infty}s_0(q)=0$. Specifically, $s_0(q)=s_r(\kappa)$ and $q = \kappa s_r^{-\delta/2} \in\RR $.
\end{proposition}

\begin{proof}
	For $d>d_c$ there is a curve of real unstable regular spectrum $s_r(\kappa)>0$ and $s_0(q)$ is given by Lemma~\ref{l:realunstablespec}, i.e., $s_0(q(\kappa))=s_r(\kappa)$.  Specifically, there is an interval $(\kappa_-^2,\kappa_+^2)$ such that $s_r(\kappa)>0$ for $\kappa^2\in(\kappa_-^2,\kappa_+^2)$. If $\kappa^2 \to \kappa_-^2,\kappa_+^2$, then $s_r(\kappa) \to 0$, which leads to $s_0(q) \to 0$ and $q(\kappa) \to \infty$. Furthermore, from \eqref{e:RescaleWavenumber} $q_{\min}^2 := \min_{\kappa^2\in(\kappa_-^2,\kappa_+^2)} q^2(\kappa) = \min_{\kappa^2\in(\kappa_-^2,\kappa_+^2)}\kappa^2/s_r^{\delta}(\kappa)$, and $q_{\min}^2 \to \infty$ as $d \to d_c$.
\end{proof}

\medskip
We illustrate Proposition \ref{p:PositiveRealSpectrumSufficientExplicit} numerically in Fig.~\ref{f:realpartspectrum}. Note that the curve from Proposition \ref{p:PositiveRealSpectrumSufficientExplicit} is only one part of the subdiffusion spectrum when $d > d_c$. Similar to the scalar case, next we discuss the (pseudo-)spectrum for $|q|\gg 1$. 

\begin{lemma}\label{l:AsymptoticSolution}
	For any $\delta\in(0,1)$ the solutions to \eqref{e:FracDR} for $|q|\gg 1$ are approximated by $s_{\infty 1}(q), s_{\infty 2}(q)$ and $s_{0\pm}(q)$ that satisfy $\lim_{|q|\to\infty} s_{0\pm}(q) = 0$ and $\lim_{|q|\to\infty}\Re(s_{\infty j}(q))=-\infty$, $j=1,2$. Moreover, there exists a $Q>0$ such that for $q>Q$ and $\delta \in (0, \theta_1/(\pi+\theta_1))$ we have
	\begin{align*}
		\Re(s_{\infty 1}(q)) & < \Re\left(\left(-Q^2\right)^{\frac{1}{1-\delta}} + \frac{a_{1}}{1-\delta} + 1\right) < 0,\\
		\Re(s_{\infty 2}(q)) & < \Re\left(\left(-d Q^2\right)^{\frac{1}{1-\delta}} + \frac{a_{4}}{1-\delta} + 1\right) < 0.
	\end{align*}
\end{lemma}

\begin{proof}
	We seek the asymptotic approximation of solutions to \eqref{e:FracDR} for $|q|\gg 1$ by rescaling $q=\kappa/\varepsilon$, so $q\to\infty$ for $\varepsilon\to 0$. We show in Appendix \ref{s:ProofAsymptotic} that the approximate solutions are given by
	\begin{align}
		s_{\infty 1}(q) &= \left(-q^2\right)^{\frac{1}{1-\delta}}+\frac{a_{1}}{1-\delta}+\cO(q^{\alpha-\beta}), \label{e:AsymptoticInfinitySolution1}\\
		s_{\infty 2}(q) &= \left(-dq^2\right)^{\frac{1}{1-\delta}}+\frac{a_{4}}{1-\delta}+\cO(q^{\alpha-\beta}), \label{e:AsymptoticInfinitySolution2}\\
		s_{0\pm}(q) &= \varepsilon^{2/\delta}y_{1\pm}^{1/\delta}+o(\varepsilon^{2/\delta}), \label{e:AsympoticZeroSol}
	\end{align}
	where $\beta>\alpha=\frac{2}{1-\delta}>0$, and $y_{1\pm}$ are the solutions to the following quadratic equation
	\begin{equation}\label{e:ZeroSolHigherOrder}
	d\kappa^4y_1^2-( a_{1}d+ a_{4})\kappa^2y_1+a_{1}a_{4}-a_{2}a_{3}=0.
	\end{equation}
	These lead to the claimed results. We refer to Appendix \ref{s:ProofAsymptotic} for a detailed proof.
\end{proof}

\begin{remark}
We emphasise that $s_{0\pm}(q)$ do not converge to the regular spectrum pointwise in $q$ as $\delta \to 0$ as they move out of the principal branch for any fixed $d$ as $\delta \to 0$ (cf.\ Corollary \ref{c:smalldeltanonexist} and Fig.~\ref{f:InstabilityRegion} below). It is the combination of parts $s_{\infty 1}$ and $s_{\infty 2}$, which converges to the regular spectrum locally uniformly in $q$, relating to Theorem \ref{t:ConvergenceModelA}.
\end{remark}

\subsection{Spectral instability for large wavenumbers}\label{s:largewav}

Here we discuss the stability and instability of (pseudo-)spectrum of \eqref{e:FracLinearSystem} for $|q|\gg 1$ in $\Omega_0$. Our aim is to give a complete picture of (pseudo-)spectrum and to understand how (pseudo-)spectrum moves in $\Omega_0$ to instability. The onset of Turing instability, in particular the critical diffusion coefficients $d$ for the subdiffusion model, has been studied in \cite{Nec2007}. We reformulate the instability results scattered in \cite{Nec2007} and combine these with the convergence theorem in Section~\ref{s:Convergence} as well as additional results of this section. This yields Theorem \ref{t:RealPartSpectrumInfinityWavenumber}; recall $d_c$ denotes the critical diffusion coefficient for the onset of the Turing instability with normal diffusion.

\begin{theorem}\label{t:RealPartSpectrumInfinityWavenumber}
	For any $\delta\in(0,1)$, there exists a unique $d_\delta^\infty \in (-a_{4}/a_{1}, d_c)$ and $Q>0$ such that for $d > d_\delta^\infty$ and any $|q|>Q$ there is spectrum $s(q) \in \Omega_0^+$. In particular, $\sup(\Re(\Lambda_\ss)) > \sup(\Re(\Lambda_\reg))$ for $d \in (d_\delta^\infty,d_c)$. Moreover, for any fixed $d_* \in (0, d_c)$, there exists a $\delta_*\in(0,1]$ such that $\sup(\Re(\Lambda_\ss)) < 0$ for any $\delta\in (0,\delta_*)$.
\end{theorem}

Since $d_\delta^\infty < d_c$ the subdiffusive transport of model \eqref{e:Henry2000} destabilises the homogeneous steady state \emph{before} normal diffusion does. This was already found in \cite{Henry2002} and may seem counterintuitive as subdiffusion heuristically slows down dynamics~\cite{Metzler2000}.
As mentioned, we add here a more detailed analysis and include the transition to instability. Recall that for the subdiffusion model, positive real parts imply exponential growth, so Theorem~\ref{t:RealPartSpectrumInfinityWavenumber} means  exponential growth of Fourier modes with arbitrarily large wavenumbers. Moreover, the critical diffusion coefficient $d_{\delta}^\infty$ has Turing-Hopf character, since the solutions to linearisation become unstable through oscillatory modes. 

The remainder of this section combined proves in particular Theorem~\ref{t:RealPartSpectrumInfinityWavenumber}.

\medskip
From \eqref{e:AsympoticZeroSol}, we observe that $s_0 \in \Omega_0^+$ ($\in \Omega_0^-$) if $y_1^{1/\delta} \in \Omega_0^+$ ($\in \Omega_0^-$) so we only discuss $y_1^{1/\delta}$. In case $y_1^{1/\delta} \in i\RR$ the sign of $\Re(s_0)$ is determined by higher order terms, which we do not consider. 

We start from \eqref{e:ZeroSolHigherOrder} and take $\kappa=\pm1$, so $q = \pm 1/\varepsilon$ and $|q|\to\infty$ as $\varepsilon\to0$. Then \eqref{e:ZeroSolHigherOrder} becomes
\begin{equation}\label{e:ZeroSolHigherOrderFixk}
	P(y_1):=dy_1^2-( a_{1}d+ a_{4})y_1+a_{1}a_{4}-a_{2}a_{3}=0,
\end{equation}
which is quadratic in $y_1$ and, since $d>0$, the minimum of $P$ is given by
\begin{equation*}
	P_{\min} = [4d(a_{1}a_{4}-a_{2}a_{3})-(a_{1}d+a_{4})^2]/(4d).
\end{equation*}

\medskip
\noindent\textbf{Case 1:} $P_{\min}\leq0$. There are two real solutions to \eqref{e:ZeroSolHigherOrderFixk}. 

\smallskip
\textit{Case 1a:} $a_{1}d+a_{4}>0$, i.e., $d>-a_{4}/a_{1}$. Equation \eqref{e:ZeroSolHigherOrderFixk} has one or two positive solutions, which leads to positive $y_{1\pm}^{1/\delta}$. It holds that
\begin{equation}\label{e:CriticalDiffRatio-finfty}
P_{\min}\leq0\quad \Rightarrow \quad F(d):= a_{1}^2d^2+(4a_{2}a_{3}-2a_{1}a_{4})d+a_{4}^2\geq0,
\end{equation}
whose roots $d_{fr-}$ and $d_{fr+}$ are both positive, but we exclude $d\leq d_{fr-}$ because $d_{fr-}<-a_{4}/a_{1}$ does not satisfy our assumption. Notably, $d_{fr+}=d_c$ is the Turing bifurcation point of system \eqref{e:RegAISystem} and for $d>d_c$ there exists a curve of real and strictly positive $s_0(q)$ as $q\in[\pm q_{\min},\pm\infty)$, cf.\ Proposition \ref{p:PositiveRealSpectrumSufficientExplicit}. For $d=d_c$, \eqref{e:ZeroSolHigherOrderFixk} has a positive double root $y_1$ which corresponds to positive $y_1^{1/\delta}$, thus there exists a $Q>0$ such that $\Re(s_{0}(q))>0$ for all $|q|>Q$. However, we do not know whether $\Im(s_{0}(q))$ is zero or not due to the higher order term.

\smallskip
\textit{Case 1b:} $a_{1}d+a_{4}<0$, i.e., $d<-a_{4}/a_{1}$. Equation \eqref{e:ZeroSolHigherOrderFixk} has one or two negative solutions and in the present case we exclude $d\geq d_{fr+}$. 
Since $\arg(y_{1\pm}) = \pi$, we have $\arg(y_{1\pm}^{1/\delta}) = \pi/\delta$ so that $y_{1\pm}^{1/\delta} \in \Omega_0$ if $\delta \in (\pi/(\pi+\theta_1),1)$ and thus in fact $y_{1\pm}^{1/\delta} \in \Omega_0^-$. In conclusion, if $\delta \in (\pi/(\pi+\theta_1),1)$ and $d\leq d_{fr-}$, then $y_1^{1/\delta} \in \Omega_0^-$. See also Remark \ref{r:CoincideComplexReal}.

\smallskip
\textit{Case 1c:} $a_{1}d+a_{4}=0$, i.e., $d=-a_{4}/a_{1}$. Since $(a_{1}a_{4}-a_{2}a_{3})>0$, the minimum $P_{\min}$ is always positive, which is the next case.

\bigskip
\noindent\textbf{Case 2:} $P_{\min}>0$. In this case, the solutions to \eqref{e:ZeroSolHigherOrderFixk} are complex conjugate. Denote $b:= a_{1}d+ a_{4}$, $\Delta:=4d(a_{1}a_{4}-a_{2}a_{3})-( a_{1}d+ a_{4})^2$; note $P_{\min}>0$ implies $\Delta > 0$.

\medskip
The following lemma provides the explicit formula of critical diffusion coefficient $d_\delta^\infty$, cf.\ the implicit result \cite[eq.~4.25]{Nec2007}.

\begin{lemma}\label{l:RealPartPositive}
	For any $\delta\in(0,1)$, there exists $d_\delta^\infty \in (-a_{4}/a_{1}, d_c)$ with $\lim_{\delta \to 0} d_\delta^\infty = d_c$, and $Q>0$ such that $s_{0\pm}(q) \in \Omega_0^+$ if $d>d_\delta^\infty$ and $|q|>Q$. Here $d_\delta^\infty$ is the larger root of 
	\begin{equation}\label{e:RealPartPositives}
H(d) := a_{1}^2 d^2 + (4(a_{2} a_{3} - a_{1} a_{4})\cos^2(\pi\delta/2) + 2a_{1} a_{4}) d + a_{4}^2 = 0.
	\end{equation}
\end{lemma}

\begin{proof}
	First, consider $b>0$, i.e., $d>-a_{4}/a_{1}$. The solutions to \eqref{e:ZeroSolHigherOrderFixk} are $y_{1\pm}=(b\pm i\sqrt{\Delta})/(2d)=:\rho e^{\pm i\theta}$, where $\pm\theta:=\arg (y_{1\pm})$ and $\rho > 0$. 
	Clearly, $y_{1\pm}^{1/\delta} \in \Omega_0^+$ if $\arg(y_{1\pm}^{1/\delta}) = \pm\theta/\delta \in (-\pi/2,\pi/2)$. Since $b>0$, we have $\theta=\arg (y_{1+})=\arctan(\sqrt{\Delta}/b)$. Thus $\theta/\delta \in (-\pi/2,\pi/2) \Rightarrow \arctan(\sqrt{\Delta}/b) \in (-\pi\delta/2,\pi\delta/2)$. Since $\arctan(\sqrt{\Delta} / b)>0$, we obtain $\arctan(\sqrt{\Delta}/b) \in (0,\pi\delta/2)$ which leads to 
	\begin{align*}
		4d(a_{1}a_{4}-a_{2}a_{3})-( a_{1}d+ a_{4})^2<( a_{1}d+ a_{4})^2\tan^2(\pi\delta/2),
	\end{align*}
	or equivalently $H(d)>0$. The roots $d_-$ and $d_\delta^\infty$ are both real valued, but we exclude $d<d_-<d_\delta^\infty$ because $d_-<-a_{4}/a_{1}$ contradicts the assumption.
	
	\smallskip
	For $b \leq 0$, i.e., $d \leq -a_{4}/a_{1}$ we have $\arg(y_{1+}) \geq \pi/2$ and thus $\arg(y_{1+}^{1/\delta}) \geq \pi/(2\delta) \notin (-\pi/2,\pi/2)$ as well as $\arg(y_{1-}^{1/\delta}) \leq -\pi/(2\delta) \notin (-\pi/2,\pi/2)$. 
		
	\smallskip
	The fact that $d_\delta^\infty < d_c$ for $\delta\in(0,1)$ follows by straightforward comparison of the solutions to \eqref{e:CriticalDiffRatio-finfty} and \eqref{e:RealPartPositives}.
\end{proof}

\medskip
The following makes the dependence of the pseudo-spectrum on the choice of branch cut explicit. From Fig.~\ref{f:InstabilityRegion} we can see that the range of existence is decreasing when $\theta_1 \to 0$, i.e., $\BC_0^{\theta_1}$ moves to the negative real line. In particular, for the canonical choice $\theta_1 = 0$ the pseudo-spectrum $s_{\infty 1},s_{\infty 2}$ is invisible. This explains the absence in \cite{Henry2002,Henry2005,Nec2007,Nec2008}.
\begin{figure}[t]
	\centering
	\subfloat[$\theta_1=\pi/2$]{\includegraphics[width=0.4\linewidth]{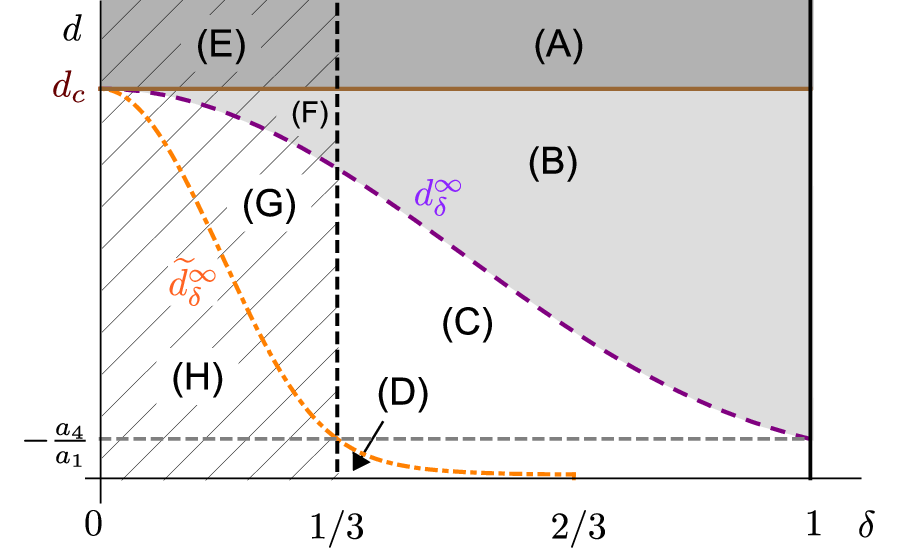}}
	\hfil
	\subfloat[$\theta_1 = 0$]{\includegraphics[width=0.4\linewidth]{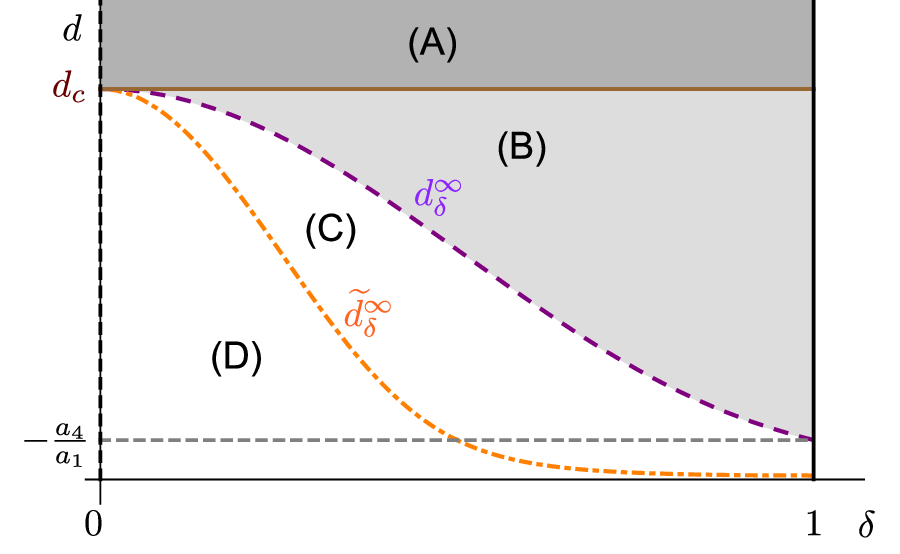}}
	\caption{Existence and stability of large wavenumber (pseudo-)spectrum. Brown solid horizontal lines: Turing instability threshold for normal diffusion in \eqref{e:RegAISystem}; purple dashed curves $d = d_\delta^\infty$: Turing-Hopf threshold for \eqref{e:ModelA}, i.e., $s_{0\pm} >0$ in regions (A), (E), whereas $s_{0\pm} \in \Omega_0^+$ with nonzero imaginary parts in regions (B), (F). Orange dashed dotted curves $d = \widetilde{d}_\delta^\infty$: existence of $s_{0+}\in \Omega_0^-$ in regions (C), (G) and $s_{0\pm}\notin \Omega_0$ in regions (D), (H). Vertical dashed line $\delta = \theta_1/(\pi+\theta_1)$ where $\delta=1/3$ in (a) and $\delta=0$ in (b): existence of $s_{\infty 1},s_{\infty 2} \in \Omega_0^-$ in regions (E)--(H) and $s_{\infty 1},s_{\infty 2} \notin \Omega_0$ in regions (A)--(D).}
	\label{f:InstabilityRegion}
\end{figure}

\begin{proposition}\label{p:ExistencePseudo}
	For any $\delta \in (0,1)$, there exist $Q>0$ and $\widetilde{d}_{\delta}^\infty \in [0,d_{\delta}^\infty)$ with $\lim_{\delta \to 0} \widetilde{d}_\delta^\infty = d_c$ such that for any $|q|>Q$ we have $s_{0+}(q) \in \Omega_0^-$ if $d \in (\widetilde{d}_{\delta}^\infty,d_{\delta}^\infty)$, and for any $\delta \in (0,\frac{\pi}{\pi+\theta_1})$, $s_{0\pm}(q) \notin \Omega_0$ if $d < \widetilde{d}_{\delta}^\infty$. More specifically, $\widetilde{d}_{\delta}^\infty$ is given by
	\begin{equation*}
		\widetilde{d}_{\delta}^\infty =
		\begin{cases}
		\widetilde{d}_{\delta+}^\infty, & \delta \in (0,\frac{\pi}{2(\pi+\theta_1)})\\
		-\frac{a_{4}}{a_{1}}, & \delta = \frac{\pi}{2(\pi+\theta_1)}\\
		\widetilde{d}_{\delta-}^\infty, & \delta \in (\frac{\pi}{2(\pi+\theta_1)},\frac{\pi}{\pi+\theta_1})\\
		\end{cases}
	\end{equation*}
	where $\widetilde{d}_{\delta+}^\infty$ and $\widetilde{d}_{\delta-}^\infty$ are the larger and smaller roots of
	\begin{equation}\label{e:ExistenceEqn}
\widetilde{H}(d):=a_{1}^2 d^2+(4(a_{2}a_{3} -a_{1}a_{4})\cos^2((\pi+\theta_1)\delta) + 2a_{1}a_{4})d+a_{4}^2 = 0.
	\end{equation} 
\end{proposition}

\begin{proof}
	The somewhat technical proof is given in Appendix \ref{s:DiscussionLemmas}.
\end{proof}

\begin{corollary}\label{c:smalldeltanonexist}
	For any fixed $d \in (0, d_c)$, there exist $\delta_\textrm{e}\in(0,1]$ and $Q>0$ such that $s_{0\pm}(q) \notin \Omega_0$ for any $\delta\in (0,\delta_\textrm{e})$ and any $|q|>Q$.
\end{corollary}

\medskip
Concerning the approximations $s_{\infty 1}$ and $s_{\infty 2}$, we have the following lemma.
\begin{lemma}\label{l:InfinitySolutionExistence}
	For any $\delta \in (0,1)$ there exists a $Q>0$ such that for any $|q|>Q$ we have $s_{\infty 1}(q),\, s_{\infty 2}(q)\in\Omega_0^-$ if $\delta \in (0,\frac{\theta_1}{\pi+\theta_1})$ whereas $s_{\infty 1}(q),\, s_{\infty 2}(q)\notin\Omega_0$ if $\delta \in [\frac{\theta_1}{\pi+\theta_1},1)$.
\end{lemma}

\begin{proof}
	Recall the approximations \eqref{e:AsymptoticInfinitySolution1}, \eqref{e:AsymptoticInfinitySolution2} for $|q|\gg1$ in the proof of Lemma \ref{l:AsymptoticSolution}. When $|q|\to\infty$, the arguments of $s_{\infty 1}$ and $s_{\infty 2}$ are determined by $\left(-q^2\right)^{\frac{1}{1-\delta}}$ and $\left(-dq^2\right)^{\frac{1}{1-\delta}}$, respectively, and are given by $\arg (s_{\infty 1}) = \arg (s_{\infty 2}) = \frac{\pi}{1-\delta}$. If $\frac{\pi}{1-\delta} \in (-\pi+\theta_1,\pi+\theta_1) \Rightarrow \delta < \frac{\theta_1}{\pi+\theta_1}$, then $s_{\infty 1},s_{\infty 2}\in\Omega_0$. Moreover, $\frac{\pi}{1-\delta} > \pi/2$, so $s_{\infty 1},\,s_{\infty 2} \in \Omega_0^-$.
\end{proof}

\medskip
Theorem~\ref{t:RealPartSpectrumInfinityWavenumber} now follows from combining Theorem \ref{t:ConvergenceModelA}, Proposition \ref{p:PositiveRealSpectrumSufficientExplicit}, Lemma \ref{l:RealPartPositive}, Proposition \ref{p:ExistencePseudo} and Lemma \ref{l:InfinitySolutionExistence}.

\subsection{Numerical computations of (pseudo-)spectra}\label{s:NumericalModel1}

Here we present some numerical computations of (pseudo-)spectrum for finite wavenumber. We choose $\theta_1 = \pi/2$ to give a most complete picture of the pseudo-spectrum. With $s=z^{m}$ and $\delta=\ell/m\in(0,1)$, $\ell,m\in\mathbb{Z}_+$ we get the dispersion relation
\begin{equation}\label{e:FracDRz}
D_\ss(z^m,q^2)=\left(z^{m}+z^\ell q^2- a_{1}\right)\left(z^{m}+z^\ell dq^2- a_{4}\right)- a_{2}a_{3}=0.
\end{equation}
in the $z$-plane. The principal branch $\Omega_0 = \left\{s\in\CC\setminus\{0\} : \arg (s)\in\left(-\frac{\pi}{2},\frac{3\pi}{2}\right)\right\}$ of $s$ corresponds to the branch in $z$-plane given by $\Sigma_0^m := \left\{z\in\CC\setminus\{0\} : \arg (z)\in\left(-\frac{\pi}{2m},\frac{3\pi}{2m}\right)\right\}$. We set $a_{1}=1/2$, $a_{2}=-3/16$, $a_{3}=8$, $a_{4}=-1$ so the regular Turing threshold is $d_c \approx 19.798$.

\medskip
Fig.~\ref{f:SpectrumComplexPlane} illustrates Theorem \ref{t:ConvergenceModelA}: the subdiffusion (pseudo-)spectra and regular spectra in the complex plane are plotted. We observe that when $\delta$ decreases, the subdiffusion (pseudo-)spectrum approaches the regular spectrum in $\Omega_0$.
\begin{figure}[t]
	\centering
	\subfloat[$\delta=1/6$]{\includegraphics[width=0.3\textwidth]{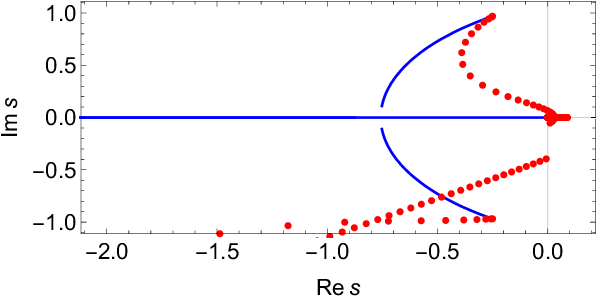}}
	\hfil
	\subfloat[$\delta=1/20$]{\includegraphics[width=0.3\textwidth]{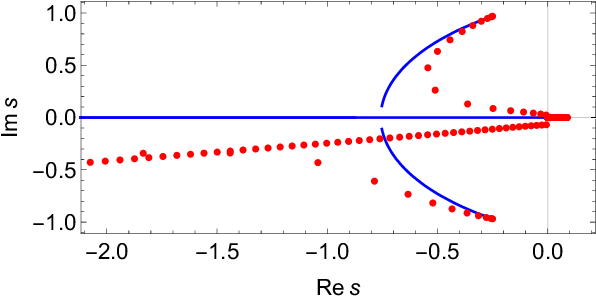}}
	\hfil
	\subfloat[$\delta=1/300$]{\includegraphics[width=0.3\textwidth]{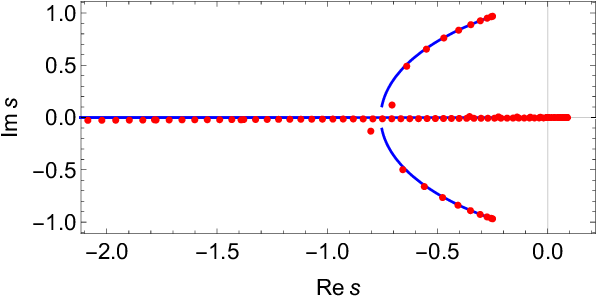}}
	\caption{Comparison of numerically computed spectra of \eqref{e:FracLinearSystem} beyond the regular Turing instability for diffusion coefficient $d=30$ and wavenumbers $q = 0.02n,\,n=0,1,\dots,150$. Blue solid lines: regular spectra; red dotted lines: subdiffusion (pseudo-)spectra.}
	\label{f:SpectrumComplexPlane}
\end{figure}

\medskip
Fig.~\ref{f:realpartspectrum} illustrates Lemma \ref{p:PositiveRealSpectrumSufficientExplicit}: when the regular spectrum has positive part, unstable subdiffusion spectrum along the scaled curve appears from large wavenumbers and moves towards smaller wavenumbers as $d$ increases. Note that maximum real parts of regular spectrum and scaled curve are the same. For decreasing real parts the wavenumber of the scaled curve tends to infinity.
\begin{figure}[t]
	\centering
	\subfloat[$d=20 > d_c$]{\includegraphics[width=0.3\linewidth]{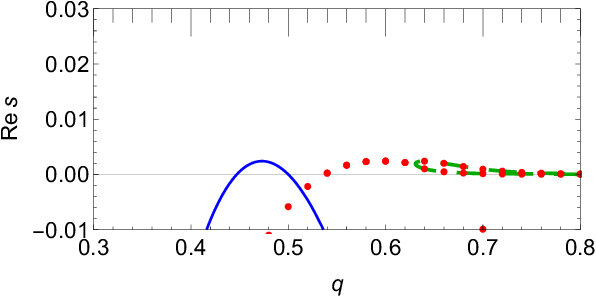}}
	\hfil
	\subfloat[$d=21 > d_c$]{\includegraphics[width=0.3\linewidth]{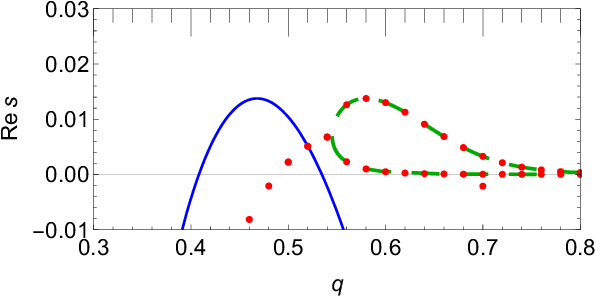}}
	\hfil
	\subfloat[$d=22 > d_c$]{\includegraphics[width=0.3\linewidth]{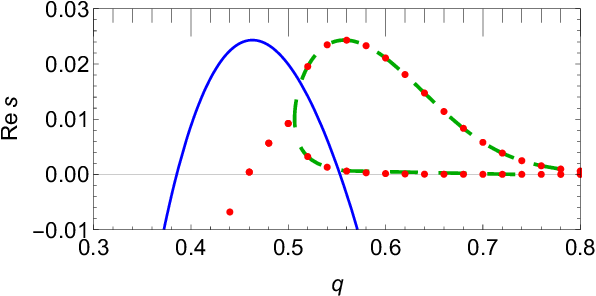}}
	\caption{Comparison of real parts of spectra of versus wavenumber beyond the regular Turing instability in \eqref{e:FracLinearSystem} for anomalous exponent $\delta=1/10$ and $q = 0.02n,\,n=0,1,\dots,150$. Blue solid lines: regular spectra; red dotted lines: subdiffusion (pseudo-)spectra; green dashed lines: real and positive curves scaled from the positive part of regular spectra.}
	\label{f:realpartspectrum}
\end{figure}

\medskip
In Fig.~\ref{f:realpartspectrumzoomin} we exhibit the behaviour of (pseudo-)spectrum as $d$ changes. In Fig.~\ref{f:FracSpecPB-complexzoomin-m10d15}, $d<\widetilde d_\delta^\infty\approx16.5$ so the real part of the regular spectrum and the more `unstable' subdiffusion pseudo-spectrum are both negative. As $d$ increases, the pseudo-spectrum moves towards imaginary axis and $\sup(\Re(\Lambda_\ss))=0$ for $\widetilde d_\delta^\infty < d < d_\delta^\infty\approx19.4$, cf.\ Fig.~\ref{f:FracSpecPB-complexzoomin-m10d18}. For $d>d_\delta^\infty$, the Turing-Hopf threshold, the subdiffusion spectrum is unstable, whereas the regular spectrum is stable as $d<d_c$, the regular Turing threshold, cf.\ Fig.~\ref{f:FracSpecPB-complexzoomin-m10d19p6}. Finally, when $d$ is larger than $d_c$, both spectra are unstable and additional purely real subdiffusion spectrum emerges as the scaled curve of regular spectrum, cf.\ Fig.~\ref{f:FracSpecPB-complexzoomin-m10d20}. These variations coincide with moving through regions (H), (G), (F), (E) in Fig.~\ref{f:InstabilityRegion}, respectively.
\begin{figure}[t]
	\centering
	\subfloat[$d=15 < \widetilde d_\delta^\infty$]{\includegraphics[width=0.35\linewidth]{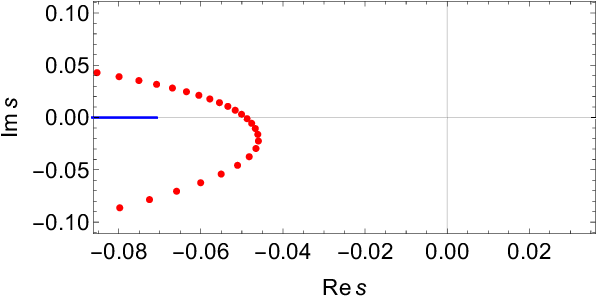}\label{f:FracSpecPB-complexzoomin-m10d15}}
	\hfil
	\subfloat[$\widetilde d_\delta^\infty<d=18 < d_\delta^\infty$]{\includegraphics[width=0.35\linewidth]{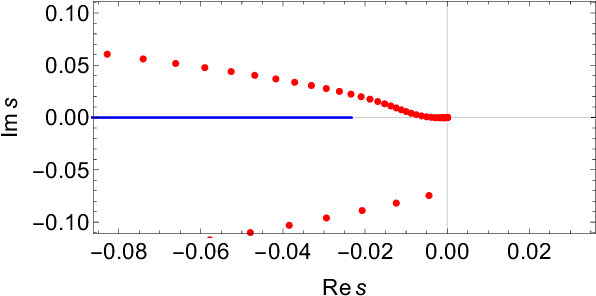}\label{f:FracSpecPB-complexzoomin-m10d18}}
	\hfil
	\subfloat[$d_\delta^\infty < d = 19.6 < d_c$]{\includegraphics[width=0.35\linewidth]{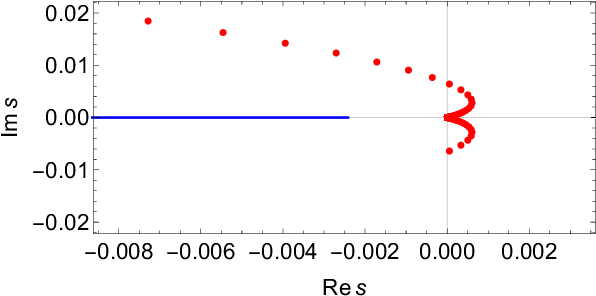}\label{f:FracSpecPB-complexzoomin-m10d19p6}}
	\hfil
	\subfloat[$d=20 > d_c$]{\includegraphics[width=0.35\linewidth]{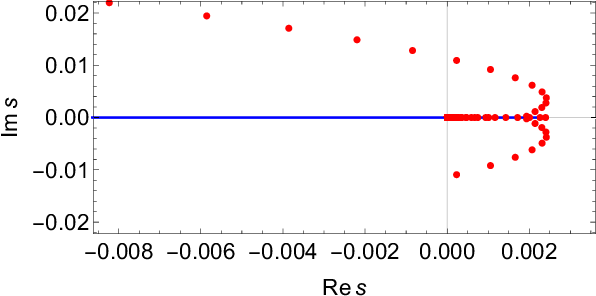}\label{f:FracSpecPB-complexzoomin-m10d20}}
	\caption{Comparison of spectra of \eqref{e:FracLinearSystem} in the principal branch $\Omega_0$ for anomalous exponent $\delta=1/10$ and wavenumber $q = 0.01n,\,n=0,1,\dots,300$. Blue solid lines: regular spectra; red dotted lines: subdiffusion (pseudo-)spectra.}
	\label{f:realpartspectrumzoomin}
\end{figure}

\medskip
Finally, in Fig.~\ref{f:comparespectrum} we compare the subdiffusion (pseudo-)spectrum in $\Omega_0$ and in $\Sigma_0^m$ in order to illustrate how pseudo-spectrum moves across the branch cut. For fixed $\delta$ and increasing $d$, the pseudo-spectrum in the $z$-plane moves into $\Sigma_0^m$, giving rise to the part of pseudo-spectrum which tends to zero. For $d > d_c$, the spectra in the $z$-plane move along the real line, which leads to the parts of the spectrum that tends to zero in $\Omega_0$. For fixed $d$ and decreasing $\delta$, the region $\Sigma_0^m$ becomes narrower, which removes certain pseudo-spectrum.
\begin{figure}[t]
	\centering
	\subfloat[$\delta=1/10$, $d=15<d_\delta^\infty$]{\includegraphics[width=0.35\linewidth]{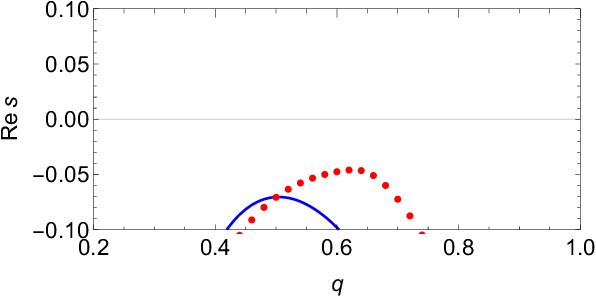}\label{f:FracSpecPB-comparerealpart-m10d15}}
	\hfil
	\subfloat[$\delta=1/10$, $d=15<d_\delta^\infty$]{\includegraphics[width=0.35\linewidth]{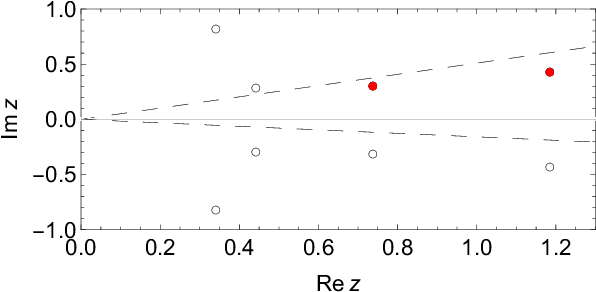}\label{f:FracSpecPB-comparecomplexplane-m10d15}}
	\hfil
	\subfloat[$\delta=1/10$, $d=18<d_\delta^\infty$]{\includegraphics[width=0.35\linewidth]{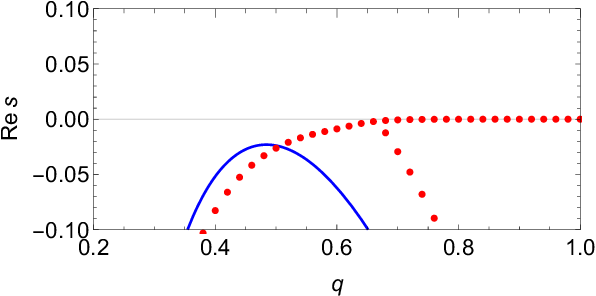}\label{f:FracSpecPB-comparerealpart-m10d18}}
	\hfil
	\subfloat[$\delta=1/10$, $d=18<d_\delta^\infty$]{\includegraphics[width=0.35\linewidth]{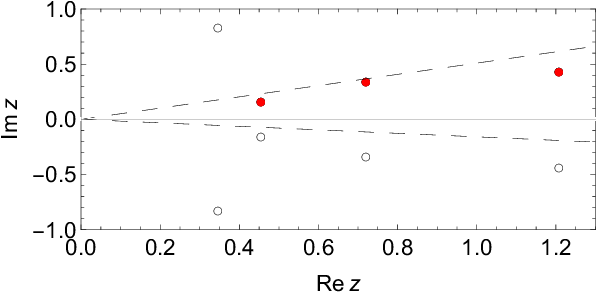}\label{f:FracSpecPB-comparecomplexplane-m10d18}}
	\hfil
	\subfloat[$\delta=1/20$, $d=18<d_\delta^\infty$]{\includegraphics[width=0.35\linewidth]{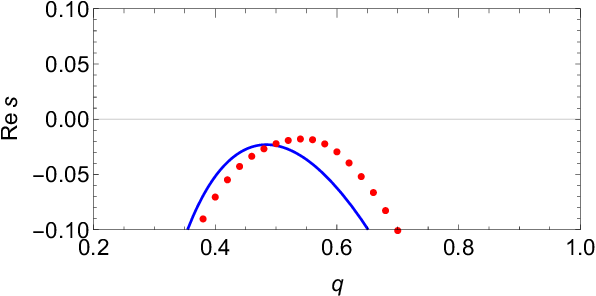}\label{f:FracSpecPB-comparerealpart-m20d18}}
	\hfil
	\subfloat[$\delta=1/20$, $d=18<d_\delta^\infty$]{\includegraphics[width=0.35\linewidth]{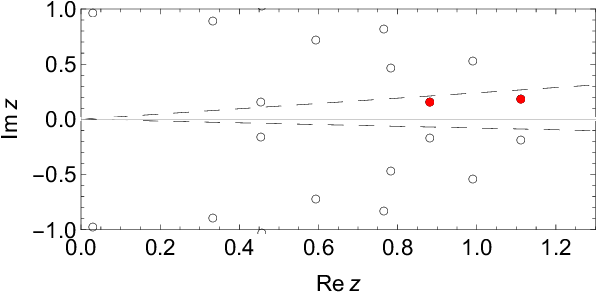}\label{f:FracSpecPB-comparecomplexplane-m20d18}}
	\hfil
	\subfloat[$\delta=1/20$, $d=22>d_c$]{\includegraphics[width=0.35\linewidth]{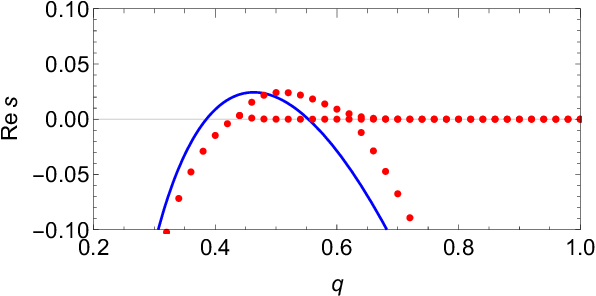}\label{f:FracSpecPB-comparerealpart-m20d22}}
	\hfil
	\subfloat[$\delta=1/20$, $d=22>d_c$]{\includegraphics[width=0.35\linewidth]{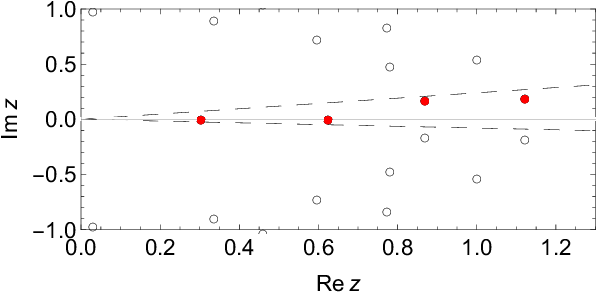}\label{f:FracSpecPB-comparecomplexplane-m20d22}}
	\caption{Model \eqref{e:FracLinearSystem}. Left column: the real part of spectra versus wavenumber; blue solid lines: regular spectrum; red dotted lines: subdiffusion (pseudo-)spectrum. Right column: (pseudo-)spectrum for $q=0.7$ in $z$-plane; hollow circles: roots of \eqref{e:FracDRz}; red dots: roots of \eqref{e:FracDRz} in $\Sigma_0^m$ with $\theta_1 = \pi/2$; dashed lines: borders of $\Sigma_0^m$.}
	\label{f:comparespectrum}
\end{figure}

\section{Subdiffusion with creation and annihilation}\label{s:Model2}

In this section we discuss the model \eqref{e:Langlands2008} with two components, namely
\begin{equation}\label{e:ModelB}
\partial_t \u = \D e^{\A t} \fD_{0,t}^{1-\gamma} \left( e^{-\A t} \partial_x^2 \u \right) + \A \u ,\quad \u\in\RR^2,
\end{equation}
which was derived from a process that only removes particles that have jumped to a location during the dynamics. Indeed, it was numerically shown in \cite{Langlands2008} that the Green's function is non-negative. Similar to the previous section we refine and augment some of the analysis done in the literature, in particular regarding the transition to Turing instability.

\subsection{Scalar case}\label{s:scalarmodel}

We recall the scalar case \eqref{e:Henry2006} and apply the Fourier-Laplace transform to obtain the dispersion relation
\begin{equation}\label{e:ModelB-DRscalar}
d_\ca(s,q^2) := (s-a)^{1-\gamma}\left((s-a)^\gamma+dq^2\right) = 0, \quad s \in \Omega_a,
\end{equation}
where $\Omega_a := \{s\in \CC\setminus\{a\} : \arg(s-a)\in (-\pi+\theta_1, \pi+\theta_1),\, \theta_1 \in (0,\pi/2) \}$ is the extended domain from $\Re(s)>a$ of the Laplace transform, cf.\ Remark \ref{r:ExtendReason}. We denote  $\Omega_a^+ := \{s\in \CC\setminus\{a\} : \arg(s-a)\in (-\pi/2, \pi/2) \}$ and $\Omega_a^{0-} := \Omega_a \setminus \Omega_a^+$.

Notably, $s=a$ is a constant solution to $d_{\ca}(s,q^2) = 0$, but not in $\Omega_a$.

\begin{definition}
	We call the set of roots $\lambda_\ca^+ := \{s\in\Omega_a^+: d_\ca(s,q^2) = 0\ \text{for a}\ q\in\RR \}$ \emph{(subdiffusion) spectrum} of the linear operator $\cL_\ca^\scalar :=de^{at}\fD_{0,t}^{1-\gamma}(e^{-at}\partial_x^2\, \cdot) + a$,
	and the set of roots $\lambda_\ca^{0-} := \{s\in\Omega_a^{0-}: d_\ca(s,q^2) = 0\ \text{for a}\ q\in\RR \}$ \emph{(subdiffusion) pseudo-spectrum} of $\cL_\ca^\scalar$.
\end{definition}

\medskip
We formulate the simple explicit solutions of \eqref{e:ModelB-DRscalar} as a lemma for reference and illustrate it numerically in Fig.~\ref{f:ModelBScalar}.

\begin{lemma}
	For any $\gamma \in (\frac{\pi}{\pi+\theta_1},1)$ and $|q|>0$, the solution to \eqref{e:ModelB-DRscalar} is $s_*(q) = (dq^2)^{1/\gamma} e^{i\pi/\gamma} + a \in \Omega_a$. In particular, $\lim_{|q|\to 0}\Re(s_*(q)) = a$ and $\lim_{|q|\to\infty}\Re(s_*(q)) = -\infty$; if $\gamma \in (0,\frac{\pi}{\pi+\theta_1}]$, then $s_*(q)\notin\Omega_a$ for each $q$; moreover, $s_*(0)\notin\Omega_a$.
\end{lemma}

\begin{proof}
	The solution is $s(q)= (-dq^2)^{1/\gamma}+a = (dq^2)^{1/\gamma} e^{i\pi/\gamma} + a$. Since $\arg(s-a) = \pi/\gamma$ we have $s \in \Omega_a$ if $\pi/\gamma \in (-\pi+\theta_1,\pi+\theta_1) \Rightarrow \gamma \in (\frac{\pi}{\pi+\theta_1},1)$. Since $\pi/\gamma > \pi$, we obtain that $\Re(s-a) < 0$ for $|q|>0$, and $\lim_{|q|\to\infty} \Re(s-a) = -\infty$. Moreover, $s(q) \notin \Omega_a$ if $\gamma \notin (\frac{\pi}{\pi+\theta_1},1)$, and $s(0) = a \notin\Omega_a$.
\end{proof}
\begin{figure}[t]
	\centering
	\subfloat[$a = -1/2$]{\includegraphics[width=0.3\linewidth]{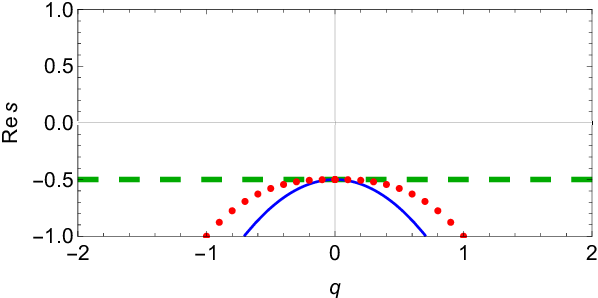}}
	\hfil
	\subfloat[$a = 0$]{\includegraphics[width=0.3\linewidth]{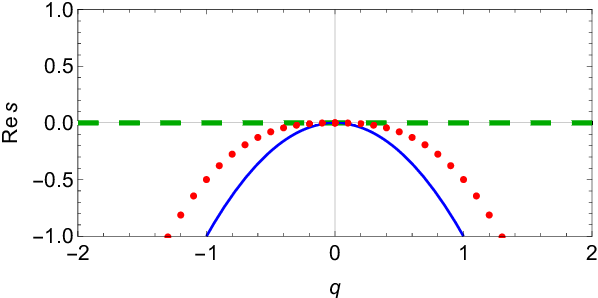}}
	\hfil
	\subfloat[$a= 1/2$]{\includegraphics[width=0.3\linewidth]{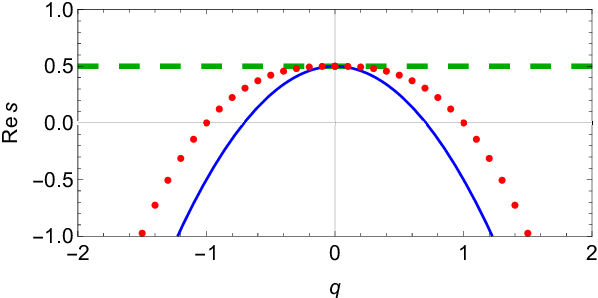}}
	\caption{Model \eqref{e:Henry2006}: comparison of real parts of spectra versus wavenumber with the branch cut angle $\theta_1 = \pi/2$ and $\gamma=3/4>2/3$. Blue solid lines: regular spectra; red dotted lines: subdiffusion (pseudo-)spectra; green dashed lines: $s(q) = a\notin\Omega_a$. Taking $\gamma=1/4<2/3$ the 
results are the same except the pseudo-spectrum is empty.}
	\label{f:ModelBScalar}
\end{figure}

\medskip
Notably, unlike the (pseudo-)spectrum of \eqref{e:ModelA-scalar}, the (pseudo-)spectrum of \eqref{e:Henry2006} is strictly stable for $|q|\gg 1$ and $\gamma\in(\frac{\pi}{\pi+\theta_1},1)$. 

\medskip
The solution to \eqref{e:Henry2006} with initial condition $u(x,0)=\delta(x)$ has been discussed in~\cite[Section V.B]{Henry2006} by establishing the relation with \eqref{e:Subdiffusion}, which is given by
\[
u(x,t) = \Phi(x,t) e^{at}, \quad t>0.
\]
Recall that $\Phi(x,t)$ is the Green's function to \eqref{e:Subdiffusion} given by \eqref{e:GreensFunction}. Notably, it follows that the non-trivial solutions to \eqref{e:Henry2006} is exponentially growing (decaying) for $a>0$ ($a<0$) and $t\gg 1$.

\medskip
The solution to the system \eqref{e:Langlands2008} with $\D = d\cdot\I$ has been discussed as well~\cite[Section III]{Langlands2008}, which is given by
\[
\u(x,t) = \Phi(x,t)e^{\A t}\u_0,
\]
with initial condition $\u(x,0) = \delta(x)\u_0$. The eigenvalues of $\A$ determine the decay as they are both negative, so the Turing instability cannot happen. For general diffusion matrix this approach fails and other instabilities occur as discussed next.

\subsection{Spectral analysis}\label{s:turingbifurcation}

We turn to the system \eqref{e:ModelB} with general diagonal diffusion matrix $\D=\diag(1,d)$, cf.\ \cite[eq.~3.5]{Nepomnyashchy2016}, where the (pseudo-)spectrum determines the temporal features of solutions analogous to Theorem~\ref{t:ILT-ss}. Moreover, we discuss the convergence of subdiffusion (pseudo-)spectrum when the subdiffusive exponent $\gamma\to1$ as well as approximate (pseudo-)spectra to detect instabilities. 

\medskip
In order to simplify the problem, as in~\cite{Langlands2008} we assume that $\A$ is a diagonalisable matrix, so that 
\begin{equation*}
	\P =
	\begin{pmatrix}
	a_{2}			&	a_{2}\\
	\mu_1-a_{1}	&	\mu_2-a_{1}
	\end{pmatrix},\quad
	\P ^{-1}=
	\begin{pmatrix}
	\frac{a_{1}-\mu_2}{a_{2}(\mu_1-\mu_2)}	&	\frac{1}{\mu_1-\mu_2}\\
	-\frac{a_{1}-\mu_1}{a_{2}(\mu_1-\mu_2)}	&	-\frac{1}{\mu_1-\mu_2}
	\end{pmatrix}
\end{equation*}
gives $\bar{\A}:=\P ^{-1}\A \P =\mathrm{diag}(\mu_1,\mu_2)$ (cf.\ \cite[eq.~27]{Nec2007a}), where $\mu_1$ and $\mu_2$ are the eigenvalues of $\A$. The conditions for Turing instability become $\tr(\A )=\mu_1+\mu_2<0$ and $\det(\A )=\mu_1\mu_2>0$. Without loss of generality we assume throughout that 
\[
\Re(\mu_1)\geq \Re(\mu_2).
\]
Changing coordinates $\u =\P \w $ turns \eqref{e:ModelB} into
\begin{equation}\label{e:DiagonalSystem}
	\partial_t\w   = \bar{\D }e^{\bar{\A}t}\fD_{0,t}^{1-\gamma}\left(e^{-\bar{\A}t}\partial_x^2 \w \right)+\bar{\A}\w ,
\end{equation}
where $\bar{\D }:=\P ^{-1}\D \P = \begin{pmatrix}
d_{1} & d_{2}\\
d_{3} & d_{4}
\end{pmatrix} $. Transforming \eqref{e:DiagonalSystem} into Fourier space yields
\begin{align}\label{e:ModelBFourier}
	\partial_t\hat\w = -q^2\bar{\D }e^{\bar{\A}t}\fD_{0,t}^{1-\gamma}\left(e^{-\bar{\A}t} \hat\w \right)+\bar{\A}\hat\w.
\end{align}
Laplace transform then gives for $\Re(s-\mu_1),\,\Re(s-\mu_2)>0$ the dispersion relation
\begin{equation}\begin{split}
&  D_\ca(s,q^2):=\det\left(s\I -\bar{\A}+\left(s\I -\bar{\A}\right)^{1-\gamma}\bar{\D }q^2\right)\\
 & = \det\left(\left(s\I -\bar{\A}\right)^{1-\gamma}\right)\det\left(\left(s\I -\bar{\A}\right)^{\gamma}+\bar{\D }q^2\right)\\
  &= \left(s-\mu_1\right)^{1-\gamma}\left(s-\mu_2\right)^{1-\gamma} \left[\left(\left(s-\mu_1\right)^\gamma+d_{1}q^2\right)\left(\left(s-\mu_2\right)^\gamma+d_{4}q^2\right)-d_{2}d_{3}q^4\right] = 0. \label{e:FracDRModelB}
	\end{split}\end{equation}
As in Section~\ref{s:Model1}, we extend the domain to $\Omega_\ca\subset \CC$ as follows. Clearly, the finite branch points of \eqref{e:FracDRModelB} are $\mu_1, \mu_2$. Taking $s$ along a contour around $\mu_1$ and $\mu_2$, the change in argument of $s$ is $\gamma(2\pi+2\pi)=4\pi\gamma$ and for $\gamma\neq1/2$ we have $e^{4\pi\gamma i}\neq1$ so $s=\infty$ is a branch point of $(s-\mu_1)^{\gamma}(s-\mu_2)^{\gamma}$. As we are most interested in the vicinity of $\gamma=1$, we take $\gamma \in (1/2,1)$. 

\smallskip
\noindent The conditions $\Re(\mu_1), \Re(\mu_2)<0$ for a Turing instability yield two cases: 
\begin{itemize}
\item[\textbf{cc}:] $\mu_1$, $\mu_2$ are complex conjugate with negative real parts;
\item[\textbf{nr}:] $\mu_1$, $\mu_2$ are negative real and $\mu_1>\mu_2$ (since the Turing conditions imply $a_2a_3<0$). 
\end{itemize}	

In order to keep non-integer powers of positive reals on the positive real axis, we choose the principal branch as follows:
	
\begin{itemize}
\item[\textbf{cc}:] $\Omega_\ca^\cc := \{s \in \CC\setminus\{\mu_1,\mu_2\} : \arg(s-\mu_1) \in (-\pi-\theta_1,\pi-\theta_1),\ \arg(s-\mu_2) \in (-\pi+\theta_2,\pi+\theta_2),\ \theta_1,\,\theta_2 \in (0,\pi/2),\ \Im(\mu_1) > 0 > \Im(\mu_2) \}$ (cf.\ Fig.~\ref{f:Rouche-cc}), i.e., branch cuts $\BC_{\mu_1}^{-\theta_1}$ and $\BC_{\mu_2}^{\theta_2}$.

\item[\textbf{nr}:]  $\Omega_\ca^\nr := \{s \in \CC\setminus\{\mu_1,\mu_2\} : \arg(s-\mu_1) \in (-\pi+\theta_1,\pi+\theta_1),\ \arg(s-\mu_2) \in (-\pi+\theta_2,\pi+\theta_2),\ \theta_1,\,\theta_2 \in (0,\pi/2),\ \mu_2 < \mu_1 < 0 \}$ (cf.\ Fig.~\ref{f:Rouche-nr}), i.e., branch cuts $\BC_{\mu_1}^{\theta_1}$ and $\BC_{\mu_2}^{\theta_2}$.
\end{itemize}
We further select $\theta_1,\theta_2$ such that for each $q$ the solutions of $D_\ca(s,q^2)=0$ do not lie on the branch cuts.
We write $\Omega_\ca$ for $\Omega_\ca^\cc$ or $\Omega_\ca^\nr$ whenever the case is not relevant, and denote: 
\begin{gather*}
\Omega_\ca^+ := \{s\in\Omega_\ca:\Re(s)>\Re(\mu_1)\},\quad \Omega_\ca^- := \{s\in\Omega_\ca:\Re(s)<\Re(\mu_1)\},\\ \Omega_\ca^{0+} := \Omega_\ca \setminus \Omega_\ca^-,\quad \Omega_\ca^{0-} := \Omega_\ca \setminus \Omega_\ca^+.
\end{gather*}

\begin{definition}
	We call the set of roots $\Lambda_\ca^+ := \{s\in\Omega_\ca^+: D_\ca(s,q^2)=0\ \text{for a}\ q\in\RR \}$ \emph{(subdiffusion) spectrum} of the linear operator $\cL_\ca :=\D e^{\A t} \fD_{0,t}^{1-\gamma} (e^{-\A t} \partial_x^2\,\cdot) + \A$, and the set of roots $\Lambda_\ca^{0-} := \{s\in\Omega_\ca^{0-}: D_\ca(s,q^2) = 0\ \text{for a}\ q\in\RR \}$ \emph{(subdiffusion) pseudo-spectrum} of $\cL_\ca$.
\end{definition}

We denote the union as 
$\Lambda_\ca := \Lambda_\ca^+ \cup \Lambda_\ca^{0-}$.

\medskip
In preparation, consider the non-trivial factor of $D_\ca$, 
\begin{align*}
D_{\ca2}(s,q^2) &:= \left(\left(s - \mu_1\right)^\gamma + d_{1} q^2\right) \left(\left(s - \mu_2\right)^\gamma + d_{4} q^2\right) - d_{2} d_{3} q^4 = 0, \quad s \in \Omega_\ca.
\end{align*}
We note that if $1-\gamma=n/m\in\mathbb{Q}$ then it can be cast as a polynomial with respect to the two variables $z_j=(s-\mu_j)^{1/m}$, $j=1,2$, of degree $m-n$ in each variable, given by
\[
D_{\ca2}(z_1,z_2,q^2) = (z_1^{m-n} + d_1q^2)(z_2^{m-n} + d_4q^2) - d_2d_3q^4.
\] 
By resultant theory, the polynomial system $D_{\ca2}(z_1,z_2,q^2) = 0$, $z_1^m + \mu_1 = z_2^m + \mu_2$ has a finite number of roots $(z_1,z_2)$, which implies that $D_{\ca2}(s,q^2) = 0$ has a finite number of roots $s$, including multiplicities.

Next we state the analogue of Theorem~\ref{t:ILT-ss} for the present model.  
Recall that for $q=0$ equation \eqref{e:ModelBFourier} reduces to $\partial_t\hat\w = \bar\A\hat\w$, whose solutions decay exponentially due to the Turing conditions. 
Hence, in the following we focus on $q\neq 0$ and with initial data $\hat\w(q,0) = (\hat w_{10},\hat w_{20})$ as well as $\mu_\Delta:=\mu_1-\mu_2$, we define the complex numbers
\begin{align*}
	C_{\bp,1} &:= \frac{(\mu_\Delta^\gamma+ d_4 q^2)\hat w_{10} - d_2 q^2\hat w_{20}}{d_1q^2(\mu_\Delta^\gamma + d_4 q^2) - d_2 d_3 q^4}\frac{\sin(\pi(1-\gamma))}{\pi}\Gamma(\gamma),\\
	C_{\bp,2} &:= \frac{((-\mu_\Delta)^\gamma+ d_1 q^2)\hat w_{20} - d_3 q^2\hat w_{10}}{d_4q^2((-\mu_\Delta)^\gamma + d_1 q^2) - d_2 d_3 q^4}\frac{\sin(\pi(1-\gamma))}{\pi}\Gamma(\gamma),\\
	C_{\bp,3} &:= \frac{d_3 ((\mu_\Delta^\gamma+d_4 q^2)\hat w_{10} - d_2q^2\hat w_{20}) }{\mu_\Delta^{1-\gamma}q^2(d_2 d_3 q^2-d_1(\mu_\Delta^{\gamma}+d_4 q^2))^2} \frac{\sin(\pi\gamma)}{\pi}\Gamma(1+\gamma).
\end{align*}
Note that $\Re(\mu_\Delta)=0$ in case \textbf{cc}, and $\mu_\Delta>0$ in case \textbf{nr} due to the Turing conditions.
In the above quantities the denominators are nonzero in $\Omega_\ca$ for $q\neq 0$ since $D_{\ca2}(\mu_j,q^2)\neq 0$, $j=1,2$, 
and $\mu_\Delta\neq0$ by assumption. 
We have $(C_{\bp,1},C_{\bp,2},C_{\bp,3})\neq 0$ for initial data outside the kernel of the matrix
\[
\begin{pmatrix}
\mu_\Delta^\gamma + d_4 q^2 & -d_2 q^2\\
-d_3 q^2 & (-\mu_\Delta)^\gamma + d_1 q^2
\end{pmatrix},
\]
and for almost all initial data each of $C_{\bp,1},C_{\bp,2},C_{\bp,3}$ is non-zero.

\begin{theorem}\label{t:ILT-ca}
	Let $\gamma\in(0,1)\cap\mathbb{Q}$ and
	$\lambda:=\sup(\Re(\Lambda_\ca))$. Let $\hat\w(q,t)$ be the solution to \eqref{e:ModelBFourier} with nonzero initial data $\hat\w_0$. It holds that 
	\begin{itemize}
	\item[(1)] If $\lambda\geq\Re(\mu_1)$ and $S^+:=\{(s,q)\in\Omega_\ca^{0+}\times (\RR\setminus\{0\}): D_\ca(s,q^2)=0,\ \text{and}\ \Re(s)\text{ maximal }\}\neq \emptyset$, then for any $(s_0,q_0)\in S^+$ we have $\hat\w(q,t) = C_{\exp}t^{k-1} e^{s_0t} + o(t^{k-1}e^{\Re(s_0)t})$ with $C_{\exp}\in\CC^2$ nonzero for almost all initial data and $k$ the multiplicity of $s_0$ as the root of $D_{\ca2}(s,q_0^2)=0$.

	\item[(2)] If (i) $\lambda=\Re(\mu_1)$ and $S^+=\emptyset$ or (ii) $\lambda\in(\Re(\mu_2),\Re(\mu_1))$, then $S^-:=\{(s,q)\in\Omega_\ca^-\times (\RR\setminus\{0\}): D_\ca(s,q^2)=0,\ \text{and}\ \Re(s)\text{ maximal } \} \neq \emptyset$ and for any $(s_0,q_0)\in S^-$, with $k, C_{\exp}$ as in item (1) it holds that 
	\begin{align*}
\hat\w(q,t) = 
	\begin{pmatrix}
C_{\bp,1} & 0\\ C_{\bp,3}\,t^{-1} & C_{\bp,2}
\end{pmatrix}\begin{pmatrix}
e^{\mu_1 t} \\ e^{\mu_2 t}
\end{pmatrix}t^{-\gamma} + C_{\exp}t^{k-1} e^{s_0t} + o(t^{-\gamma} e^{\Re(\mu_1) t}).
\end{align*}
	\item[(3)] If (i) $\lambda\leq\Re(\mu_2)$ or (ii) $\Lambda_\ca=\emptyset$, then for any $q\in \RR\setminus\{0\}$ it holds that
	 	\begin{align*}
\hat\w(q,t) = 
	\begin{pmatrix}
C_{\bp,1} & 0\\ C_{\bp,3}\,t^{-1} & C_{\bp,2}
\end{pmatrix}\begin{pmatrix}
e^{\mu_1 t} \\ e^{\mu_2 t}
\end{pmatrix}t^{-\gamma} + o(t^{-\gamma} e^{\Re(\mu_1) t}).
\end{align*}
	\end{itemize}
\end{theorem}
We defer the technical proof to Appendix \ref{s:InverseLaplaceNonzeroBranchPoint}. Concerning case (1) in the theorem we remark that if $\lambda>\Re(\mu_1)$, then the set $S^+$ is guaranteed to be non-empty.
Regarding cases (2) and (3) in the theorem, we note that the leading order term in the second component of $\hat\w(q,t)$ is $C_{\bp,2} e^{\mu_2 t}t^{-\gamma}$ in case \textbf{cc}, while it is $C_{\bp,3} e^{\mu_1 t}t^{-\gamma-1}$ in case \textbf{nr}, i.e., $\mu_\Delta>0$. We also have absorbed a  term of order $e^{\mu_2 t}t^{-\gamma-1}$ into the remainder term.

\begin{remark}
The statement also holds for $\Re(\mu_1)=0$ and scalar equations, and thus includes \eqref{e:Henry2006} as well as the subdiffusion equation \eqref{e:Subdiffusion}. Notably, for the latter Theorem~\ref{t:ILT-ca} gives the decay as $t^{-\gamma}$ (for rational $\gamma$), which coincides with that derived in Section~\ref{s:Subdiffusion} via \eqref{e:AsympML}.
\end{remark}

\begin{remark}\label{l:rational}
As for Theorem~\ref{t:ILT-ss} the method of proof only allows for rational $\gamma$ and it would be 
interesting to investigate the general case. 
In contrast to Theorem \ref{t:ILT-ss}, the coefficients $C_{\bp,j},j=1,2,3$ decay as $q^{-2}$ in Theorem \ref{t:ILT-ca} so that we expect some smoothening of initial data.
\end{remark}

Similar to Theorem~\ref{t:ILT-ss}, Theorem \ref{t:ILT-ca} reveals that the roots of \eqref{e:FracDRModelB} determine the temporal behaviour of $\hat{\w}$. However, in contrast to Theorem \ref{t:ILT-ss}, the Turing instability condition $\Re(\mu_1)<0$ implies exponential decay for stable (pseudo-)spectrum. The key point is that on the one hand $\mu_1,\mu_2$ are the branch points of \eqref{e:ModelBFourier}, while the branch point of \eqref{e:ModelAFourier} is the origin.  On the other hand, $s=\mu_1, \mu_2$ also solves the dispersion relation. This is one of the reasons for the algebraic correction $t^{-\gamma}$ compared with the algebraic decay $t^{\gamma-2}$ in Theorem \ref{t:ILT-ss}.

\begin{remark}\label{r:relatetoNepo}
As mentioned in Section~\ref{s:modelling}, model \eqref{e:Langlands2008} is different from \eqref{e:Nec2007a} discussed in \cite{Nec2007a} due to \eqref{e:modelBdifference}, as well as the dispersion relations, cf.\ \eqref{e:FracDRModelB} and \cite[eq.~24]{Nec2007a}. This can be illustrated for the scalar case. The dispersion relation of \eqref{e:Nec2007a} reads $(s-a)^\gamma+dq^2=0$ which has no roots on the branch point $s=a$ for $q\neq 0$, whereas \eqref{e:ModelB-DRscalar} does. Then according to Theorem \ref{t:ILT-ss}, each Fourier mode of \eqref{e:Nec2007a} with $q\neq0$ decays algebraically as $t^{\gamma-2}$, whereas the one of \eqref{e:Henry2006} decays exponentially by Theorem \ref{t:ILT-ca}.
	
	In the multispecies case, the eigenvalues of $\A$ are roots of \eqref{e:FracDRModelB} as well as the branch points, whereas they are not roots of \cite[eq.~24]{Nec2007a} in general. This implies the Fourier modes of \eqref{e:Langlands2008} and \eqref{e:Nec2007a} have different decays. Another difference in the system case in \cite[eq.~28]{Nec2007a} is the perhaps unrecognised implicit assumption that $\P^{-1}(s\I - \A)^{\gamma} \P$ is equal to $(s\I - \P^{-1}\A\P)^{\gamma}$.
\end{remark}

\subsection{Convergence to regular spectrum}\label{s:convergence}

As in Theorem~\ref{t:ConvergenceModelA} the subdiffusion (pseudo-)spectrum converges to the regular spectrum as $\gamma\to1$. The regular dispersion relation of \eqref{e:RegLinearSystem} in the coordinates $\u =\P \w $ reads
\begin{equation}\label{e:normaldisprel}
D_{\reg}(s,q^2)=(s-\mu_1+d_{1}q^2)(s-\mu_2+d_{4}q^2)-d_{2}d_{3}q^4=0, \quad s\in\CC
\end{equation}
which allows for direct comparison with the non-trivial factor of \eqref{e:FracDRModelB} given by
\begin{equation}\label{e:FracDRModelB2}
	D_{\ca2}(s,q^2) :=  \left(\left(s - \mu_1\right)^\gamma + d_{1} q^2\right) \left(\left(s - \mu_2\right)^\gamma + d_{4} q^2\right) - d_{2} d_{3} q^4 = 0, \quad s \in \Omega_\ca.
\end{equation}
Notably, $s = \mu_1,\,\mu_2$ solve \eqref{e:FracDRModelB} for any $q$ but are not in $\Omega_{\ca}$, while these are roots for \eqref{e:normaldisprel} at $q = 0$, and generically not otherwise. 

\begin{theorem}\label{t:ConvergenceModelB}
	For any compact set $K\subset \subset \Omega_\ca$, $\lim_{\gamma \to 1}(K \cap \Lambda_\ca) = (K \cap \Lambda_\reg)$.
\end{theorem}

\begin{proof}
The basis of the proof is the analogue of Lemma \ref{l:ConvergeLocallyUniformly}. In both cases, the application of Rouch\'{e} theorem near and away from the branch points are completely analogous to that in Lemma \ref{l:ConvergeLocallyUniformly}, cf.\ Fig.~\ref{f:convergence}.
\end{proof}
\begin{figure}[t]
	\centering
	\subfloat[]{\includegraphics[width=0.35\linewidth]{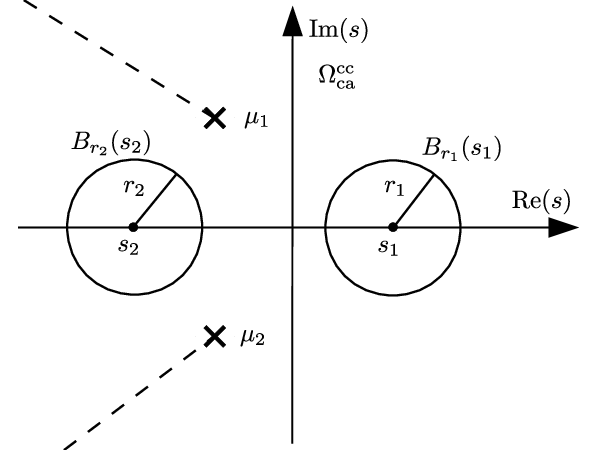}\label{f:Rouche-cc}}
	\hfil
	\subfloat[]{\includegraphics[width=0.35\linewidth]{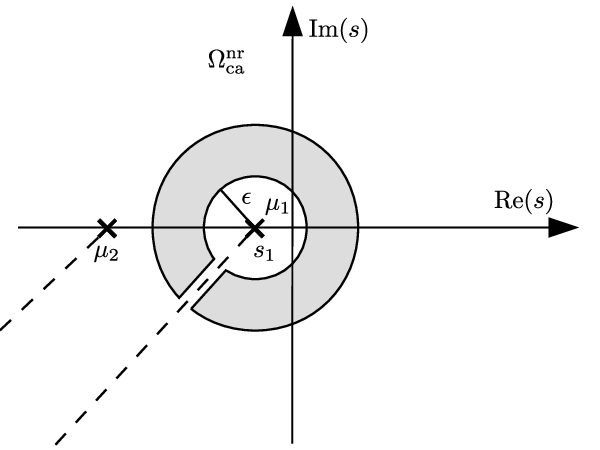}\label{f:Rouche-nr}}
	\caption{Branch cuts (dashed) and points (crosses) for \eqref{e:FracDRModelB}, and regular spectrum (dots) $s_1,\,s_2$. (a) Case \textbf{cc}; (b) case \textbf{nr}.}
	\label{f:convergence}
\end{figure}

\subsubsection{Spectrum near zero and Turing instability}

Since the onset of Turing instability concerns spectrum near the origin, we can analyse this by Taylor expanding the dispersion relation at $s=0$. 

\medskip
We will show that the Taylor expansion of $D_{\ca 2}=0$ to quadratic order is
	\begin{equation}\label{e:Taylorapproxdisprel}
	\widetilde{D}_{\ca 2}(s,q^2) := s^2 +\gamma^{-1}\left(\beta_1 q^2-\tr(\A )\right)s+\gamma^{-2}h(q^2)=0,
	\end{equation}
	where $\beta_1(\gamma,d) := c_{1}+c_{4}d$, $\beta_2(\gamma) := a_{1}c_{4}-a_{2}c_{3}$, $\beta_3(\gamma) := a_{4}c_{1}-a_{3}c_{2}$, and
	\begin{align*}
	\C &:=\P \left(-\bar{\A}\right)^{1-\gamma}\P ^{-1} = \begin{pmatrix}
	c_{1}	&	c_{2}\\
	c_{3}	&	c_{4}
	\end{pmatrix},&
	h(q^2) &:= \det(\C )\cdot dq^4-\left(\beta_2 d+\beta_3\right)q^2+\det(\A ).
	\end{align*}
	Note that $\det(\C ) = \det ((-\bar{\A})^{1-\gamma}) = (\mu_1\mu_2)^{1-\gamma} > 0$.
Now the computation of the Turing instability parameters for \eqref{e:Taylorapproxdisprel} follows, with some caveats, the approach for two component systems with regular diffusion. Since this gives the onset of instability through the origin, the same follows for the subdiffusion spectrum. As detailed in the following Theorem statement, the Turing instability parameters are
	\begin{align*}
	\gamma_\cc &:= \frac{1}{\theta}\arccot\left(\frac{a_{1}}{|\mu|\sin(\theta)}+\cot(\theta)\right) \in(0,1), \\
	\gamma_\nr &:= \left(\ln\left(\frac{a_{1}-\mu_1}{a_{1}-\mu_2}\right)\right)\left(\ln\left(\frac{\mu_1}{\mu_2}\right)\right)^{-1}\in(0,1),\\
	d_\gamma &:= -\frac{\beta_3}{\beta_2}+\frac{2}{\beta_2^2}\left(\det(\A \C )+\sqrt{\left(\det(\A \C )\right)^2-\beta_2\beta_3\det(\A \C )}\right), &
	q_\gamma^2 &:=\frac{\beta_2d_\gamma+\beta_3}{2d_\gamma\det(\C )}.
	\end{align*}
We write $\Lambda_\ca^\epsilon:= \Lambda_\ca\cap B_{\epsilon}(0)$ for the spectrum within a small ball $B_{\epsilon}(0)$ near the origin and set $\lambda_\epsilon := \sup(\Re(\Lambda_\ca^\epsilon))$. 

\begin{remark} 
The following Theorem holds verbatim for \eqref{e:Taylorapproxdisprel} instead of \eqref{e:FracDRModelB2} and any $\epsilon$, i.e., replacing $\Lambda_\ca^\epsilon$ by the roots of $\widetilde{D}_{\ca 2}$.
\end{remark}

\begin{theorem}\label{t:CriticalDiffusionRatioModel2}
For any sufficiently small $\epsilon>0$ there exists a unique minimal anomalous exponent $\gamma_{\A }\in(0,1)$ given by $\gamma_\cc$ in case \textbf{cc} and $\gamma_\nr$ in case \textbf{nr}, such that for any $\gamma \in (0,\gamma_\A )$, either $\lambda_{\epsilon} < 0$, or $\Lambda_\ca^\epsilon = \emptyset$, and for any $\gamma\in(\gamma_{\A },1)$ the following hold.
	\begin{itemize}
		\item[(1)] $d_\gamma$ is the unique critical diffusion coefficient such that $\sgn(\lambda_{\epsilon}) = \sgn(d - d_\gamma)$.
		\item[(2)] For $d = d_\gamma$, there exists a unique critical wavenumber given by $q_{\gamma}$ such that $\lambda_{\epsilon} = 0$, $\Lambda_\ca^\epsilon\cap i\RR=\{0\}$, and $\widetilde{D}_{\ca 2}(0,q^2) = 0$ precisely for $q=q_\gamma$.
		\item[(3)] $\lim_{\gamma \to \gamma_\A }d_\gamma = +\infty$ and $\lim_{\gamma \to 1} d_\gamma = d_c$.
	\end{itemize}
	\end{theorem}

\begin{proof}
	\underline{Case \textbf{cc}}: Taylor expanding, with $\mu_j\neq0$, $j=1,2$, 
	\[
	(s-\mu_j)^\gamma = (-\mu_j)^\gamma + \gamma(-\mu_j)^{\gamma-1} s + \cO(|s|^2),
	\]
	gives the approximate dispersion relation \eqref{e:FracDRModelB2} in $B_{\epsilon}(0)$ quadratic in $s$ as
	\begin{equation}\label{e:Taylorapproxdisprelation}
	\left[(-\mu_1)^\gamma+\gamma(-\mu_1)^{\gamma-1}s+d_{1}q^2\right]\left[(-\mu_2)^\gamma+\gamma(-\mu_2)^{\gamma-1}s+d_{4}q^2\right]-d_{2}d_{3}q^4=0.
	\end{equation}
	In order to see the dependence on the diffusion ratio $d$, we perform the following transform. First, multiply \eqref{e:Taylorapproxdisprelation} with $\gamma^{-1}(-\mu_1)^{1-\gamma}\gamma^{-1}(-\mu_2)^{1-\gamma}$ and rewrite the resulting equation as
	\begin{equation*}
	\det\left(s\I +\gamma^{-1}q^2\bar{\D }\left(-\bar{\A}\right)^{1-\gamma}-\gamma^{-1}\bar{\A}\right)=0.
	\end{equation*}
	Second, change coordinates through left-multiplying by $\P$ and right-multiplying by $\P ^{-1}$ to
	\begin{equation*}
	\det\left(s\I +\gamma^{-1}q^2\D \P \left(-\bar{\A}\right)^{1-\gamma}\P ^{-1}-\gamma^{-1}\A \right)=0.
	\end{equation*}
	Here the matrix $\C =\P \left(-\bar{\A}\right)^{1-\gamma}\P ^{-1}$ can be expressed as
	\begin{equation*}
	\C =
	\begin{pmatrix}
	c_{1}	&	c_{2}\\
	c_{3}	&	c_{4}
	\end{pmatrix} = 
	\begin{pmatrix}
	\frac{(-\mu_1)^{1-\gamma}(a_{1}-\mu_2)-(-\mu_2)^{1-\gamma}(a_{1}-\mu_1)}{\mu_1-\mu_2}				&	\frac{a_{2}((-\mu_1)^{1-\gamma}-(-\mu_2)^{1-\gamma})}{\mu_1-\mu_2}\\
	\frac{(a_{1}-\mu_1)(a_{1}-\mu_2)((-\mu_2)^{1-\gamma}-(-\mu_1)^{1-\gamma})}{a_{2}(\mu_1-\mu_2)}	&	\frac{(-\mu_2)^{1-\gamma}(a_{1}-\mu_2)-(-\mu_1)^{1-\gamma}(a_{1}-\mu_1)}{\mu_1-\mu_2}
	\end{pmatrix}.
	\end{equation*}
	from which somewhat tedious computations give $\widetilde{D}_{\ca 2}(s,q^2)$, and this is $\cO (\epsilon^2)$-close to $D_{\ca2}(s,q^2)$ for $s\in B_\epsilon(0)$. 
	
	We claim that $\beta_1>0$ if $d>1$. Set $-\mu_1=|\mu| e^{i\theta}$ and $-\mu_2=|\mu| e^{-i\theta}$, where without loss of generality $\theta \in (0,\pi/2)$. Then $\beta_1$  simplifies, which directly shows the claim:
	\begin{align*}
	\beta_1(\gamma,d) & = \tfrac{a_{1}(d-1)\left(\left(-\mu_2\right)^{1-\gamma}-\left(-\mu_1\right)^{1-\gamma}\right)+\mu_1\mu_2\left(\left(-\mu_1\right)^{-\gamma}-\left(-\mu_2\right)^{-\gamma}\right)+d\left(\left(-\mu_2\right)^{2-\gamma}-\left(-\mu_1\right)^{2-\gamma}\right)}{\mu_1-\mu_2}\\
	& = \frac{a_{1}(d-1)|\mu|^{-\gamma}\sin(\theta(1-\gamma))+|\mu|^{1-\gamma}\sin(\theta\gamma)+d|\mu|^{1-\gamma}\sin(\theta(2-\gamma))}{\sin(\theta)}.
	\end{align*}
	
	Since $\tr(\A )<0$, the prefactor of $s$ in \eqref{e:Taylorapproxdisprel} is positive so that \eqref{e:Taylorapproxdisprel} has a solution $s(q)>0$ if and only if $h(q^2)<0$ for some $q$. We note that $\det(\A )>0$ and
	\begin{equation*}
	\det(\C ) = \det ((-\bar{\A})^{1-\gamma}) = (\mu_1\mu_2)^{1-\gamma} > 0,
	\end{equation*}
	so $h(q^2)<0$ for some real $q$ requires
	\begin{equation}\label{e:necessaryconditionmodel2}
	\beta_2 d+\beta_3>0,
	\end{equation}
	where $\beta_3<0$ follows by straightforward calculation. As to the sign of $\beta_2$ we compute 
	\begin{align*}
	\beta_2(\gamma) & = -\frac{\mu_1\mu_2}{\mu_1-\mu_2}\left[a_{1}\left(\left(-\mu_2\right)^{-\gamma}-\left(-\mu_1\right)^{-\gamma}\right)+\left(\left(-\mu_2\right)^{1-\gamma}-\left(-\mu_1\right)^{1-\gamma}\right)\right]\\
	& = \frac{|\mu|^{1-\gamma}}{\sin(\theta)}\left[a_{1}\sin(\theta\gamma)-|\mu|\sin(\theta(1-\gamma))\right].
	\end{align*}
 Hence, $H(\gamma):=a_{1}\sin(\theta\gamma)-|\mu|\sin(\theta(1-\gamma))$ has the sign of $\beta_2(\gamma)$, and $\beta_2(\gamma)=0$ if and only if $H(\gamma)=0$, which is equivalent to
	\begin{equation}\label{e:mingammaequivalenteqn}
	\cot(\theta\gamma) = \frac{a_{1}}{|\mu|\sin(\theta)}+\cot(\theta).
	\end{equation}
	The solution to \eqref{e:mingammaequivalenteqn} is given by $\gamma_\cc$; note that it depends on $\A$ only. From \eqref{e:mingammaequivalenteqn} we know that $\cot(\theta\gamma_\cc) > \cot(\theta)$, so $\gamma_\cc < 1$. Since 
	\begin{equation*}
	\frac{\dif}{\dif\gamma}H(\gamma) = a_{1} \theta \cos(\theta\gamma) + \mu \theta \cos(\theta(1-\gamma)) > 0.
	\end{equation*}
we have $\sgn(H(\gamma))=\gamma-\gamma_\cc$. For $\gamma<\gamma_\cc$ the condition \eqref{e:necessaryconditionmodel2} implies $d < -\beta_3/\beta_2 < 0$, which violates the assumption $d>0$. Hence, the following conditions must be satisfied
	\begin{equation}\label{e:ModelB-InstabilityConditions}
	\gamma > \gamma_\cc \quad \mathrm{and} \quad d > -\beta_3/\beta_2>0,
	\end{equation}
	and $\gamma\to\gamma_\cc^+$ implies $\beta_2 \to 0^+$, which means $d \to +\infty$. 
	
	Straightforward computations give the minimum of $h(q^2)$ and the associated argument 
	\begin{equation*}
	h_{\min} = \frac{4d\det(\A \C )-\left(\beta_2d+\beta_3\right)^2}{4d\det(\C )}, \quad q^2_{\min} = \frac{\beta_2d+\beta_3}{2d\det(\C )}.
	\end{equation*}
	If $h_{\min}<0$, then $h(q^2)<0$ for $q^2\in\left(q^2_-,q^2_+\right)$, where $q^2_{\pm}$ are the two roots of $h(q^2)=0$. Furthermore, $h_{\min}<0$ gives
	\begin{equation*}
	\beta_2^2d^2+\left(2\beta_2\beta_3-4\det(\A \C )\right)d+\beta_3^2 > 0,
	\end{equation*}
	whose boundary points are
	\begin{equation*}
	d_\pm = -\frac{\beta_3}{\beta_2}+\frac{2}{\beta_2^2}\left(\det(\A \C )\pm\sqrt{\left(\det(\A \C )\right)^2-\beta_2\beta_3\det(\A \C )}\right).
	\end{equation*}
	Since the discriminant is positive $d_{\pm}\in\RR$ exist, but $d<d_-$ does not satisfy \eqref{e:ModelB-InstabilityConditions}. Hence, $d_\gamma$ is the Turing bifurcation point and satisfies Item (1) for \eqref{e:Taylorapproxdisprel}. Since this means the onset of instability is through the origin $s=0$, the same and the following characterisation hold for the subdiffusion spectrum sufficiently close to the origin. Moreover, $\lim_{\gamma \to \gamma_\cc} d_{\gamma} = +\infty$ and $\lim_{\gamma \to 1} d_{\gamma} = d_c$. The critical wavenumber is given by $q_\gamma^2$.
	
	\medskip
	\underline{Case \textbf{nr}}: Similar to the case \textbf{cc}, the dispersion relation \eqref{e:FracDRModelB2} can be approximated by \eqref{e:Taylorapproxdisprelation} within $B_{\epsilon}(0)$, but here the minimum anomalous exponent is given by $\gamma_\nr$. We note $\mu_1\neq\mu_2$, $a_{1}\neq\mu_1$ and $a_{1}\neq\mu_2$ due to the Turing conditions on $\A$, which guarantees the existence of $\gamma_\nr$.
\end{proof}

\medskip
This theorem shows that the critical spectrum of system \eqref{e:ModelB} has the ``Turing shape'' of regular diffusion, i.e., near the origin it is real, has maxima at selected wavenumbers, and crosses the origin for increasing $d$. Moreover, the relation of reaction and subdiffusive motion in  \eqref{e:ModelB} stabilises the solution, i.e., the more anomalous the diffusion is, the more `difficult' in terms of the diffusion ratio it is for the solution to become unstable. It is even impossible to be unstable when the diffusion motion becomes too anomalous, i.e., is below $\gamma_\A$. We illustrate the stable and unstable region in Section~\ref{s:NumercialModel2}, cf.\ Fig.~\ref{f:instabilityregionmodel2} below.

\begin{remark}
Theorem~\ref{t:CriticalDiffusionRatioModel2} is valid only near the origin and thus instabilities through the imaginary axis may be missed. However, due to Theorem \ref{t:ConvergenceModelB}, if $\gamma$ is close to $1$, this cannot happen. \end{remark}

\subsection{Numerical computations of spectra}\label{s:NumercialModel2}

We illustrate Theorem \ref{t:ConvergenceModelB} with numerical computations for the cases \textbf{cc} and \textbf{nr}, respectively, where $\A$ in \eqref{e:ModelB} is given by
\begin{equation*}
\A_{\cc}=
\begin{pmatrix}
\frac{1}{2}	&	-\frac{3}{16}\\
8			&	-1
\end{pmatrix}
\quad\text{and}\quad
\A_{\nr}=
\begin{pmatrix}
1				&	1\\
-\frac{17}{8}	&	-2
\end{pmatrix}.
\end{equation*}
Here $\A_{\cc}$ has eigenvalues $\mu_{\pm}\approx-0.25\pm0.968i$, and $\A_{\nr}$ has eigenvalues $\mu_1\approx-0.854$, $\mu_2\approx-0.146$. In Fig.~\ref{f:instabilityregionmodel2} we plot the Turing instability threshold $d_\gamma$ and the minimum anomalous exponent $\gamma_{\A }$, and determine the stable and unstable region based on the results of Section~\ref{s:turingbifurcation}.
\begin{figure}[t]
	\centering
	\subfloat[Case \textbf{cc}]{\includegraphics[width=0.4\linewidth]{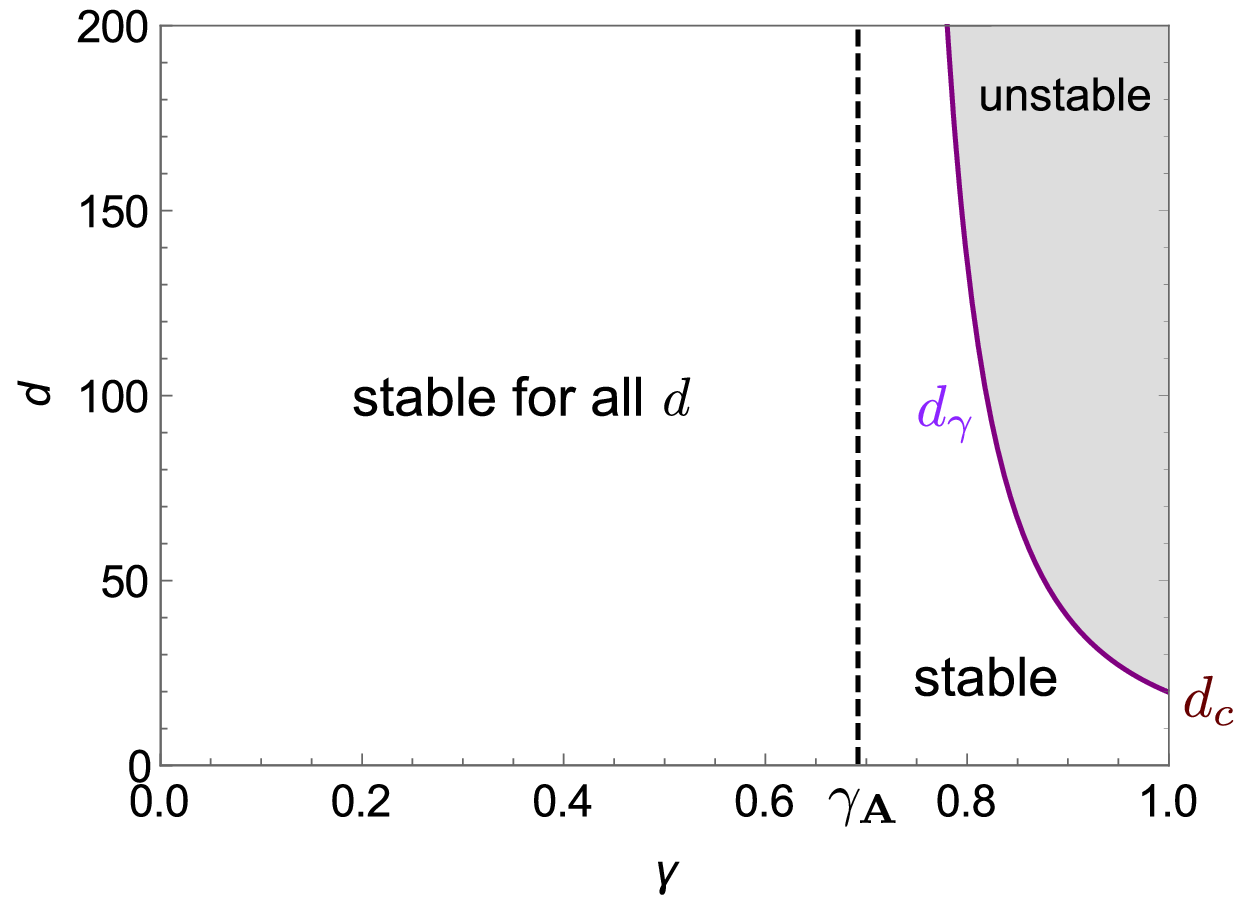}\label{f:instabilityregionmodel2complex}}
	\hfil
	\subfloat[Case \textbf{nr}]{\includegraphics[width=0.4\linewidth]{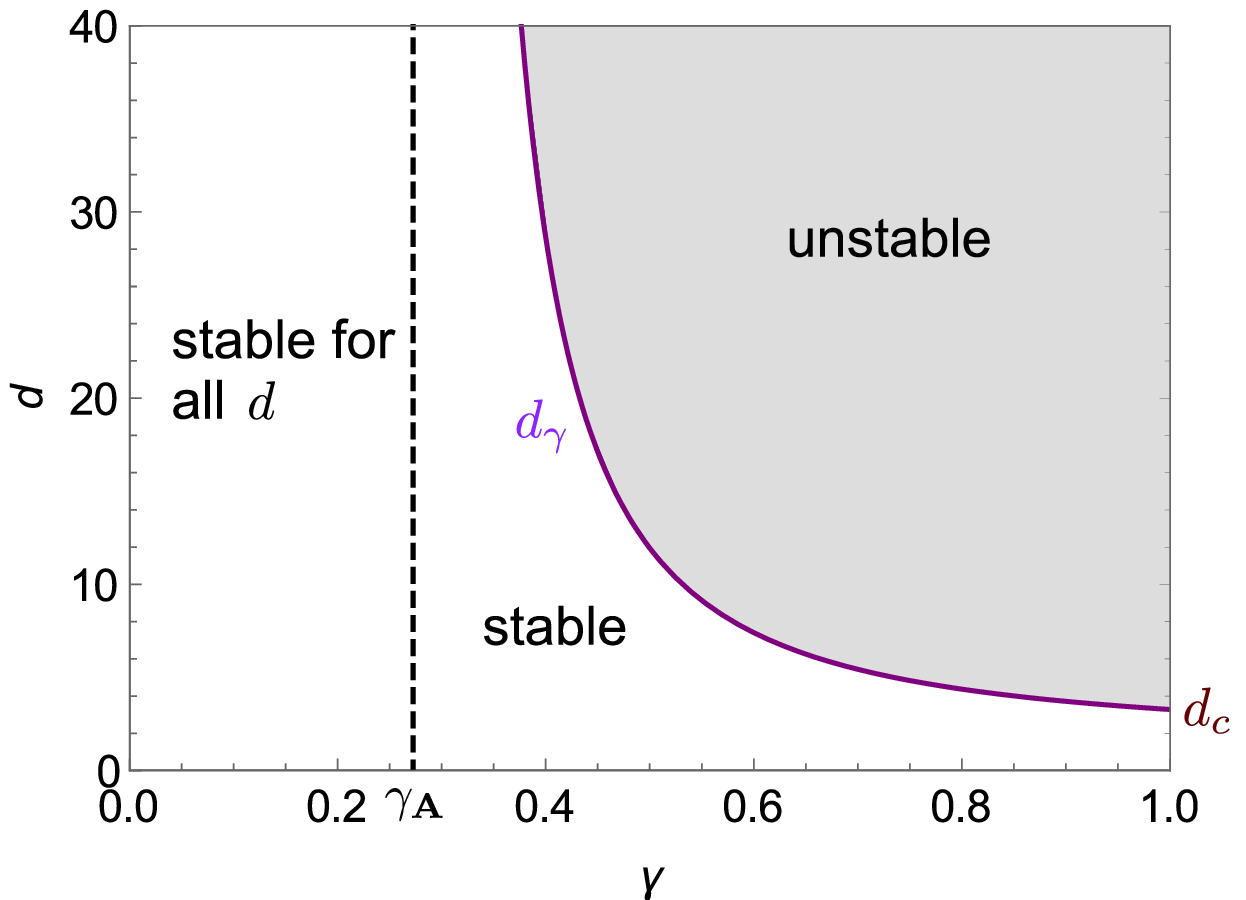}\label{f:instabilityregionmodel2negative}}
	\caption{Plotted are samples of stability and instability regions of spectra for \eqref{e:ModelB} near the origin in terms of the Turing threshold $d_\gamma$ (purple solid curves), which terminates at the regular threshold $d_c$ and lies above the minimum anomalous exponents $\gamma_\A $ (vertical dashed lines). (a) $\gamma_\A =0.69$, $d_c=19.798$ and (b) $\gamma_\A =0.27$, $d_c=3.28$.}
	\label{f:instabilityregionmodel2}
\end{figure}

\medskip
In Fig.~\ref{f:SpectrumComplexPlane-model2} we plot results with fixed diffusion coefficient $d$ and different anomalous exponents $\gamma$. As predicted, the subdiffusion (pseudo-)spectrum approaches the regular one as $\gamma$ increases towards $1$. Remark that in case \textbf{nr} a gap within the pseudo-spectra  occurs as a result of solutions to \eqref{e:FracDRModelB2} outside $\Omega_\ca$.
\begin{figure}[t]
	\centering
	\subfloat[$\gamma=4/5$]{\includegraphics[width=0.3\linewidth]{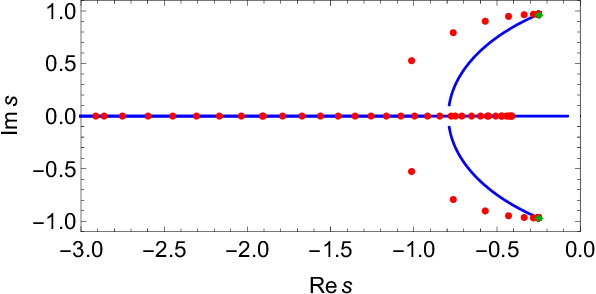}}
	\hfil
	\subfloat[$\gamma=9/10$]{\includegraphics[width=0.3\linewidth]{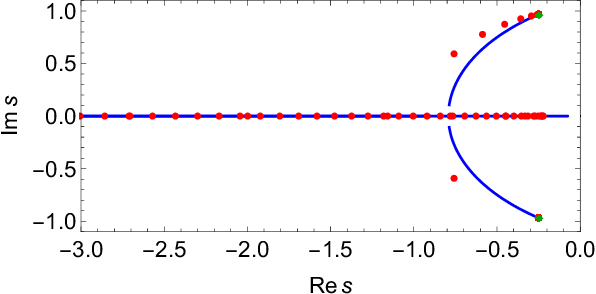}}
	\hfil
	\subfloat[$\gamma=19/20$]{\includegraphics[width=0.3\linewidth]{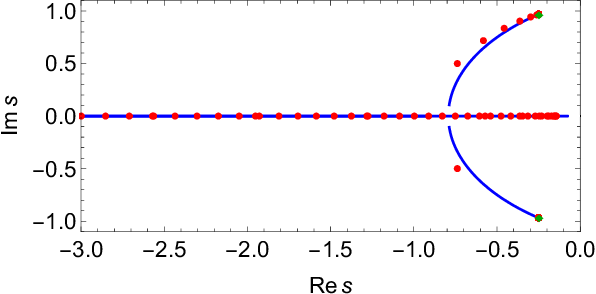}}
	\hfil
	\subfloat[$\gamma=4/5$]{\includegraphics[width=0.3\linewidth]{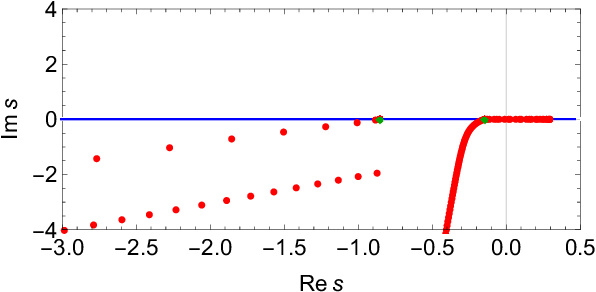}}
	\hfil
	\subfloat[$\gamma=9/10$]{\includegraphics[width=0.3\linewidth]{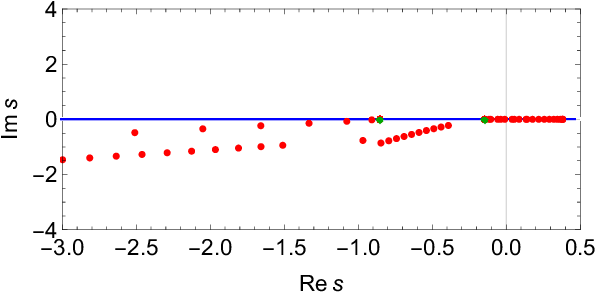}}
	\hfil
	\subfloat[$\gamma=19/20$]{\includegraphics[width=0.3\linewidth]{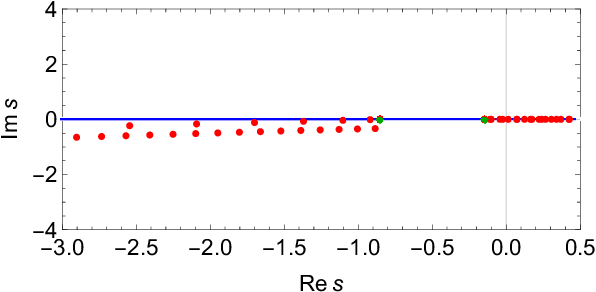}}
	\caption{Plotted are (pseudo-)spectra of \eqref{e:ModelB} for $\A=\A_{\cc},\A_{\nr}$ for wavenumbers $q = 0.04n,\,n=0,1,\dots,100$: regular diffusion ($\gamma=1$, blue solid lines) and subdiffusion $(\gamma\in(0,1)$, red dotted lines), and the eigenvalues of $\A$ (green diamonds). Top row: $\A=\A_{\cc}$ (case \textbf{cc}), $d=15$; bottom row: $\A=\A_{\nr}$ (case \textbf{nr}), $d=20$.}
	\label{f:SpectrumComplexPlane-model2}
\end{figure}

\medskip
In Fig.~\ref{f:complexspectrum} we plot results with fixed anomalous exponent $\gamma$ and different diffusion coefficients $d$ for case \textbf{cc}. As predicted, the subdiffusion spectrum becomes unstable when $d$ exceeds $d_\gamma$. Notably, the spectrum from \eqref{e:Taylorapproxdisprel} nicely approximates the numerical spectrum computed from \eqref{e:FracDRModelB2}. Furthermore, the subdiffusive transport stabilises the spectrum, since the maximum of the subdiffusion spectrum is less than that of the regular spectrum. The situation is 
similar for case \textbf{nr}, except the subdiffusion (pseudo-)spectrum in this region connects with the regular spectrum.
\begin{figure}[t]
	\centering
	\subfloat[$d=35$]{\includegraphics[width=0.3\linewidth]{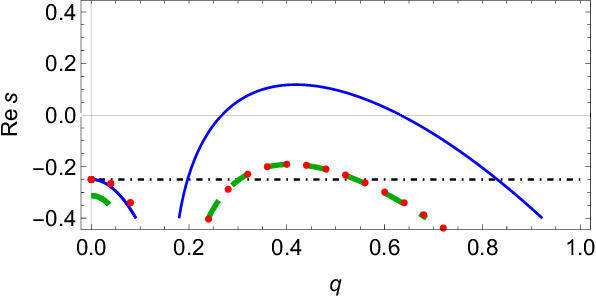}\label{f:complex-m5d35}}
	\hfil
	\subfloat[$d=d_\gamma\approx136.177$]{\includegraphics[width=0.3\linewidth]{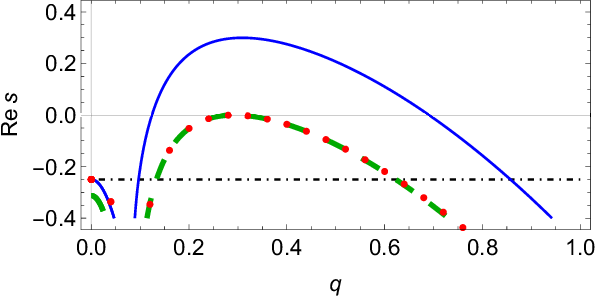}\label{f:complex-m5d136}}
	\hfil
	\subfloat[$d=250$]{\includegraphics[width=0.3\linewidth]{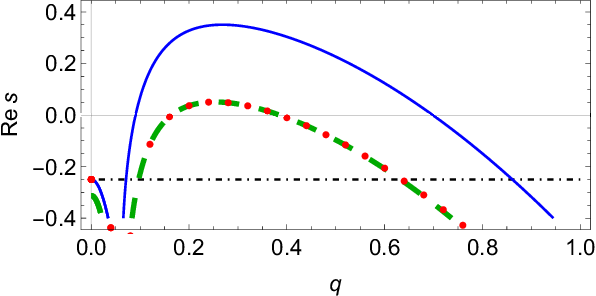}\label{f:complex-m5d250}}
	\caption{Plotted are (pseudo-)spectra of \eqref{e:ModelB} for $\A=\A_\cc$ and $\gamma=4/5$. Subdiffusion (pseudo-)spectra (red dotted) computed from \eqref{e:FracDRModelB}, regular diffusion spectra (blue solid), and approximate subdiffusion (pseudo-)spectra computed from \eqref{e:Taylorapproxdisprel} (green dashed). Horizontal dotted dashed lines are $s=\mu_1\notin\Omega_\ca$.}
	\label{f:complexspectrum}
\end{figure}

\appendix

\section{Wright function}\label{s:WrightFunction}

The Wright function is defined by the series~\cite[Section 1.11]{Kilbas2006}
\begin{equation*}
\phi(a,b;z):=\sum_{n=0}^{\infty}\frac{1}{\Gamma(an+b)}\frac{z^n}{n!},\quad a,b,z\in\CC .
\end{equation*}
If $a>-1$, this series is absolutely convergent for all $z\in\CC $ and it is an entire function of $z$. Hence, this series is uniformly convergent within $\left|z\right|<R$, where $R$ is any positive constant. Concerning the Green's function \eqref{e:GreensFunction}, we have, with $\mu:=\gamma/2\in(0,1/2)$, 
\begin{equation*}
\Phi(x,t) = \frac{1}{\sqrt{4d t^\gamma}}\sum_{n=0}^{\infty}\frac{(-1)^n}{n!\Gamma\left(1-\frac{\gamma}{2}-\frac{\gamma}{2}n\right)}\left(\frac{|x|}{\sqrt{d t^\gamma}}\right)^n = \frac{t^{-\mu}}{2\sqrt{d}}\phi\left(-\mu,1-\mu;-\frac{|x|}{\sqrt{d}}t^{-\mu}\right),
\end{equation*}
and $\Phi(x,t)>0$ for any $t>0, x\in \RR$ \cite[eq.~4.26-a]{Mainardi2001}. Since the Wright function in $\Phi(x,t)$ is uniformly convergent for $t>(|x|/(R\sqrt{d}))^{1/\mu}$, we can interchange the limit and sum for compact set of $x$, yields 
\begin{align*}
	& \lim_{t\to\infty}\frac{\Phi(x,t)}{t^{-\mu}} = \lim_{t\to\infty}\frac{1}{2\sqrt{d}}\phi\left(-\mu,1-\mu;-\frac{|x|}{\sqrt{d}}t^{-\mu}\right) \\
	&\quad = \frac{1}{2\sqrt{d}}\left(\frac{1}{\Gamma(1-\mu)} + \sum_{n=1}^{\infty}\lim_{t \to \infty}\frac{(-1)^n}{n!\Gamma\left(1-\frac{\gamma}{2}-\frac{\gamma}{2}n\right)}\left(\frac{|x|}{\sqrt{d t^\gamma}}\right)^n\right) = \frac{1}{2\sqrt{d}\Gamma(1-\mu)},
\end{align*}
so the Green's function $\Phi(x,t)\sim (2\sqrt{d}\Gamma(1-\mu)t^{\mu})^{-1}$ is algebraically decaying locally uniformly in $x$ for $t\to\infty$. Moreover, $\Phi(x,t)$ has the following asymptotic representation for $|x|/(\sqrt{d}t^{\mu})\to\infty$~\cite[eq.~4.27]{Mainardi2001},
\begin{gather*}
	\Phi(x,t) \sim \frac{t^{-\mu}}{2\sqrt{d}}A_0 Y^{\mu-1/2}\exp(-Y),\\
	A_0=(\sqrt{2\pi}(1-\mu)^\mu \mu^{2\mu-1})^{-1},\quad Y=(1-\mu)(\mu^\mu|x|t^{-\mu}/\sqrt{d})^{1/(1-\mu)}.
\end{gather*}
Therefore, $\Phi(x,t)$ is algebraically decaying in $t$ for $|x|\gg\sqrt{d}t^{\mu}$ with power $-\mu/(2-2\mu)\in(-\mu,0)$.

\section{Proof of Theorem \ref{t:ILT-ss}}\label{s:InverseLaplaceZeroBranchPoint}

This proof relies on calculating the inverse Laplace transform (ILT) with zero branch point, but where no roots of the dispersion relation are on the branch point at the origin.

\medskip
For any fixed $q$, the Fourier-Laplace solution $\fL\hat\u=(\fL\hat u_1,\fL\hat u_2)$ of \eqref{e:FracLinearSystem} can be written as 
\begin{align}\label{e:LaplaceFourierKernel}
\fL\hat u_1(q,s) = \frac{(s+ s^{\ell/m}d q^2 -a_4)\hu_{10} + a_2\hu_{20}}{(s+s^{\ell/m} q^2- a_{1})(s+s^{\ell/m} dq^2- a_{4})- a_{2}a_{3}}  =: \Psi(s),
\end{align}
cf.\ \cite{Henry2005}, with the initial conditions $\hu_1(q,0) = \hu_{10}(q)$ and $\hu_2(q,0) = \hu_{20}(q)$.
It suffices to discuss $\fL\hat u_1(q,s)$ and we suppress the subscript for convenience.

The denominator of \eqref{e:LaplaceFourierKernel} is the subdiffusion dispersion relation \eqref{e:ModelADRExtend}, where $\ell/m=\delta=1-\gamma$, $\ell<m$, and $\ell,\,m\in\mathbb{N}$ without loss coprime. Here $z\mapsto\Psi(z^m)$ is rational with denominator of degree $2m$ thus giving $2m$ poles $z_j\neq0, j=1,\ldots, 2m$, which corresponds to poles $\xi_j, j=1,\dots, 2m$ in the $s$-plane.

\medskip
First, we assume that all poles are simple. 
The formula of $\hat u(q,t)$ is given by
\begin{align*}
\hat{u}(q,t) =\frac{1}{2\pi i}\int_{c-i\infty}^{c+i\infty}\Psi(s)e^{st}\dif s,
\end{align*}
where $c>0$ is chosen to the right of the finitely many singularities. For the ILT of $\Psi$, we take a modified Bromwich contour with branch cut $\BC_{0}^{\theta_1}$, cf.\ Fig.~\ref{f:ILTModelA}, so that 
\begin{align*}
&\hat u(q,t) =\frac{1}{2\pi i}\int_{c-i\infty}^{c+i\infty}\Psi(s)e^{st}\dif s=\frac{1}{2\pi i}\lim_{T\to\infty}\int_{IA}\Psi(s)e^{st}\dif s\\
&\quad =\frac{1}{2\pi i}\Bigg(\lim_{\begin{subarray}{c}
	R\to\infty\\
	\epsilon\to0
	\end{subarray}}\oint_\Gamma-\lim_{R\to\infty}\int_{ABCD}-\lim_{\begin{subarray}{c}
	R\to\infty\\
	\epsilon\to0
	\end{subarray}}\int_{DE+FG}-\lim_{\epsilon\to0}\int_{EF}-\lim_{R\to\infty}\int_{GHI}\Bigg)\Psi(s)e^{st}\dif s.
\end{align*}
\begin{figure}[t]
	\centering
	\includegraphics[width=0.45\linewidth]{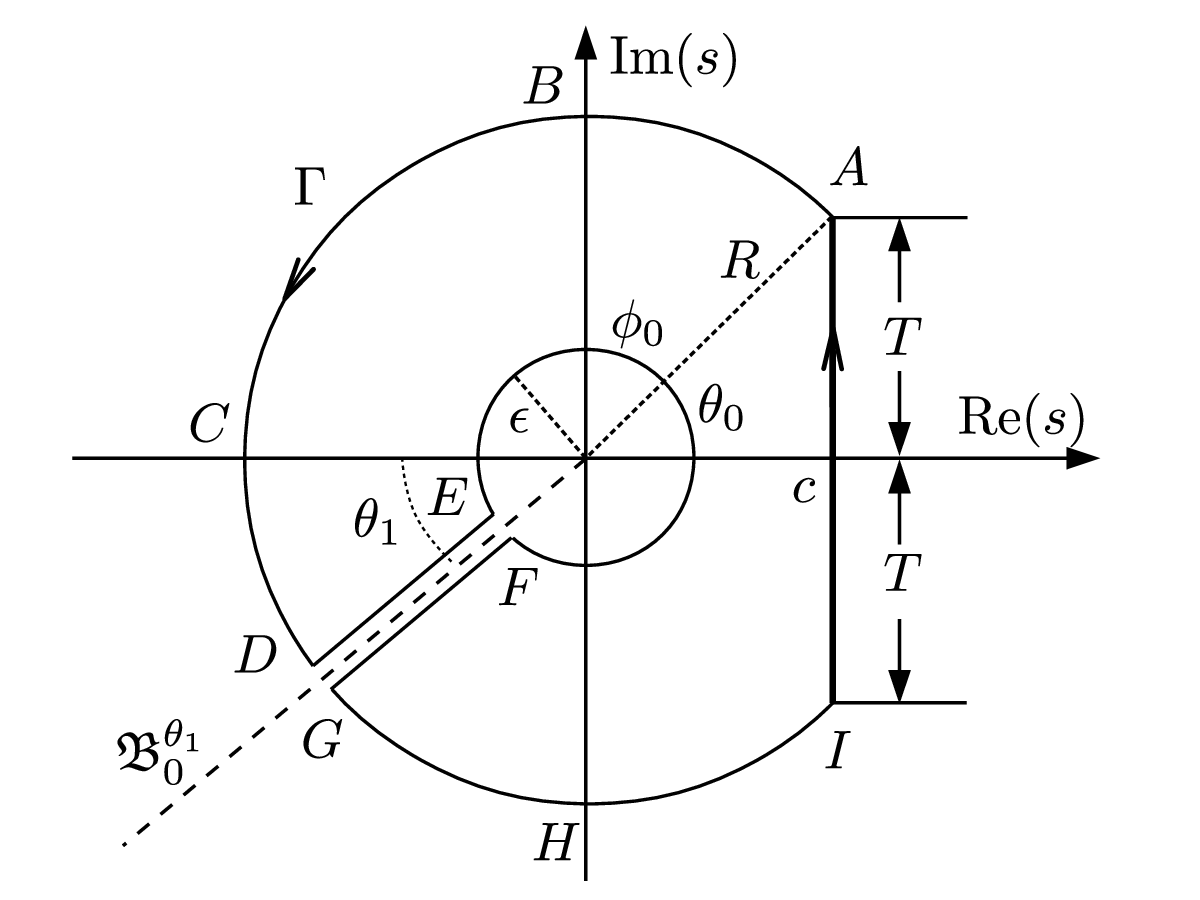}
	\caption{Notation and geometry of the integration contours.}
	\label{f:ILTModelA}
\end{figure}

We split the integration along the closed contour $\Gamma$ as follows.

\medskip
\noindent\underline{Along $EF$}: With $s=\epsilon e^{i\theta}$, $\dif s=i\epsilon e^{i\theta}\dif \theta$, $\lim_{\epsilon\to0}\Psi(\epsilon e^{i\theta}) = \frac{a_2 \hat u_{20} - a_4 \hat u_{10}}{a_1 a_4 - a_2 a_3}$ independent of $\theta$, we have
	\begin{equation*}
	\frac{1}{2\pi i}\lim_{\epsilon\to0}\int_{EF}\Psi(s)e^{st}\dif s=\frac{1}{2\pi i}\lim_{\epsilon\to0}\int_{\pi+\theta_1}^{-\pi+\theta_1}\Psi\left(\epsilon e^{i\theta}\right)e^{\epsilon e^{i\theta}t}\cdot i\epsilon e^{i\theta}\dif \theta=0.
	\end{equation*}
	
\noindent\underline{Along $ABCD$ and $GHI$}: We claim $\left|\Psi\left(Re^{i\theta}\right)\right|\leq M/R$ for a constant $M>0$ independent of $R$. Substituting $s=Re^{i\theta}$ into $\Psi(s)$ yields
\begin{align*} 
\left|\Psi(Re^{i\theta})\right| &= \left|\frac{(Re^{i\theta} + (Re^{i\theta})^{\ell/m}d q^2 -a_4)\hu_{10} + a_2\hu_{20}}{(Re^{i\theta} + (Re^{i\theta})^{\ell/m} q^2- a_{1})(Re^{i\theta} + (Re^{i\theta})^{\ell/m} dq^2- a_{4})- a_{2}a_{3}}\right|\\
&\leq \left|\frac{R}{R^2}\right| \left|\frac{(e^{i\theta} + R^{\ell/m-1}e^{i\theta\ell/m}d q^2 - R^{-1}a_4)\hu_{10} + R^{-1}a_2\hu_{20}}{(e^{i\theta} + R^{\ell/m-1}e^{i\theta\ell/m} q^2- R^{-1}a_{1})(e^{i\theta} + R^{\ell/m-1}e^{i\theta\ell/m} dq^2- a_{4})- R^{-2}a_{2}a_{3}}\right|\\
&\leq R^{-1}M,
\end{align*}
since the second term in the first inequality is continuous in $R$ and has a limit as $R\to\infty$.
It follows that $\lim_{R\to\infty}\int_{ABCD}\Psi(s)e^{st}\dif s = 0$ and $\lim_{R\to\infty}\int_{GHI}\Psi(s)e^{st}\dif s = 0$, cf.\ \cite[Theorem 7-1]{Spiegel1965}.

\medskip
\noindent\underline{Along $DE$ and $FG$}: $s=re^{i(\pi+\theta_1)}=-re^{i\theta_1}$, $s^{1/m}=r^{1/m}e^{i(\pi+\theta_1)/m}$, $\dif s=-e^{i\theta_1}\dif r$, yields
\begin{align*}
\frac{1}{2\pi i}\lim_{\begin{subarray}{c}
	R\to\infty\\
	\epsilon\to0
	\end{subarray}}\int_{DE}\Psi(s)e^{st}\dif s = \frac{e^{i\theta_1}}{2\pi i}\lim_{\begin{subarray}{c}
	R\to\infty\\
	\epsilon\to0
	\end{subarray}}\int_\epsilon^R\Psi\left(re^{i(\pi+\theta_1)}\right)e^{-re^{i\theta_1}t}\dif r=:I_{DE}.
\end{align*}

Analogously, $s=re^{i(-\pi+\theta_1)}=-re^{i\theta_1}$, $s^{1/m}=r^{1/m}e^{i(-\pi+\theta_1)/m}$, $\dif s=-e^{i\theta_1}\dif r$, yields
\begin{align*}
\frac{1}{2\pi i}\lim_{\begin{subarray}{c}
	R\to\infty\\
	\epsilon\to0
	\end{subarray}}\int_{FG}\Psi(s)e^{st}\dif s = -\frac{e^{i\theta_1}}{2\pi i}\lim_{\begin{subarray}{c}
	R\to\infty\\
	\epsilon\to0
	\end{subarray}}\int_\epsilon^R\Psi\left(re^{i(-\pi+\theta_1)}\right)e^{-re^{i\theta_1}t}\dif r=:I_{FG}.
\end{align*}
Combining $I_{DE}$ with $I_{FG}$ gives
\begin{align*}
I_{DE}+I_{FG}  = \frac{e^{i\theta_1}}{2\pi i}\lim_{\begin{subarray}{c}
	R\to\infty\\
	\epsilon\to0
	\end{subarray}}\int_\epsilon^R f(r)e^{-re^{i\theta_1}t}\dif r,
\end{align*}
where $f(r):=\Psi\left(re^{i(\pi+\theta_1)}\right)-\Psi\left(re^{i(-\pi+\theta_1)}\right)$ and $\lim_{r\to0}f(r)=0=\lim_{r\to\infty}f(r)$. Recall $\theta_1=\theta_1(q)$. Since there is no pole on $\BC_0^{\theta_1}$ for any fixed $q$, $f(r)$ is bounded, i.e., $\left|f(r)\right|\leq G_{\gamma}$ for $r\in[0,\infty)$, and a constant $G_{\gamma}>0$ for fixed $\gamma$. Hence,
\begin{align*}
\left|I_{DE}+I_{FG}\right| \leq
\frac{1}{2\pi}\lim_{\begin{subarray}{c}
	R\to\infty\\
	\epsilon\to0
	\end{subarray}}\int_{\epsilon}^{R}G_{\gamma} e^{-rt\cos(\theta_1)}\dif r=\frac{G_{\gamma}}{2\pi t\cos(\theta_1)}
	=\cO(t^{-1}).
\end{align*}
We refine this estimate as follows. For any $r_0>0$ we have $\left|\int_{r_0}^\infty f(r)e^{-re^{i\theta_1}t}\dif r\right|\leq Ct^{-1}e^{-r_0 t\cos\theta_1}$ for a constant $C>0$. Taylor expansion in $r=0$ gives 
$f(r) =  A_1 r^{\ell/m} + \cO(r^{(\ell+1)/m})$ with
\begin{gather*}
A_1:= \left(\frac{q^2 \left((-a_4^2 - a_2 a_3 d) \hu_{10} + (a_2 a_4 + 
   a_1 a_2 d) \hu_{20}\right)}{(a_1 a_4-a_2 a_3)^2}\right) e^{i\theta_1\ell/m}2i\sin(\pi\ell/m).
\end{gather*}
Therefore, 
\begin{equation}\label{e:asyInt}
\int_0^{r_0} f(r)e^{-re^{i\theta_1}t}\dif r = \int_0^{r_0} A_1 r^{\ell/m} e^{-re^{i\theta_1}t}\dif r + 
\cO\left(\int_0^{r_0} r^{(\ell+1)/m} e^{-re^{i\theta_1} t}\dif r\right). 
\end{equation}
For any $\alpha, \beta$ with $\alpha>-1, \Re(\beta)>0$ we have, as $t\to\infty$, that 
\begin{align*}
\int_0^{r_0}r^\alpha e^{-\beta r t}\dif r &= (\beta t)^{-1-\alpha}\int_0^{\beta r_0 t} w^\alpha e^{-w}\dif w = (\beta t)^{-1-\alpha}\left(\int_0^\infty w^\alpha e^{-w} \dif w - \int_{\beta r_0 t}^\infty w^\alpha e^{-w} \dif w\right) \\
& = (\beta t)^{-1-\alpha}\Gamma(1+\alpha) + \cO(t^{-1} e^{-\beta r_0 t}).
\end{align*}
Application to the right-hand side of \eqref{e:asyInt} and combination with the previous gives
\begin{align}\label{e:bcestimate}
I_{DE}+I_{FG} = \frac{A_1 e^{-i\theta_1\ell/m}}{2\pi i}\Gamma(1+\ell/m) t^{-1-\ell/m} + \cO(t^{-1-(\ell+1)/m}).
\end{align}
\begin{remark}\label{r:nonzero-ss}
In case of $\hu_2(q,t)$, the coefficient $A_2$ which plays the same role as $A_1$ in the leading order expansion is given by
(note $a_1 a_4>a_2 a_3$ by assumption)
\[
A_2:= \left(\frac{q^2 \left( (a_3 a_4 + 
   a_1 a_3 d) \hu_{10} -(a_1^2 d + a_2 a_3) \hu_{20} \right)}{(a_1 a_4-a_2 a_3)^2}\right) e^{i\theta_1\ell/m}2i\sin(\pi\ell/m).
\]
Both $A_1$ and $A_2$ vanish if and only if the initial condition lies in the kernel of the matrix
\begin{align*}
\begin{pmatrix}
- a_4^2 - a_2 a_3 d &  a_2 a_4 + a_1 a_2 d\\
 a_3 a_4 + a_1 a_3 d & - a_1^2 d - a_2 a_3
\end{pmatrix},
\end{align*}
whose determinant $(a_1a_4-a_2a_3)^2d$ is positive by assumption so that $(A_1,A_2)\neq0$ for non-trivial initial data, and the stable spectrum leads to algebraic decay as $t^{-1-\ell/m}$.

In addition, for normal diffusion we have $\ell=0$, so that $A_1=A_2=0$, consistent with the exponential decay that occurs in this case. Indeed, all the terms in the power series must have a factor $\sin(\pi k\ell/m)$ for some integers $k$ so that $I_{DE}+I_{FG} = 0$ for $\ell=0$, even though we do not explicitly compute this here. Note that $I_{DE}+I_{FG}=0$ for $\ell=0$ since there is no branch point or branch cut in case of normal diffusion.
\end{remark}

\medskip	
\noindent\underline{Along $\Gamma$:} If $\xi_j\notin\Omega_0$, i.e., there is no pole inside $\Gamma$ for any $\epsilon,R$, then Cauchy's integral theorem gives $\frac{1}{2\pi i}\lim_{R\to\infty,\epsilon\to0}\oint_\Gamma\Psi(s)e^{st}\dif s=0$ and for $\xi_j\in\Omega_0$, the residue theorem gives,
	\begin{align*}
	\frac{1}{2\pi i}\lim_{\begin{subarray}{c}
		R\to\infty\\
		\epsilon\to0
		\end{subarray}}\oint_\Gamma\Psi(s)e^{st}\dif s & = \sum_{\xi_j\in\Omega_0}\Res_{\xi_j}\left(\Psi(s)e^{st}\right) = \sum_{\xi_j\in\Omega_0}\left(\lim_{s\to\xi_j}(s-\xi_j)\Psi(s)\right)e^{\xi_j t}.
	\end{align*}
	
\medskip
All in all we infer the solution in Fourier space can be written as
\begin{align}
\hat{u}(q,t) &= I_\Gamma - (I_{DE}+I_{FG})\nonumber \\
&= \sum_{\xi_j\in\Omega_0}\left(\lim_{s\to\xi_j}(s-\xi_j)\Psi(s)\right)e^{\xi_j t} - \frac{A_1 e^{-i\theta_1\ell/m}}{2\pi i}\Gamma(1+\ell/m) t^{-1-\ell/m} + \cO(t^{-1-(\ell+1)/m}) \nonumber\\
& = C_{\exp,s_1} e^{s_1 t} + C_{\alg}\, t^{\gamma-2} + \cO(t^{-1-(\ell+1)/m}).\label{e:ILTFormulaSimplePole}
\end{align}
Here $s_1:=\mathrm{argmax}\{\Re(\xi_j): \xi_j\in \Omega_0,\, C_{\exp,s_1}\neq0\}$ for the largest exponential rate,
$C_{\exp,s_1} \neq 0$ for almost all initial conditions.
Specifically, the coefficients are given by
\begin{align*}
C_{\exp,s_1} &:= \lim_{s\to s_1} (s-s_1)\Psi(s),\\
C_{\alg} &:= - \frac{q^2 \left((-a_4^2 - a_2 a_3 d) \hu_{10} + (a_2 a_4 + 
   a_1 a_2 d) \hu_{20}\right)}{(a_1 a_4-a_2 a_3)^2} \frac{\sin(\pi\ell/m)}{\pi}\Gamma(1+\ell/m),
\end{align*}
which depends on $q$, $\ell$ and $m$.

\medskip
Second, we consider the case which poles are not all simple. As for simple poles, the integral along each arc tends to $0$ as $\epsilon\to0$ and $R\to\infty$ and \eqref{e:bcestimate} holds.
In contrast to simple poles, for $\xi_j\in\Omega_0$, 
\begin{align*}
I_\Gamma &= \sum_{\xi_j\in\Omega_0}\Res_{\xi_j}\left(\Psi(s)e^{st}\right) = \sum_{\xi_j\in\Omega_0}\left(\frac{1}{(k_j-1)!}\lim_{s\to\xi_j}\left( \frac{\dif}{\dif s} \right)^{k_j-1}\left( (s-\xi_j)^{k_j} \Psi(s)e^{s t}\right)\right)\\
&= \sum_{\xi_j\in\Omega_0}\left(\frac{1}{(k_j-1)!}\left(\lim_{s\to\xi_j} \sum_{k=0}^{k_j-1} \frac{t^k (k_j-1)!}{k!(k_j-1-k)!} \left(\frac{\dif}{\dif s}\right)^{k_j-1-k}\left((s-\xi_j)^{k_j}\Psi(s)\right) \right) e^{\xi_j t} \right)\\
&= C_{\exp,s_0,\rho} t^{\rho-1} e^{s_0 t} + \cO(t^{\rho-2} e^{s_0 t}).
\end{align*}
Hence, if $\Re(s_0)\geq 0$, then $\hu$ behaves exponentially as $t^{\rho-1}e^{s_0 t}$, where
\[s_0 := \argmax\{\Re(\xi_j):\xi_j\in\Omega_0,\,C_{\exp,s_0,\rho}\neq0\}\] has the largest exponential rate and
$\rho$ is the multiplicity of $s = s_0$. In particular,
\begin{align}\label{e:multicoeff-ss}
C_{\exp,s_0,\rho} := \frac{1}{(\rho-1)!} \lim_{s\to s_0}(s-s_0)^\rho \Psi(s).
\end{align} 
If $\Re(s_0)<0$, then $\hu$ decays as $t^{\gamma-2}$.

\section{Proof of Lemma \ref{l:AsymptoticSolution}}\label{s:ProofAsymptotic}

We recall \eqref{e:FracDR}, i.e., $D_\ss(s,q^2)=0$. Rescaling $q=\kappa/\varepsilon$ and $s=x/\varepsilon^r$, where $0<\varepsilon\ll1$ and $r>0$, substitution into \eqref{e:FracDR} and balancing the powers of $\varepsilon$ gives $r=2/(1-\delta)$. The dispersion relation becomes
\begin{align}\label{e:RescaledFracDR}
(x + x^\delta \kappa^2 - \varepsilon^{2/(1-\delta)} a_{1}) (x + x^\delta d \kappa^2 - \varepsilon^{2/(1-\delta)} a_{4}) - \varepsilon^{\frac{4}{1-\delta}} a_{2} a_{3} = 0.
\end{align}
and we seek solutions of the form 
\begin{equation}\label{e:SolExpansion}
x = x_0 + \varepsilon^\alpha x_1 + \cO(\varepsilon^\beta),\quad \beta>\alpha>0.
\end{equation}
Substitution into \eqref{e:RescaledFracDR} yields the expansion $(x_0+\varepsilon^\alpha x_1+\cO(\varepsilon^\beta))^\delta=x_0^\delta+\delta x_0^{\delta-1}(\varepsilon^\alpha x_1+\cO(\varepsilon^\beta))+\cO(\varepsilon^{2\alpha})$ for $x_0\neq0$ so that ordering powers of $\varepsilon$ gives 
\begin{align*}
& (x_0+x_0^\delta \kappa^2)(x_0+x_0^\delta d\kappa^2)\\
& +\varepsilon^\alpha\left[(x_0+x_0^\delta \kappa^2)\left(x_1+\delta x_0^{\delta-1}x_1d\kappa^2\right)+(x_0+x_0^\delta d\kappa^2)\left(x_1+\delta x_0^{\delta-1}x_1\kappa^2\right)\right]\\
& -\varepsilon^{\frac{2}{1-\delta}}\left[a_{1}(x_0+x_0^\delta d\kappa^2)+a_{4}(x_0+x_0^\delta \kappa^2)\right]\\
& +\varepsilon^{2\alpha}\left(x_1+\delta x_0^{\delta-1}x_1\kappa^2\right)\left(x_1+\delta x_0^{\delta-1}x_1d\kappa^2\right)\\
& -\varepsilon^{\alpha+\frac{2}{1-\delta}}\left[a_{1}\left(x_1+\delta x_0^{\delta-1}x_1d\kappa^2\right)+a_{4}\left(x_1+\delta x_0^{\delta-1}x_1\kappa^2\right)\right]\\
& +\varepsilon^{\frac{4}{1-\delta}}\left(a_{1}a_{4}-a_{2}a_{3}\right)\\
& +\cO(\varepsilon^\beta)(2x_0+x_0^\delta \kappa^2+x_0^\delta d\kappa^2)+\cO(\varepsilon^{\alpha+\beta} + \varepsilon^{3\alpha}) = 0.
\end{align*}
Let us compare coefficients by orders of $\varepsilon$:

\smallskip
\underline{$\cO (1)$}: $(x_0+x_0^\delta \kappa^2)(x_0+x_0^\delta d\kappa^2)=0$, solutions are $x_{01}=(-\kappa^2)^{1/(1-\delta)}$ or $x_{02}=(-d\kappa^2)^{1/(1-\delta)}$.

\smallskip
\underline{$\cO (\varepsilon^{\alpha})$}: For $x_{01}=(-\kappa^2)^{1/(1-\delta)}$, balancing the coefficients of order $\cO (\varepsilon^{2/(1-\delta)})$ and $\cO (\varepsilon^\alpha)$ yields $\alpha = 2/(1-\delta)$, and substituting $x_{01}$ into the term of order $\cO (\varepsilon^\alpha)$ gives $x_{11}=a_{1}/(1-\delta)$. Combining the scaling $s=x/\varepsilon^r$ with the expansion \eqref{e:SolExpansion} gives the approximation
\begin{align*}
s_{\infty 1} = (-q^2)^{1/(1-\delta)}+\frac{a_{1}}{1-\delta}+\cO(q^{\alpha-\beta}),
\end{align*}
Since $\arg((-q^2)^{1/(1-\delta)}) = \pi/(1-\delta)$, the real part $\Re((-q^2)^{1/(1-\delta)})<0$ if $\pi/(1-\delta) \in (\pi/2, \pi+\theta_1)$, i.e., $\delta \in (0, \theta_1/(\pi+\theta_1))$. Hence, for any $\delta \in (0, \theta_1/(\pi+\theta_1))$, there exists $Q_1>0$ such that for $q>Q_1$ we have
\begin{align}\label{e:MinusInfinitySpectrumSupremum1}
\Re(s_{\infty 1}) = \Re\left(\left(-q^2\right)^{\frac{1}{1-\delta}}+\frac{a_{1}}{1-\delta}+\cO(q^{\alpha-\beta})\right) < \Re\left(\left(-Q_1^2\right)^{\frac{1}{1-\delta}}+\frac{a_{1}}{1-\delta}+1\right) < 0.
\end{align}
Similarly, for $x_{02}=(-d\kappa^2)^{1/(1-\delta)}$, the solution is given by
\begin{equation*}
s_{\infty 2} = \left(-dq^2\right)^{\frac{1}{1-\delta}}+\frac{a_{4}}{1-\delta}+\cO(q^{\alpha-\beta}).
\end{equation*}
For any $\delta \in (0, \theta_1/(\pi+\theta_1))$, there exists $Q_2>0$ such that for $q>Q_2$ we have
\begin{equation}\label{e:MinusInfinitySpectrumSupremum2}
\Re(s_{\infty 2}) = \Re\left(\left(-dq^2\right)^{\frac{1}{1-\delta}}+\frac{a_{4}}{1-\delta}+\cO(q^{\alpha-\beta})\right)  <\Re\left(\left(-dQ_2^2\right)^{\frac{1}{1-\delta}}+\frac{a_{4}}{1-\delta}+1\right) < 0.
\end{equation}
With $Q=\max\left\{Q_1,Q_2\right\}$ \eqref{e:MinusInfinitySpectrumSupremum1} and \eqref{e:MinusInfinitySpectrumSupremum2} both hold if $q>Q$. Finally, $\delta \in (0,\theta_1/(\pi+\theta_1))$ is the necessary condition for $s_{\infty 1}$ and $s_{\infty 2}$ to lie in $\Omega_0$, cf.\ Lemma \ref{l:InfinitySolutionExistence}.

\medskip	
In the case $x_0=0$, however, we transform the problem via $s^\delta=z \Rightarrow s=z^{1/\delta}$, and \eqref{e:FracDR} reads
\begin{equation*}
D_\ss(z^{1/\delta},q^2)=(z^{1/\delta}+z q^2- a_{1}) (z^{1/\delta}+z dq^2- a_{4})- a_{2}a_{3}=0.
\end{equation*}
Rescaling $q=\kappa/\varepsilon$, $z=y/\varepsilon^\mu$, balancing the power of $\varepsilon$ gives $\mu= 2\delta/(1-\delta)$, which leads to
\begin{equation}\label{e:RescaledDispersionRelationZ}
(y^{1/\delta}+y \kappa^2-\varepsilon^{2/(1-\delta)} a_{1}) (y^{1/\delta}+y d\kappa^2-\varepsilon^{2/(1-\delta)} a_{4})-\varepsilon^{4/(1-\delta)} a_{2}a_{3}=0
\end{equation}
and we seek solutions of the form
\begin{equation}\label{e:AsymptoticFormZ}
y=y_0+\varepsilon^\alpha y_1+o(\varepsilon^\alpha).
\end{equation}
Substituting into \eqref{e:RescaledDispersionRelationZ}, and using that the first derivative of $y^{\frac{1}{\delta}}$ exists for $y=0$, we have $(y_0+\varepsilon^\alpha y_1+o(\varepsilon^\alpha))^{1/\delta}=y_0^{1/\delta}+\frac{1}{\delta}y_0^{1/\delta-1}(\varepsilon^\alpha y_1+o(\varepsilon^\alpha))+o(\varepsilon^\alpha)$, which yields
\begin{align*}
& \left(y_0^{\frac{1}{\delta}}+y_0\kappa^2\right)\left(y_0^{\frac{1}{\delta}}+y_0d\kappa^2\right)\\
& +\varepsilon^\alpha\left[\left(y_0^{\frac{1}{\delta}}+y_0\kappa^2\right)\left(\frac{1}{\delta}y_0^{\frac{1}{\delta}-1}y_1+y_1d\kappa^2\right)+\left(y_0^{\frac{1}{\delta}}+y_0d\kappa^2\right)\left(\frac{1}{\delta}y_0^{\frac{1}{\delta}-1}y_1+y_1\kappa^2\right)\right]\\
& -\varepsilon^{\frac{2}{1-\delta}}\left[ a_{1}\left(y_0^\frac{1}{\delta}+y_0d\kappa^2\right)+ a_{4}\left(y_0^\frac{1}{\delta}+y_0\kappa^2\right)\right]\\
& +\varepsilon^{2\alpha}\left(\frac{1}{\delta}y_0^{\frac{1}{\delta}-1}y_1+y_1\kappa^2\right)\left(\frac{1}{\delta}y_0^{\frac{1}{\delta}-1}y_1+y_1d\kappa^2\right)\\
& -\varepsilon^{\alpha+\frac{2}{1-\delta}}\left[ a_{1}\left(\frac{1}{\delta}y_0^{\frac{1}{\delta}-1}y_1+y_1d\kappa^2\right)+ a_{4}\left(\frac{1}{\delta}y_0^{\frac{1}{\delta}-1}y_1+y_1\kappa^2\right)\right]\\
& +\varepsilon^{\frac{4}{1-\delta}}\left(a_{1}a_{4}-a_{2}a_{3}\right)\\
& +o(\varepsilon^\alpha)\left(2y_0^\frac{1}{\delta}+y_0\kappa^2+y_0d\kappa^2\right)+o\left(\varepsilon^{2\alpha}\right)+o\left(\varepsilon^{\alpha+\frac{2}{1-\delta}}\right)=0.
\end{align*}
In the present case of $x_0 =0$ we have $y_0=0$, so the coefficients of order $\cO (1)$, $\cO (\varepsilon^\alpha)$ and $\cO (\varepsilon^{2/(1-\delta)})$ vanish. 

\smallskip
\underline{$\cO (\varepsilon^{2\alpha})$}: Balancing the coefficients of order $\varepsilon^{2\alpha}$, $\varepsilon^{\alpha+2/(1-\delta)}$ and $\varepsilon^{4/(1-\delta)}$ yields $\alpha=2/(1-\delta)$, and combining the coefficients gives
\begin{equation*}
d\kappa^4y_1^2-( a_{1}d\kappa^2+ a_{4}\kappa^2)y_1+a_{1}a_{4}-a_{2}a_{3}=0,
\end{equation*}
which has some roots $y_{1\pm}$.
Combining $s=z^{1/\delta}$, $z=y/\varepsilon^\mu$ and \eqref{e:AsymptoticFormZ} gives the solution
\begin{equation*}
s_{0\pm}=z^{1/\delta}=(y/\varepsilon^\mu)^{1/\delta} = \varepsilon^{2/\delta}y_{1\pm}^{1/\delta} + o(\varepsilon^{2/\delta}).
\end{equation*}
Hence, $\lim_{\varepsilon\to0}s_{0\pm}=0$, which implies that $\lim_{|q|\to\infty}s_{0\pm}(q)=0$.

This completes proof of Lemma \ref{l:AsymptoticSolution}.

\section{Proof of Proposition \ref{p:ExistencePseudo}}\label{s:DiscussionLemmas}

We prove Proposition \ref{p:ExistencePseudo} by the following series of lemmas. In summary, Lemma \ref{l:RealPartNegativeLessThan1/3} and Lemma \ref{l:RealPartNegativeLargerThan1/3} give the stability threshold $d_\delta^\infty$; Lemma \ref{l:RealPartNegativeLessThan1/3}, Lemma \ref{l:DiffusionRatioLessThanNecessaryCondition} and Lemma \ref{l:ComplexSpectrumForSmalld} give the existence threshold $\widetilde{d}_{\delta}^\infty$.

\medskip
We recall  $\Delta:=4d(a_{1}a_{4}-a_{2}a_{3})-( a_{1}d+ a_{4})^2$, $b:= a_{1}d+ a_{4}$, $y_{1\pm}=(b\pm i\sqrt{\Delta})/(2d)=:\rho e^{\pm i\theta}$. Note $P_{\min}>0$ implies $\Delta>0$. 

We make a case distinction in terms of the sign of $b$.

\medskip
\underline{Case $b>0$}, i.e., $d>-\frac{a_{4}}{a_{1}}$.
\begin{lemma}\label{l:RealPartNegativeLessThan1/3}
	For any $\delta \in (0,\frac{\pi}{2(\pi+\theta_1)})$, there exists a $Q>0$ such that for any $|q|>Q$ we have $s_{0+}(q) \in \Omega_0^-$ if $\widetilde{d}_{\delta+}^\infty < d < d_{\delta}^\infty$, whereas $s_{0\pm}(q) \notin \Omega_0$ if $-\frac{a_{4}}{a_{1}} < d < \widetilde{d}_{\delta+}^\infty$, with $\widetilde{d}_{\delta+}^\infty$ the larger root of \eqref{e:ExistenceEqn}.
\end{lemma}

\begin{proof}
	$\Re\left(y_{1+}^{1/\delta}\right)<0$ if $\arg \left(y_{1+}^{1/\delta}\right)=\frac{\theta}{\delta}\in\left(\frac{\pi}{2},\pi+\theta_1\right)$. We remark that we do not discuss $y_{1-}^{1/\delta}$ here, since $y_{1+}^{1/\delta}$ and $y_{1-}^{1/\delta}$ are complex conjugate, $\arg \left(y_{1-}^{1/\delta}\right)=-\arg \left(y_{1+}^{1/\delta}\right) \in (-\pi-\theta_1,-\pi/2)$, which means if $y_{1-}^{1/\delta} \in \Omega_0$ then $y_{1+}^{1/\delta} \in \Omega_0$, but not vice versa, and if $y_{1+}^{1/\delta} \notin \Omega_0$, then $y_{1-}^{1/\delta} \notin \Omega_0$. 
	
	\smallskip
	We consider $\left(\frac{\pi\delta}{2},(\pi+\theta_1)\delta\right)\subset\left(0,\frac{\pi}{2}\right)$, i.e., $\delta<\frac{\pi}{2(\pi + \theta_1)}$. The condition $\frac{\theta}{\delta}\in(\frac{\pi}{2},\pi+\theta_1)$ leads to $\arctan(\sqrt{\Delta} / b) \in (\pi\delta/2,(\pi+\theta_1)\delta)$ which yields
	\[
	b^2\tan^2(\pi\delta/2) < \Delta < b^2\tan^2((\pi+\theta_1)\delta).
	\]
	Using $c:=\cos^2(\pi\delta/2)$, $\tilde c:=\cos^2((\pi+\theta_1)\delta)$, these two inequalities can be written as 
	\begin{gather}
	H(d):=a_{1}^2d^2+(4c(a_{2}a_{3}-a_{1}a_{4}) + 2a_{1}a_{4})d+a_{4}^2<0,\label{e:RealPartNegativec}\\
	\widetilde{H}(d):=a_{1}^2d^2+(4\tilde c(a_{2}a_{3}-a_{1}a_{4}) + 2a_{1}a_{4})d+a_{4}^2>0.\label{e:RealPartNegativectilde}
	\end{gather}
	With $d_-$ from the proof of Lemma~\ref{l:RealPartPositive}, the solutions of \eqref{e:RealPartNegativec} are $d\in(d_-,d_\delta^\infty)$ and the solutions of \eqref{e:RealPartNegativectilde} are $d>\widetilde{d}_{\delta+}^\infty$; here we omit $d<\widetilde{d}_{\delta-}^\infty$ as in the proof of Lemma~\ref{l:RealPartPositive}. Using $\widetilde{d}_{\delta+}^\infty>-\frac{a_{4}}{a_{1}}$, combining these solutions gives $\widetilde{d}_{\delta+}^\infty<d<d_\delta^\infty$. 
	
	\smallskip
	Whereas, $y_{1+}\notin\Sigma_0$ if $\theta \notin (\pi\delta/2,(\pi+\theta_1)\delta)$ so we infer $\theta > (\pi+\theta_1)\delta$. Since $\arg (y_{1+}) = \arctan(\sqrt{\Delta}/b)$, we have $\arctan(\sqrt{\Delta}/b) > (\pi+\theta_1)\delta$, which implies $\widetilde{d}_{\delta-}^\infty<d<\widetilde{d}_{\delta+}^\infty \Rightarrow -\frac{a_{4}}{a_{1}}<d<\widetilde{d}_{\delta+}^\infty$.
\end{proof}

\begin{lemma}\label{l:RealPartNegativeLargerThan1/3}
	For any $\delta \in[\frac{\pi}{2(\pi+\theta_1)},1)$, there exists a $Q>0$ such that for any $|q|>Q$ we have $s_{0+}(q) \in \Omega_0^-$ if $-\frac{a_{4}}{a_{1}}<d<d_{\delta}^\infty$.
\end{lemma}

\begin{proof}
Since $b>0$, we have $\theta = \arctan(\sqrt{\Delta}/b) < \pi/2$. In combination with the assumption $\theta \in (\pi\delta/2,(\pi+\theta_1)\delta)$, we get $\theta \in (\pi\delta/2,\pi/2) \Rightarrow \arctan(\sqrt{\Delta}/b) \in (\pi\delta/2,\pi/2)$ which yields
	\[
	\Delta^2 > b^2 \tan^2(\pi\delta/2).
	\]
	Hence
	\begin{equation*}
	H(d)=a_{1}^2d^2+(2a_{1}a_{4}+4c(a_{2}a_{3}-a_{1}a_{4}))d+a_{4}^2<0,
	\end{equation*}
	The solution is $d_-<d<d_{\delta}^\infty$. Note $d_-<-\frac{a_{4}}{a_{1}}$, thus we have $-\frac{a_{4}}{a_{1}}<d<d_\delta^\infty$.
\end{proof}

\medskip
\underline{Case $b<0$}, i.e., $d<-\frac{a_{4}}{a_{1}}$

\begin{lemma}\label{l:DiffusionRatioLessThanNecessaryCondition}
In each of the following, for any $\delta$ in the given interval there exists a $Q>0$ such that for any $|q|>Q$ the given statement holds.
\begin{itemize}
\item[(1)] For $[\frac{\pi}{\pi+\theta_1},1)$ we have $s_{0+}(q)\in\Omega_0^-$ if $d<-\frac{a_{4}}{a_{1}}$;
\item[(2)]	For $(0,\frac{\pi}{2(\pi+\theta_1)}]$ we have $s_{0\pm}(q)\notin\Omega_0$ if $d<-\frac{a_{4}}{a_{1}}$;
\item[(3)]	For $\left(\frac{\pi}{2(\pi+\theta_1)},\frac{\pi}{\pi+\theta_1}\right)$ we have $s_{0\pm}(q)\notin\Omega_0$ if $d<\widetilde{d}_{\delta-}^\infty$, and $s_{0+}(q)\in\Omega_0^-$ if $\widetilde{d}_{\delta-}^\infty < d < -\frac{a_{4}}{a_{1}}$, where $\widetilde{d}_{\delta-}^\infty$ is the smaller root of \eqref{e:ExistenceEqn}.
\end{itemize}
\end{lemma}

\begin{proof}
	Since $b<0$, if $y_{1-} \in \Sigma_0$, then $y_{1+} \in \Sigma_0$, but not vice versa, and if $y_{1+} \notin \Sigma_0$, then $y_{1-} \notin \Sigma_0$. Hence we only consider $y_{1+}$. It is straightforward that $\arg (y_{1+})\in\left(\frac{\pi}{2},\pi\right)$, so $\arg \left(y_{1+}^{1/\delta}\right)\in\left(\frac{\pi}{2\delta},\frac{\pi}{\delta}\right)$. 
	
	\smallskip
	First, $y_{1+}^{1/\delta} \in \Omega_0$ if $\left(\frac{\pi}{2\delta},\frac{\pi}{\delta}\right) \subseteq (-\pi+\theta_1,\pi+\theta_1) \Rightarrow \delta \geq \frac{\pi}{\pi+\theta_1}$. From $\delta \in [\frac{\pi}{\pi+\theta_1},1)$ we have $\arg\left(y_{1+}^{1/\delta}\right)\in\left(\frac{\pi}{2\delta},\frac{\pi}{\delta}\right) \subset (\frac{\pi}{2},\pi+\theta_1)$, which means $\Re\left(y_{1+}^{1/\delta}\right)<0$ if $\delta \in [\frac{\pi}{\pi+\theta_1},1)$.
	
	\smallskip
	Second, $y_{1+}^{1/\delta}\notin\Omega_0$ if $\frac{\pi}{\delta} \leq -\pi+\theta_1$ or $\frac{\pi}{2\delta} \geq \pi+\theta_1$, so we have $\delta \leq \frac{\pi}{2(\pi+\theta_1)}$, which means $y_{1+}^{1/\delta}\notin\Omega_0$ if $\delta \in (0,\frac{\pi}{2(\pi+\theta_1)}]$.
	
	\smallskip
	Next, we consider the case $\delta \in (\frac{\pi}{2(\pi+\theta_1)},\frac{\pi}{\pi+\theta_1})$: $y_{1+} \notin \Sigma_0$ if $\arg(y_{1+}) \notin ((-\pi+\theta_1)\delta,(\pi+\theta_1)\delta)$. Since $b<0$, $\Re(y_{1+})<0$ we infer $\arg (y_{1+}) \in ((\pi+\theta_1)\delta,\pi)$. From the assumption $\delta \in (\frac{\pi}{2(\pi+\theta_1)},\frac{\pi}{(\pi+\theta_1)})$ we have $(\pi+\theta_1)\delta \in \left(\frac{\pi}{2},\pi\right)$, so $\arg (y_{1+})\in\left(\frac{\pi}{2},\pi\right)$. Now $\arg (y_{1+}) = \arctan(\sqrt{\Delta}/b) + \pi > (\pi+\theta_1)\delta \Rightarrow \sqrt{\Delta}/b > \tan((\pi+\theta_1)\delta-\pi) = \tan((\pi+\theta_1)\delta) \Rightarrow \Delta < b^2\tan^2((\pi+\theta_1)\delta)$. The solution is $d<\widetilde{d}_{\delta-}^\infty$ or $d>\widetilde{d}_{\delta+}^\infty$ (omit, because $\widetilde{d}_{\delta+}^\infty>-\frac{a_{4}}{a_{1}}$). That means $y_{1+} \notin \Sigma_0$ if $d < \widetilde{d}_{\delta-}^\infty$. Whereas $y_{1+} \in \Sigma$ if $\arg(y_{1+}) \in (\frac{\pi}{2},(\pi+\theta_1)\delta) \Rightarrow \Delta > b^2\tan^2((\pi+\theta_1)\delta)$, the solutions of the latter are $\widetilde{d}_{\delta-}^\infty<d<\widetilde{d}_{\delta+}^\infty \Rightarrow \widetilde{d}_{\delta-}^\infty<d<-\frac{a_{4}}{a_{1}}$. 
\end{proof}

\begin{remark}\label{r:CoincideComplexReal}
	When $\delta=\frac{\pi}{\pi+\theta_1}$, $\tilde{c}=\cos^2(\pi)=1$, which leads to $\widetilde{H}(d)=F(d)$. Therefore the solutions to $\widetilde{H}(d)=0$ and $F(d)=0$ are equivalent, i.e., $\widetilde{d}_{\delta-}^\infty=d_{fr-}$.  Hence, the first statement in Lemma \ref{l:DiffusionRatioLessThanNecessaryCondition} implies Case 1b of Section~\ref{s:largewav}.
\end{remark}

\medskip
\underline{Case $b=0$}, i.e., $d=-\frac{a_{4}}{a_{1}}$, we have the following lemma.

\begin{lemma}\label{l:ComplexSpectrumForSmalld}
	For any $\delta \in (\frac{\pi}{2(\pi+\theta_1)},1)$, there exists a $Q>0$ such that for any $|q|>Q$ we have $s_{0+}(q)\in\Omega_0^-$ if $d=-\frac{a_{4}}{a_{1}}$, and for any $\delta\in(0,\frac{\pi}{2(\pi+\theta_1)}]$, $s_{0\pm}(q) \notin \Omega_0$ if $d=-\frac{a_{4}}{a_{1}}$.
\end{lemma}

\begin{proof}
	Assume that $d=-\frac{a_{4}}{a_{1}}$, i.e., $b=0$, then $\Re(y_{1\pm})=0$ and $\arg (y_{1\pm}) = \pm\pi/2$ so that again it suffices to consider $y_{1+}$. $y_{1+}^{1/\delta} \in \Omega_0$ if $\frac{\pi}{2\delta} \in (-\pi+\theta_1,\pi+\theta_1) \Rightarrow \delta > \frac{\pi}{2(\pi+\theta_1)}$, i.e., $\delta \in (\frac{\pi}{2(\pi+\theta_1)},1)$, which leads to $\arg\left(y_{1+}^{1/\delta}\right) = \frac{\pi}{2\delta} \in \left(\frac{\pi}{2},\pi+\theta_1\right)$ and implies $\Re\left(y_{1+}^{1/\delta}\right)<0$.
	
	\smallskip
	When $\delta \in (0,\frac{\pi}{2(\pi+\theta_1)}]$, $\max((\pi+\theta_1)\delta) = \pi/2$. Hence the argument of $\Sigma_0$ is less than $\pi/2$, whereas $\arg (y_{1+}) = \pi/2$, so $y_{1+}^{1/\delta}\notin\Omega_0$.
\end{proof}

\section{Proof of Theorem \ref{t:ILT-ca}}\label{s:InverseLaplaceNonzeroBranchPoint}

In the following, we omit some parts that are very similar to parts of the proof of Theorem \ref{t:ILT-ss}, see Appendix \ref{s:InverseLaplaceZeroBranchPoint}.

\medskip
We consider the inverse Laplace transform of $\fL\hat w_{j}$, $j=1,2$ for fixed $q$.
We first discuss $\fL\hat w_1(q,s)$ and the case when all roots of $D_{\ca2}(s,q^2)$ are simple. The solution $\fL\hat w_1$ can be written as
\begin{align*}
\Psi(s) :=&\ \frac{\left(s-\mu_2+d_4 q^2 (s-\mu_2)^{\frac{n}{m}}\right)\hat w_1(q,0) - d_2 q^2 (s-\mu_2)^{\frac{n}{m}}\hat w_2(q,0)}{\left(s-\mu_1+d_{1}q^2(s-\mu_1)^{\frac{n}{m}}\right)\left(s-\mu_2+d_{4}q^2(s-\mu_2)^{\frac{n}{m}}\right)-(s-\mu_1)^{\frac{n}{m}}(s-\mu_2)^{\frac{n}{m}}d_{2}d_{3}q^4}\\
=&\ \frac{\left((s-\mu_2)^{\frac{m-n}{m}} + d_4 q^2\right)\hat w_{10} - d_2 q^2 \hat w_{20}}{(s-\mu_1)^{\frac n m}\left[\left((s-\mu_1)^{\frac{m-n}{m}} + d_{1}q^2\right) \left((s-\mu_2)^{\frac{m-n}{m}} + d_{4}q^2\right) - d_2 d_3 q^4\right]},
\end{align*}
where the initial conditions are $\hat w_1(q,0) = \hat w_{10}(q)$ and $\hat w_2(q,0) = \hat w_{20}(q)$. We recall $\gamma = 1-n/m$. The inverse Laplace transform gives 
\begin{align*}
&\hat{w}_1(q,t)  = \frac{1}{2\pi i}\int_{c-i\infty}^{c+i\infty}\Psi(s)e^{st}\dif s.
\end{align*}

\medskip
Let $s-\mu_1 = R e^{i\theta}$ and denote $\mu_\Delta := \mu_1-\mu_2$. We claim $|\Psi(s)| < M/R$ for $R>0$ with suitable constant $M>0$. Indeed,
\begin{align*}
|\Psi(s)| &= \left| \frac{\left((Re^{i\theta} + \mu_\Delta)^{\frac{m-n}{m}} + d_4 q^2\right)\hat w_{10} - d_2 q^2 \hat w_{20}}{(Re^{i\theta})^{\frac n m}\left[\left((Re^{i\theta})^{\frac{m-n}{m}} + d_{1}q^2\right) \left((Re^{i\theta} + \mu_\Delta)^{\frac{m-n}{m}} + d_{4}q^2\right) - d_2 d_3 q^4\right]} \right|\\
&\leq \left| \frac{R^{\frac{m-n}{m}}}{R^{1+\frac{m-n}{m}}} \right|  \left| \frac{\left((e^{i\theta} + \mu_\Delta R^{-1})^{\frac{m-n}{m}} + R^{-\frac{m-n}{m}}d_4 q^2\right)\hat w_{10} - R^{-\frac{m-n}{m}}d_2 q^2 \hat w_{20}} {e^{i\theta\frac n m} \left(e^{i\theta\frac{m-n}{m}} + R^{-\frac{m-n}{m}}d_{1}q^2\right) \left((e^{i\theta} + \mu_\Delta R^{-1})^{\frac{m-n}{m}} + R^{-\frac{m-n}{m}}d_{4}q^2\right) - R^{-\frac{2(m-n)}{m}}d_2 d_3 q^4} \right|\\
&\leq R^{-1} M,
\end{align*}
where the last inequality is satisfied since the second term in the first inequality is continuous in $R$ and has a limit as $R\to\infty$.

\medskip
\underline{Case \textbf{cc}}: We calculate $\hat w(q,t)$ and choose the branch cuts $\BC_{\mu_1}^{-\theta_1}$, $\BC_{\mu_2}^{\theta_2}$ with corresponding principal branch $\Omega_\ca^\cc$ from Section~\ref{s:turingbifurcation}, cf.\ Fig.~\ref{f:ILTModelB-cc} for details. Then $\hat w_1$ can be written as
\begin{align*}
\hat w_1(q,t)  &= \frac{1}{2\pi i}\lim_{R\to\infty}\int_{JA}\Psi(s)e^{st}\dif s\\
& = \frac{1}{2\pi i}\left(\lim_{\begin{subarray}{c}
	R\to\infty\\
	\epsilon\to0
	\end{subarray}}\oint_{\Gamma}-\lim_{R\to\infty}\int_{AB+EF+IJ}-\lim_{\begin{subarray}{c}
	R\to\infty\\
	\epsilon\to0
	\end{subarray}}\int_{BC+DE}-\lim_{\epsilon\to0}\int_{CD+GH}-\lim_{\begin{subarray}{c}
	R\to\infty\\
	\epsilon\to0
	\end{subarray}}\int_{FG+HI}\right)\Psi(s)e^{st}\dif s.
\end{align*}
Since $\left|\Psi(s)\right|< M/R$, the integrals along $AB$, $EF$, $IJ$ vanish, cf.\ \cite[Theorem 7-1]{Spiegel1965}. Similar to the proof of Theorem \ref{t:ILT-ss}, the integral along $CD$, $GH$ vanish as well.
\begin{figure}[t]
	\centering
	\includegraphics[width=0.45\linewidth]{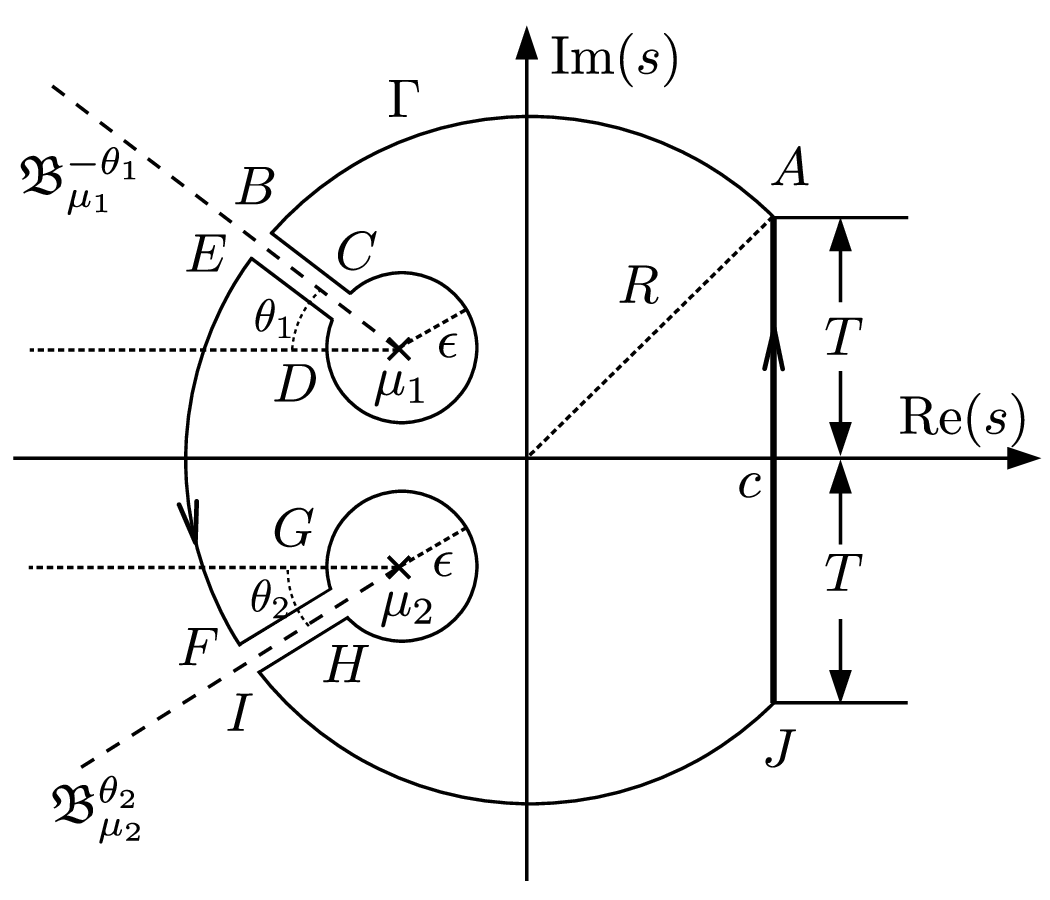}
	\caption{Notation and geometry of the integration contours for case \textbf{cc}.}
	\label{f:ILTModelB-cc}
\end{figure}

\medskip
Along $BC$: $s-\mu_1=re^{i(\pi-\theta_1)}$, $\dif s=-e^{-i\theta_1}\dif r$, along $DE$: $s-\mu_1=re^{i(-\pi-\theta_1)}$, $\dif s=-e^{-i\theta_1}\dif r$ so that
\begin{align*}
I_{BC} &:= \frac{1}{2\pi i}\lim_{\begin{subarray}{c}
	R\to\infty\\
	\epsilon\to0
	\end{subarray}}\int_{BC}\Psi(s)e^{st}\dif s = \frac{1}{2\pi i}\int_0^\infty\Psi\left(\mu_1+re^{i(\pi-\theta_1)}\right)e^{(\mu_1-re^{-i\theta_1})t-i\theta_1}\dif r,\\
I_{DE} &:= \frac{1}{2\pi i}\lim_{\begin{subarray}{c}
	R\to\infty\\
	\epsilon\to0
	\end{subarray}}\int_{DE}\Psi(s)e^{st}\dif s = -\frac{1}{2\pi i}\int_0^\infty\Psi\left(\mu_1+re^{i(-\pi-\theta_1)}\right)e^{(\mu_1-re^{-i\theta_1})t-i\theta_1}\dif r.
\end{align*}
With $f_1(r):=\Psi\left(\mu_1+re^{i(\pi-\theta_1)}\right)-\Psi\left(\mu_1+re^{i(-\pi-\theta_1)}\right)$, combining the above two integrals gives 
\begin{align*}
I_{BC} + I_{DE} = \frac{1}{2\pi i} e^{\mu_1 t-i\theta_1} \int_0^\infty f_1(r) e^{-r e^{-i\theta_1}t} \dif r.
\end{align*}
Since there is no pole on the branch cut and $\lim_{r\to\infty}f_1(r) = 0$, for any $r_0>0$ we have $|f_1(r)|$ is uniformly bounded. Hence, for any $r_0>0$, $\left| \int_{r_0}^\infty f_1(r) e^{-r e^{-i\theta_1}t} \dif r \right| \leq C t^{-1} e^{-r_0 t \cos\theta_1}$ for a constant $C>0$. Denote $\tilde f_1(r) := f_1(r)\cdot r^{n/m}$, then Taylor expansion near $r=0$ for $q\neq0$ gives $\tilde f_1(r) = A_1 + \cO(r^{1/m})$ with
\begin{align*}
A_1:= \tilde f_1(0) =  -\frac{\left((\mu_\Delta)^{1-n/m}+ d_4q^2\right) \hat w_{10} - d_2 q^2\hat w_{20}}{d_1 q^2((\mu_\Delta)^{1-n/m} + d_4 q^2) - d_2 d_3 q^4} e^{i\theta_1 n/m} 2 i \sin(\pi n /m).
\end{align*}
Therefore, for fixed $q\neq 0$, we obtain
\begin{align*}
\int_0^{r_0} f_1(r) e^{-r e^{-i\theta_1}t} \dif r &= \int_0^{r_0} \tilde f_1(r) r^{-\frac n m} e^{-r e^{-i\theta_1}t} \dif r\\
&= \int_0^{r_0} A_1 r^{-\frac n m} e^{-r e^{-i\theta_1}t} \dif r + \cO\left( \int_0^{r_0} r^{-\frac n m + \frac 1 m} e^{-r e^{-i\theta_1}t} \dif r \right),
\end{align*}
and analogous to the computation of \eqref{e:bcestimate}, we have
\begin{align}\label{e:BCestimate-ca}
I_{BC} + I_{DE} = \frac{A_1}{2\pi i}\Gamma(1-n/m) e^{-i\theta_1 n/m} e^{\mu_1 t}t^{-1+n/m} + \cO\left(e^{\mu_1 t} t^{-1+\frac n m-\frac 1 m}\right).
\end{align}
\begin{remark}
We emphasise that this decomposition is not uniformly valid in $q\approx 0$. At $q=0$, we have $\Psi(s)|_{q=0} = \hat w_{10}(s-\mu_1)^{-1}$, hence $f_1(r)|_{q=0} = \hat w_{10} r^{-1}\left(e^{i(-\pi+\theta_1)}-e^{i(\pi+\theta_1)}\right)$. This vanishes for all $r\in(0,+\infty)$ which means $I_{BC}+I_{DE} = 0$, consistent with the exponential decay at $q=0$.
\end{remark}

In the remainder of this proof, we fix $q\neq0$.

\medskip
Along $\Gamma$: If $s\notin\Omega_\ca^\cc$, then $\lim_{R\to\infty,\,\epsilon\to0}\oint_{\Gamma}\Psi(s)e^{st}\dif s=0$. For $s\in\Omega_\ca^\cc$ the residue theorem gives,
\begin{align*}
I_\Gamma := \frac{1}{2\pi i}\lim_{\begin{subarray}{c}
	R\to\infty\\
	\epsilon\to0
	\end{subarray}}\oint_{\Gamma}\Psi(s)e^{st}\dif s 
= \sum_{\xi_j\in\Omega_\ca^\cc}\Res_{\xi_j}\left(\Psi(s)e^{st}\right) = \sum_{\xi_j\in\Omega_\ca^\cc}\left(\lim_{s\to\xi_j}(s-\xi_j)\Psi(s)\right)e^{\xi_j t}.
\end{align*}

\medskip
Along $FG$: $s-\mu_2=re^{i(\pi+\theta_2)}$, $\dif s=-e^{i\theta_2}\dif r$; along $HI$: $s-\mu_2=re^{i(-\pi+\theta_2)}$, $\dif s=-e^{i\theta_2}\dif r$. Analogous to the computation of $I_{BC}+I_{DE}$, with $f_2(r) := \Psi\left(\mu_2+re^{i(\pi+\theta_2)}\right) - \Psi\left(\mu_2+re^{i(-\pi+\theta_2)}\right)$, we obtain
\begin{align*}
I_{FG}+I_{HI} = \frac{1}{2\pi i}e^{\mu_2 t + i\theta_2} \int_0^\infty f_2(r) e^{-r e^{i\theta_2} t} \dif r. 
\end{align*} 
In contrast to $f_1(r)$, the function $f_2(r)$ is bounded for $r\in[0,\infty)$ and $\lim_{r\to0}f_2(r) = 0 = \lim_{r\to\infty}f_2(r).$ Analogous to the computation of \eqref{e:bcestimate} it follows that
\begin{align}
I_{FG}+I_{HI} &= \frac{B_1 e^{-i\theta_2(1-n/m)}}{2\pi i}\Gamma(2-n/m) t^{-2+n/m} e^{\mu_2 t} + \cO(t^{-2+n/m-1/m} e^{\mu_2 t}),\label{e:BCestimate-ca2}\\
B_1 &:= -\frac{d_2(d_3 q^2 \hat w_{10} - ((-\mu_\Delta)^{1-n/m} + d_1 q^2) \hat w_{20})}{(-\mu_\Delta)^{n/m}q(d_2 d_3 q^2-d_4((-\mu_\Delta)^{1-n/m}+d_1 q^2))^2} e^{i\theta_2(1-n/m)} 2i \sin(\pi(1-n/m)),\nonumber
\end{align}
whose rate if decay is larger than that of $I_{BC}+I_{DE}$.

\medskip
All in all, the first component of the solution can be written as 
\begin{align}\label{e:sol-ca}
\hat w_1(q,t) &= I_\Gamma - (I_{BC}+I_{DE}) - (I_{FG}+I_{HI})\nonumber\\
&= C_{\exp,s_1} e^{s_1 t} + C_{\bp,1} e^{\mu_1 t}t^{-\gamma} + \cO(e^{\mu_1 t}t^{-1+\frac n m - \frac 1 m}),\end{align}
where $s_1 := \argmax\{\Re(\xi_j):\xi_j\in\Omega_\ca^\cc,\,C_{\exp,s_1}\neq0\}$ for the largest exponential rate and $I_{FG}+I_{HI}$ has been absorbed into the last term as it is of order $e^{\mu_2 t}t^{-1-\gamma}$. Specifically, the coefficients $C_{\bp,1}$ is as given in the theorem statement, and
\begin{align*}
C_{\exp,s_1} &:= \lim_{s\to s_1} (s-s_1)\Psi(s).
\end{align*}

Since $\Re(\mu_1)<0$, for $\Re (s_1) \geq \Re(\mu_1)$ the first term in \eqref{e:sol-ca} is dominant while for $\Re(s_1) < \Re(\mu_1)$ it is the second term in \eqref{e:sol-ca}, and $\hat w_1$ decays as $t^{-\gamma}e^{\mu_1 t}$. Note that if $\{s:s\in\Omega_\ca^\cc, D_{\ca}(s,q^2)=0,q\in\RR\}=\emptyset$, then the decay fully depends on the second term. 

\begin{remark}\label{r:complexconjugatemultiplepoles}
	In case \textbf{cc}, if the poles besides $\mu_1,\,\mu_2$ are multiple, then the integral along $\Gamma$ is given by 
\begin{align*}
I_\Gamma &= \sum_{\xi_j\in\Omega_\ca^\cc}\Res_{\xi_j}\left(\Psi(s)e^{st}\right) = \sum_{\xi_j\in\Omega_\ca^\cc}\left(\frac{1}{(k_j-1)!}\lim_{s\to\xi_j}\left( \frac{\dif}{\dif s} \right)^{k_j-1}\left( (s-\xi_j)^{k_j} \Psi(s)e^{s t}\right)\right)\\
&= \sum_{\xi_j\in\Omega_\ca^\cc}\left(\frac{1}{(k_j-1)!}\left(\lim_{s\to\xi_j} \sum_{k=0}^{k_j-1} \frac{t^k (k_j-1)!}{k!(k_j-1-k)!} \left(\frac{\dif}{\dif s}\right)^{k_j-1-k}\left((s-\xi_j)^{k_j}\Psi(s)\right) \right) e^{\xi_j t} \right)\\
&= C_{\exp,s_0,\rho} t^{\rho-1} e^{s_0 t} + \cO(t^{\rho-2} e^{s_0 t}).
\end{align*}
The integrals along the branch cuts are given by \eqref{e:BCestimate-ca}, \eqref{e:BCestimate-ca2}, and those along other paths vanish. Hence, if $\Re(s_0)\geq\Re(\mu_1)$, then $\hat w_1$ behaves exponentially as $t^{\rho-1}e^{s_0 t}$, where $s_0 := \argmax\{\Re(\xi_j):\xi_j\in\Omega_\ca^\cc,\,C_{\exp,s_0,\rho}\neq0\}$ for the largest exponential rate and $\rho$ is the multiplicity of $s = s_0$. In particular,
\begin{align*}
C_{\exp,s_0,\rho} := \frac{1}{(\rho-1)!} \lim_{s\to s_0}(s-s_0)^\rho \Psi(s).
\end{align*} 
If $\Re(s_0)<\Re(\mu_1)$, then $\hat w_1$ decays as $t^{-\gamma}e^{\mu_1t}$.
\end{remark}

For the second component $\hat w_2(q,t)$, the integral along the branch cut $\BC_{\mu_2}^{\theta_2}$ has the form
\begin{gather*}
\frac{A_2}{2\pi i}\Gamma(1-n/m) e^{i\theta_2 n/m} e^{\mu_2 t}t^{-1+n/m} + \cO\left(e^{\mu_2 t} t^{-1+\frac n m-\frac 1 m}\right),\\
A_2 := -\frac{\left((-\mu_\Delta)^{1-n/m}+ d_1 q^2\right)\hat w_{20} - d_3q^2 \hat w_{10}}{d_4 q^2((-\mu_\Delta)^{1-n/m} + d_1 q^2) - d_2 d_3 q^4} e^{-i\theta_2 n/m} 2 i \sin(\pi n /m).
\end{gather*}
Analogous to the discussion in Remark~\ref{r:nonzero-ss}, if the initial data $(\hat w_{10},\hat w_{20})$ is not in the kernel of the matrix 
\[
\begin{pmatrix}
\mu_\Delta^\gamma + d_4 q^2 & -d_2 q^2\\
-d_3 q^2 & (-\mu_\Delta)^\gamma + d_1 q^2
\end{pmatrix}
\]
then the coefficients $A_1$ and $A_2$ satisfy $(A_1,A_2)\neq(0,0)$.  
Proceeding as above for $\hat w_1$ we find $C_{\bp,2}$ as claimed in Theorem \ref{t:ILT-ca} and $(C_{\bp,1},C_{\bp,2})\neq (0,0)$ generically.

\begin{remark}
In case \textbf{nr}, the calculation is completely analogous to the case \textbf{cc}, so we omit the details. Since here $\mu_1>\mu_2$, the second component of the solution $\hat w_2(q,t)$ has the form
\begin{align}
\hat w_2(q,t) = C_{\exp,s_0,\rho}t^{\rho-1} e^{s_0 t} + C_{\bp,2} e^{\mu_2 t}t^{-\gamma} + C_{\bp,3}e^{\mu_1 t}t^{-1-\gamma} + \cO(e^{\mu_1} t^{-1-\gamma-1/m}),
\end{align}
with coefficients given in the theorem statement. In particular, the coefficient $C_{\bp,3}$ is derived from the integral along $BC$ and $DE$.
In contrast to \eqref{e:sol-ca}, the term of order $e^{\mu_2 t}t^{-\gamma}$ is not leading order for $\Re(s_0)<\Re(\mu_1)$. Instead, the solution $\hat w_2$ behaves as $e^{\mu_1 t} t^{-1-\gamma}$ for $\Re(s_0)<\Re(\mu_1)$. 
\end{remark}

\end{document}